\documentclass[opre,nonblindrev]{informs3_hide}
\OneAndAHalfSpacedXI

\usepackage{amsmath,amsfonts,amssymb}
\usepackage{mathtools}
\usepackage{bm}
\usepackage[mathscr]{euscript}
\usepackage{enumitem}
\usepackage[dvipsnames]{xcolor}
\usepackage{graphicx}
\usepackage{float}
\usepackage{multirow}
\usepackage{booktabs}
\usepackage{array}
\usepackage[expansion=false]{microtype}
\usepackage{fix-cm}

\usepackage{xfrac}

\DeclareMathOperator{\pr}{\mathbb P}
\DeclareMathOperator{\E}{\mathbb E}
\DeclareMathOperator{\ind}{\mathbb I}

\newcommand{\dd}{\mathop{}\!\mathrm{d}}

\def\SFm{\mathsf{m}}
\def\SFb{\mathsf{b}}

\providecommand{\R}{\mathbb{R}}
\providecommand{\W}{\mathcal{W}}
\providecommand{\F}{\mathscr{F}}
\providecommand{\Borel}{\mathcal{B}}
\providecommand{\ip}[2]{\bigl\langle #1,#2\bigr\rangle}

\newcommand{\opnorm}[1]{{\lvert\kern-0.25ex
   \lvert\kern-0.25ex
   \lvert #1
   \rvert\kern-0.25ex
   \rvert\kern-0.25ex
   \rvert}}

\usepackage{natbib}
 \bibpunct[, ]{(}{)}{,}{a}{}{,}
 \def\bibfont{\small}

\TheoremsNumberedThrough

\theoremstyle{TH}

\ECRepeatTheorems

\EquationsNumberedThrough

\MANUSCRIPTNO{}

\begin{document}

\RUNAUTHOR{Li, Luo, and Zhang}
\RUNTITLE{Shrinking-Tube Concentration for Adaptive Markovian SA}

\TITLE{Shrinking-Tube Concentration for Adaptive Markovian Stochastic Approximation}

\ARTICLEAUTHORS{
\AUTHOR{Jin Li, Ye Luo}
\AFF{Faculty of Business and Economics, The University of Hong Kong, Pokfulam Road, Hong Kong SAR, \EMAIL{jli1@hku.hk}, \EMAIL{kurtluo@hku.hk}}
\AUTHOR{Xiaowei Zhang} 
\AFF{Department of Industrial Engineering and Decision Analytics, The Hong Kong University of Science and Technology, Clear Water Bay, Hong Kong SAR, \EMAIL{xiaoweiz@ust.hk}}
}

\ABSTRACT{Adaptive algorithms increasingly make decisions while reshaping the dynamics that generate their future data. 
We establish a shrinking-tube concentration bound for projected
stochastic approximation driven by an adaptive Markov chain.
The bound guarantees, with high probability, that every iterate
after a chosen time remains within a tolerance around the target that tightens
over time. The probability of any exit after the chosen time admits a
polynomially decaying upper bound, and a matching lower bound shows that its
polynomial exponent cannot be improved in general under finite second moments. The result therefore identifies a
sharp tradeoff between how quickly the tolerance shrinks and how rapidly the
probability of any future exit decreases. We also extend the analysis to
recursions with additional martingale-difference noise and predictable bias, showing how
growth in the martingale-difference noise scale slows the decay of the exit-probability bound while
predictable bias restricts the admissible tube shrinkage. The proof combines
backward kernel replacement, a finite-time
mean-squared-error bound, and a blockwise maximal first-exit argument.
We apply the theory to inventory learning with stockout-dependent demand
and fixed stockout costs, and quantify how numerical gradient accuracy
affects the all-future reliability of the resulting policies.
}


\maketitle

\section{Introduction}\label{sec:introduction}

Adaptive algorithms increasingly operate inside the systems they seek to
optimize. Stochastic approximation (SA) and its variants update decisions from
streaming observations and provide an algorithmic foundation for online
optimization and reinforcement learning \citep{KushnerYin03}. In a service
system, a pricing or capacity update changes demand and congestion, while the
next waiting-time observation still reflects customers admitted under earlier
choices \citep{LiLiangChenZhang26}. In inventory management, stockouts can affect future demand, so a change
in the base-stock level alters both current costs and the observations
used in later updates \citep{CheDongTong26}.  Each iterate is therefore both a deployed decision
and part of the mechanism that generates future data.

Continuous adaptation calls for more than an accurate decision at one chosen time. An iterate may be close to the target when it is evaluated and later move away. Requiring every subsequent iterate to remain within one fixed tolerance rules out such departures, but the trajectory may continue to wander indefinitely within the same coarse neighborhood. A continuously operated learning system should deliver both lasting accuracy and progressive refinement. After a finite time, every subsequent decision should remain close to the target under a tolerance that tightens as information accumulates.

We formalize this requirement through the event
\begin{equation*}
\left\{
\lVert\theta_k-\theta^\star\rVert
\leq c k^{-\beta}
\text{ for all }k\geq k_0
\right\}.
\end{equation*}
On this event, \(\theta_k\) is confined to 
a ball centered at \(\theta^\star\) with radius \(c k^{-\beta}\) at iteration \(k\); viewed across iterations,
these balls form a tube around the target. For \(\beta>0\), its radius
decreases with \(k\), giving a \emph{shrinking tube}; for \(\beta=0\), the
radius remains fixed. This event requires the entire trajectory
from \(k_0\) onward to remain inside this tube.
Under adaptive feedback, a late excursion can degrade the current decision, change the transition distribution that generates later observations, and create a new transient in the learning process.
Our shrinking-tube concentration bound controls the probability that any
iterate from \(k_0\) onward leaves this tube.
Figure~\ref{fig:tube-vs-cylinder} illustrates the
difference between fixed containment and a guarantee whose certified tolerance
continues to tighten.

\begin{figure}[ht]
    \FIGURE{
    \includegraphics[width=0.8\textwidth]{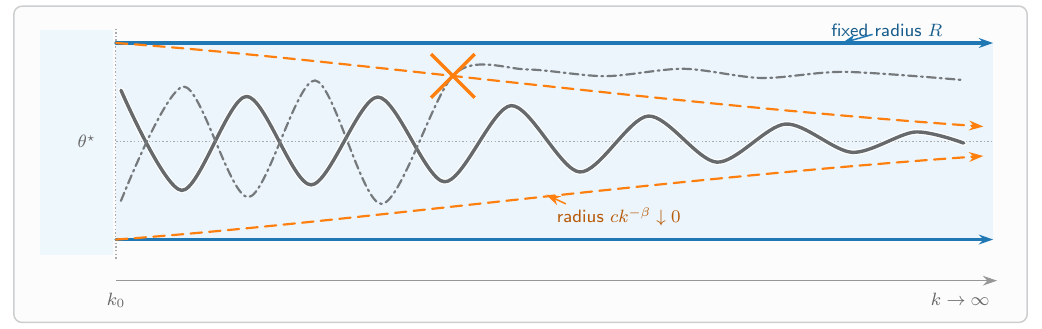}
    }{Shrinking Tube versus Fixed-Radius Cylinder.
    \label{fig:tube-vs-cylinder}}
    {}
\end{figure}

Three effects determine how quickly the bound on the probability of a future
tube exit decreases as \(k_0\) increases: diminishing step sizes attenuate new
stochastic fluctuations, the shrinking tube tightens the admissible error,
and an infinite horizon creates infinitely many opportunities for exit.
We quantify their combined effect for projected SA driven by an adaptive
Markov chain.

The first analytical challenge is a \emph{stationarity gap}. At time \(k\), the
iterate \(\theta_k\) selects the transition kernel \(P_{\theta_k}\), whereas the
current state still reflects kernels selected by earlier iterates. If the
parameter were held fixed, the stationary distribution of the corresponding
homogeneous chain would define the mean update direction. Along the adaptive
path, the observed update is generally not conditionally centered at the
mean field associated with \(\theta_k\). A change in one iterate can
affect subsequent transitions, and the Markov state may be unbounded.
The analysis must control this transient bias under moment conditions
while preserving the stability provided by the mean field.

The second challenge is a \emph{time-uniformity gap}. A fixed-time
mean-squared-error (MSE) bound can certify accuracy at a checkpoint but does not
directly control the probability that any later iterate leaves the prescribed
shrinking tube. Applying a pointwise probability bound at
every future iteration and summing the results generally produces a
nonsummable series in the parameter range of interest. Nearby exits also
depend on strongly overlapping portions of the trajectory. Treating them as
separate events discards the temporal structure that makes all-future control
possible.

We resolve these challenges through a sharp trajectory concentration theorem
and a proof framework that controls adaptive Markov bias and all future
stochastic excursions at their natural timescales.  The main contributions
are summarized next.

\subsection{Contributions}

First, we establish shrinking-tube concentration for projected SA driven by an adaptive Markov chain on a potentially unbounded state space and prove that its polynomial decay exponent is optimal.
For step size of order \(k^{-\delta}\) and tube radius of order
\(k^{-\beta}\), 
Theorem~\ref{thm:trajectory} bounds 
the probability that an SA trajectory exits the tube at any iteration
\(k\geq k_0\).
The upper bound decays with \(k_0\) with polynomial exponent \(2(\delta-\beta)-1\), up to a squared logarithmic factor, when 
\(1/2<\delta\leq1\) and \(0\leq\beta<\delta-1/2\). 
Theorem~\ref{thm:sharpness} constructs examples under the same assumptions for which the exit-probability lower bound has an exponent arbitrarily close to that of the upper bound. 
Consequently, no larger polynomial exponent can hold in general under these assumptions. 
The upper and lower bounds identify the worst-case 
tradeoff between how quickly the tube shrinks and how quickly its exit-probability bound decreases.

Second, we develop a proof framework for SA with 
adaptive Markovian data.  Backward kernel replacement converts the bias from
changing kernels into geometrically discounted local perturbations. Combined 
with reference-chain geometric ergodicity, the resulting adaptive-to-reference
comparison bound yields a Lyapunov drift condition used to derive finite-time
MSE bounds.  A blockwise maximal first-exit argument then controls all
crossings within each time interval and upgrades fixed-time checkpoint
accuracy to one all-future event.  

Third, we extend the SA recursion driven by an adaptive Markov chain to accommodate martingale-difference noise and predictable directional bias arising from simulation and numerical approximation. We establish finite-time MSE and
shrinking-tube bounds.  If the martingale-difference noise scale
grows as \((k+1)^{\omega_{\mathsf M}}\) and the predictable bias decays as \((k+1)^{-\omega_{\mathsf B}}\), then the shrinking-tube concentration bound has polynomial exponent
\(2(\delta-\beta-\omega_{\mathsf M})-1\), provided \(0\leq\omega_{\mathsf M}<\delta-\beta-1/2\),
and \(\omega_{\mathsf B}>\beta.\) 
We prove that this exponent cannot be improved in general.
We apply these results to an inventory management problem with
stockout-dependent demand and fixed stockout costs. Standard pathwise differentiation misses the effect of stock-level changes on the expected fixed stockout cost. A finite-difference estimator captures this effect but introduces a tradeoff between bias and sampling variability. Our bounds show how to choose the finite-difference radius schedule to obtain shrinking-tube guarantees for the learned stock levels.

\subsection{Relation to Prior Work}\label{sec:literature}

Recent finite-sample SA theory has clarified how Markov dependence, stability
geometry, and the tail behavior of stochastic updates shape convergence at a
prescribed iteration.
\citet{ChenZhangDoanClarkeMaguluri22} derive finite-sample rates for nonlinear
SA driven by a fixed Markov chain through Lyapunov drift, while
\citet{ChenMaguluriShakkottaiShanmugam24} develop an arbitrary-norm Lyapunov
framework for contractive mean operators.  \citet{LawWaltonYang26} identify a
different concentration regime: nonvanishing inward drift, together with
sub-exponential increments, yields exponential tails for the iterate error and
supports geometric high-probability convergence under staged step sizes.
These results make the roles of dependence and stability conditions explicit,
with the error at a selected iteration as their central probabilistic object.

A closely related line studies concentration over the future trajectory.
\citet{ChandakBorkarDodhia22} derive bounds from time \(n_0\) onward for
contractive SA with martingale-difference noise and an 
finite-state Markov chain with adaptive (i.e., iterate-dependent) kernels.  For each fixed \(n_0\), their bound on the
iterate error approaches a positive limit as the iteration index increases.
\citet{ChenMaguluriZubeldia25} establish maximal concentration for
contractive SA with potentially unbounded iterates under conditionally
centered additive or multiplicative noise.  Their results bound the trajectory's error by a decreasing function of time
and yield sub-Gaussian or Weibull tails under the
corresponding light-tail conditions.  Recent studies extend shrinking
maximal bounds to fixed Markovian sampling and heavier-tailed stochastic
noise
\citep{QianXieLiuZhang24,AgrawalMaguluriZubeldia26}.  This  literature develops
the two central ingredients of the present problem along separate directions:
an adaptive Markov kernel and a boundary that tightens over an
unbounded future.

This paper establishes their joint form under a finite-second-moment regime.
The iterate selects the kernel that generates the next observation, so the
observed update is not centered relative to the mean
field associated with the current parameter.  We control this feedback on a
potentially unbounded Markov state space under moment conditions.  Theorem~\ref{thm:trajectory} gives one event on which
\(\theta_k\) remains inside the prescribed tube \(c k^{-\beta}\) for all \(k\geq k_0\).  For
step-size exponent \(\delta\), the exit-probability bound has polynomial exponent
\(2(\delta-\beta)-1\), up to a squared logarithmic factor.
Theorem~\ref{thm:sharpness} shows that no larger polynomial exponent holds
in general under these assumptions.  The extension in
Section~\ref{sec:structured-extension} further
shows how growing martingale-difference noise shifts this exponent and how
predictable bias restricts the admissible tube rates.

The rest of the paper is organized as follows.
Section~\ref{sec:problem-setup} formulates the adaptive Markovian SA model
and states the assumptions. Section~\ref{sec:canonical-results} states
and interprets the shrinking-tube theorem and outlines its proof.
Section~\ref{sec:backward-telescoping} develops backward kernel
replacement and the adaptive-to-reference comparison bound.
Section~\ref{sec:MSE} proves the fixed-time MSE bound, and
Section~\ref{sec:concen-bounds} converts that bound into shrinking-tube
concentration. Section~\ref{sec:sharpness} establishes sharpness of the
polynomial exponent. Section~\ref{sec:structured-extension} extends the
results to incorporate additional martingale-difference noise and predictable bias.
Section~\ref{sec:inventory-application} develops the inventory
application, and Section~\ref{sec:conclusions} concludes. The e-companion
contains the complete proofs.

\section{Problem Formulation}
\label{sec:problem-setup}

SA learns a parameter by repeatedly turning observed
data into an update.  Its goal is to make an average update direction vanish;
the precise average for our setting is defined in Section~\ref{subsec:reference-chain-target}.  At
iteration \(k\), the projected update is
\begin{align}
    \theta_{k+1}
    =\Pi_\Theta\Bigl(
        \theta_k+\alpha_kF(\theta_k,W_k)
      \Bigr),
      \label{eq:update-scheme-projection}
\end{align}
for \(k\geq0\),
where \(\theta_k\in\Theta\subseteq\R^d\) is the current iterate,
\(\alpha_k\) is the step size, and \(\Pi_\Theta\) is the Euclidean projection
onto the feasible set \(\Theta\), defined by \(\Pi_\Theta(x):=\underset{\theta\in\Theta}{\arg\min}\,
    \lVert x-\theta\rVert\) for \(x\in\R^d\).

\subsection{Adaptive Markov Chain and Kernels}
\label{subsec:adaptive-noise}

The state \(W_k\in\W\) carries the data used in the update, where \(\W\) is
a Borel subset of a Euclidean space.  The map
\(F:\Theta\times\W\to\R^d\) is jointly Borel measurable.

As a simple benchmark, suppose each new state observation is drawn
independently from a fixed distribution. The observations
affect the iterates through \(F\), while subsequent observations continue to be
generated from the fixed distribution. Thus, the influence runs in one direction.

Here, the influence runs in both directions. The current state \(W_k\) enters
\(F(\theta_k,W_k)\) and shapes the next iterate \(\theta_{k+1}\). Meanwhile, \(\theta_k\), which may represent a decision or policy, selects
the transition rule for \(W_{k+1}\). Thus, the algorithm learns from the state
while also changing how the state evolves. 

To describe the second direction, associate each \(\theta\in\Theta\) with a
Markov transition kernel \(P_\theta\) on \(\W\).  For each
\(E\in\Borel(\W)\), the map
\((\theta,w)\mapsto P_\theta(w,E)\) is Borel measurable.  If the parameter
were held fixed at \(\theta\), as in the classical framework, then \(P_\theta(w,E)\)
would be the one-step probability of moving from state \(w\) into \(E\).
In the actual recursion, the current iterate \(\theta_k\) selects the
iterate-dependent Markov kernel \(P_{\theta_k}\) used for the next transition.

All variables live on \((\Omega,\mathscr A,\pr)\), and
\((\theta_0,W_0)\) is a \(\Theta\times\W\)-valued random element.
To specify the information available when the next state is generated, define
the filtration
\begin{equation*}
    \F_k:=\sigma(\theta_0,W_0,\theta_1,W_1,\ldots,\theta_k,W_k),
\end{equation*}
for \(k\geq0\). 
Unless stated otherwise, equalities and inequalities between random
quantities, including conditional expectations and conditional probabilities,
are understood \(\pr\)-almost surely.

We call \(\{W_k:k\geq 0\}\) the \emph{adaptive Markov chain} associated with the SA
recursion in \eqref{eq:update-scheme-projection} when
\begin{align*}
    \pr(W_{k+1}\in E\mid\F_k)
    &=P_{\theta_k}(W_k,E)
\end{align*}
for all \(E\in\Borel(\W)\) and \(k\geq0\). 
The transition identity is conditional on the joint history \(\F_k\);
the state sequence \(\{W_k:k\geq0\}\) need not be Markov with respect
to its own natural filtration.

Policy-dependent sampling also arises in reinforcement learning, where
updating the policy changes the distribution of subsequent state--action
observations even when the environment's transition rules remain fixed
\citep{KondaTsitsiklis03}. In inventory learning, the decision can affect
future demand through stockouts \citep{CheDongTong26}. These feedback
mechanisms motivate the iterate-dependent transition model.
Section~\ref{sec:inventory-application} develops a full inventory
application, including the conditions needed to apply the results.

\subsection{Mean Field and Target}
\label{subsec:reference-chain-target}

To define the target of the SA recursion, first fix the parameter at a deterministic value
\(\theta\).  The state then evolves under the same transition kernel
\(P_\theta\) at each step, forming a homogeneous Markov chain
that we call the \emph{reference chain}. 
Assuming  
this chain has a unique stationary distribution \(\nu_\theta\),
averaging the update direction under \(\nu_\theta\) defines the
mean field:
\begin{equation}\label{eq:mean-field}
    \bar F(\theta):=\int_\W F(\theta,w)\,\nu_\theta(dw).
\end{equation}
We call \(\theta^\star\) a target if it is a root of this field, i.e.,
\(\bar F(\theta^\star)=0\).  Existence and location of the designated target
are imposed in the basic assumptions below.
Thus, \(\bar F(\theta)\) describes the long-run update direction when the
parameter is held fixed at \(\theta\). This fixed-parameter average defines
the target, while the adaptive algorithm generally has a different one-step
conditional mean direction.

For \(k\geq1\), \(F(\theta_k,W_k)\) uses a transient state:
\(\bar F(\theta_k)\) averages under \(\nu_{\theta_k}\), whereas \(W_k\) was
generated using \(P_{\theta_{k-1}}\) and still carries the effects of earlier
transitions.  Consequently, the conditional center of the observed direction
generally differs from \(\bar F(\theta_k)\).  
Our analysis therefore compares the time-inhomogeneous transition product of
the adaptive Markov chain with a homogeneous reference chain associated with a
past iterate. A central challenge is to control how much the iterate changes
after the past value defining the reference chain. These changes affect the
transition kernels and thereby create a discrepancy between the adaptive and
reference dynamics.

\subsection{Basic Assumptions}

\begin{assumption}
    \label{assump:learning-rate}
There are positive constants \(\alpha_0\) and \(\delta\in(0,1]\) such that \(\alpha_k=\alpha_0 k^{-\delta}\) for all \(k\geq 1\).
\end{assumption}

Polynomially decaying step sizes are standard 
\citep{KushnerYin03}. Asymptotic analyses often impose
\(1/2<\delta\leq1\) to establish convergence of \(\theta_k\) to
\(\theta^\star\). In this range, \(\sum_k\alpha_k=\infty\) allows the
iterates to keep adjusting, while \(\sum_k\alpha_k^2<\infty\) limits the
accumulated effect of stochastic fluctuations. The shrinking-tube concentration result
in Theorem~\ref{thm:trajectory} also requires this range. By contrast, the
fixed-time MSE bound in Proposition~\ref{prop:expectation} applies throughout
the broader range \(0<\delta\leq1\).

\begin{assumption}
    \label{assump:compact}
The set \(\Theta\subseteq\R^d\) is nonempty, compact, and convex.  
\end{assumption}

Under Assumption~\ref{assump:compact}, $\Theta$ has a finite diameter 
\(
    d_\Theta:=\sup_{\theta,\theta'\in\Theta}
    \lVert\theta-\theta'\rVert<\infty
\), and the projection \(\Pi_\Theta\) is \emph{nonexpansive}: 
    \(\lVert\Pi_\Theta(x)-\Pi_\Theta(y)\rVert
    \leq \lVert x-y\rVert\) 
    for all \(x,y\in\R^d\).

To ensure that the mean field is well defined, we impose Lyapunov-type stability
conditions uniformly over the family of reference chains indexed by
\(\theta\in\Theta\).
For \(p\in\{1,2\}\), define the polynomial Lyapunov functions
\begin{equation*}
    U_p(w):=1+\lVert w\rVert^p.
\end{equation*}

\begin{assumption}
\label{assump:ergodic}
\begin{enumerate}[
label=\textnormal{(\roman*)},
ref=\textnormal{(\roman*)}
]
\item\label{part:ergodic-reference}
For each \(\theta\in\Theta\), the kernel \(P_\theta\) is
irreducible and aperiodic.  There are a positive integer \(N\),
a nonempty Borel set \(\W_0\subseteq\W\), a positive constant \(\epsilon\), and a
probability measure \(\varphi\) on \(\W\) such that \(P_\theta^N(w,E)\geq\epsilon\varphi(E)\) for all \(\theta\in\Theta\), \(w\in\W_0\), 
    and \(E\in\Borel(\W)\).
For a measurable set \(A\), define \(\ind_A\) as its indicator function.
Moreover, there are nonnegative constants \(\lambda<1\) and
\(b\) such that for all \(\theta\in\Theta\) and \(w\in\W\), 
\begin{equation}\label{eq:drift-condition}
    (P_\theta U_2)(w)
    \leq\lambda U_2(w)+b\ind_{\W_0}(w).
\end{equation}
\item\label{part:ergodic-initial} 
The initialization satisfies 
    \(\E[U_2(W_0)]<\infty\), where the expectation is taken under the initial distribution of
$(\theta_0,W_0)$.

\end{enumerate}
\end{assumption}

Assumption~\ref{assump:ergodic}\ref{part:ergodic-reference} ensures that
each \(P_\theta\) has a unique stationary distribution \(\nu_\theta\) and that
\(\sup_{\theta\in\Theta}\nu_\theta(U_2)<\infty\). 
Furthermore, \(P_\theta\) is \(U_1\)-geometrically ergodic
\citep{MeynTweedie09}: for each \(\theta\in\Theta\), there are positive constants \(C\) and \(\rho\in(0,1)\) such that
\begin{equation}\label{eq:uniform-geometric-ergodicity}
    \bigl|(P_\theta^k f)(w)-\nu_\theta(f)\bigr|
    \leq C\rho^kU_1(w),
\end{equation} 
for all \(w\in\W\), \(k\geq0\), and measurable \(f\) with
\(|f|\leq U_1\).  In general, the constants \(C\) and \(\rho\) in this
fixed-\(\theta\) statement may depend on \(\theta\); we suppress that
dependence in the notation.  Our additional assumptions remove this
dependence: the \(U_1\)-weighted kernel continuity in
Assumption~\ref{assump:kernel-Lipschitz}, together with compactness of
\(\Theta\), allows the two constants to be chosen independently of
\(\theta\).  Thus the same \(C\) and \(\rho\) in
\eqref{eq:uniform-geometric-ergodicity} work simultaneously for all
\(\theta\in\Theta\).  See Lemma~\ref{lemma:exp-ergodic} for details.  
Conditions similar to Assumption~\ref{assump:ergodic}\ref{part:ergodic-reference}
have been used to establish convergence of actor--critic algorithms
\citep{KondaTsitsiklis03} and to analyze stochastic gradient methods with
adaptive data \citep{CheDongTong26}.

Assumption~\ref{assump:ergodic}\ref{part:ergodic-initial} requires the adaptive Markov chain to have a finite
second moment at initialization.  Together with the uniform drift condition
\eqref{eq:drift-condition}, it ensures that \(W_k\) has a finite second moment conditional on the observed history.

\begin{assumption}
    [Kernel Continuity]
\label{assump:kernel-Lipschitz}
There is a positive constant \(L_P\) such that
\begin{equation*}
    \bigl|(P_\theta f)(w)-(P_{\theta'}f)(w)\bigr|
    \leq L_PU_1(w)\lVert\theta-\theta'\rVert,
\end{equation*} 
for all \(\theta,\theta'\in\Theta\), \(w\in\W\), and any Borel measurable
\(f:\W\to\R\) satisfying \(|f(y)|\leq U_1(y) \) for all \(y\in\W\),
where \((P_\theta f)(w):=\int_\W f(y)P_\theta(w,dy)\). 
\end{assumption}

\begin{assumption}
    \label{assump:growth}
There are positive constants \(L_F\) and \(C_F\) such that for all \(\theta,\theta'\in\Theta\) and \(w\in\W\),
\begin{align}
    \lVert F(\theta,w)-F(\theta',w)\rVert
    &\leq L_F(1+\lVert w\rVert)\lVert\theta-\theta'\rVert,
    \label{eq:Lipschitz-F}\\*
    \sup_{\theta\in\Theta}\lVert F(\theta,w)\rVert
    &\leq C_F(1+\lVert w\rVert).
    \label{eq:LinearGrowth-F}
\end{align}
\end{assumption}

Condition \eqref{eq:Lipschitz-F} controls how much the update direction changes when the current iterate is replaced by a past iterate. Under condition \eqref{eq:LinearGrowth-F}, the update direction in the SA recursion \eqref{eq:update-scheme-projection} may grow linearly with the norm of the unbounded Markov state.

\begin{assumption}
\label{assump:interior-target}
There is \(\theta^\star\in\operatorname{int}(\Theta)\) such that
\(\bar F(\theta^\star)=0\).
\end{assumption}

Because \(\theta^\star\) lies in the interior, it has a buffer contained in \(\Theta\).
For all sufficiently large \(k_0\), with high probability \(\theta_k\) remains in this buffer for all \(k\geq k_0\); the projection is then inactive, so excursions can be controlled by adding up the successive updates. This additive representation turns local update bounds into control of the entire future trajectory.

\begin{assumption}[Quasi-Strong Monotonicity (QSM)]
\label{assump:qsm}
There is a positive constant \(\zeta\) such that for all \(\theta\in\Theta\),
\begin{equation*}
    \ip{\theta-\theta^\star}{\bar F(\theta)}
    \leq-\zeta\lVert\theta-\theta^\star\rVert^2.
\end{equation*}
\end{assumption}

Assumption~\ref{assump:qsm} requires the mean field \(\bar F(\theta)\) to point sufficiently strongly toward \(\theta^\star\), thereby providing negative drift in the squared parameter error.\footnote{The
term \emph{quasi-strong monotonicity} comes from monotone-operator 
\citep{PengXuYanYin16}.  Strong monotonicity
imposes the corresponding inequality on every pair \(\theta,\theta'\), whereas
QSM anchors it at \(\theta^\star\) and is therefore weaker.}
Related conditions appear in analyses of stochastic gradient methods
and reinforcement learning algorithms
\citep{Spall03,ChenZhangDoanClarkeMaguluri22,CheDongTong26,LiLiangChenZhang26}.\footnote{An alternative approach to SA analysis assumes contraction of the operator
\(\mathcal T:=I+\bar F\), where \(I\) denotes the identity mapping
\citep{Borkar21,ChandakBorkarDodhia22,
ChenMaguluriShakkottaiShanmugam24,ChenMaguluriZubeldia25}.
If \(\mathcal T\) is a Euclidean contraction with modulus \(\kappa<1\) and
fixed point \(\theta^\star\), then QSM holds with
\(\zeta=1-\kappa\); the converse need not hold. This contraction provides
the deterministic decrease underlying a Lyapunov drift condition for the
stochastic recursion and hence supports finite-time error bounds.
Contraction in another norm, including the sup norm commonly used for
Bellman operators, is generally incomparable with Euclidean QSM.}
This condition also makes \(\theta^\star\) the unique root: if \(\bar F(\theta)=0\), then \(0\leq-\zeta\lVert\theta-\theta^\star\rVert^2\), and hence \(\theta=\theta^\star\).

\section{Shrinking-Tube Concentration: An Overview}
\label{sec:canonical-results}

We now turn to the main question of the paper: when does an adaptive SA
trajectory permanently enter an improving accuracy regime?  The challenge is
that the same iterates whose accuracy we seek to control also change the
Markov dynamics that generate future data.  Earlier kernel choices therefore
continue to influence current updates, while new stochastic fluctuations can
still produce later departures from the target.

\subsection{Main Result and Implications}
\label{subsec:canonical-main-result}
\begin{theorem}
\label{thm:trajectory}
Suppose Assumptions~\ref{assump:learning-rate}--\ref{assump:qsm} hold, with
\(1/2<\delta\leq1\).
When \(\delta=1\), also suppose that \(2\zeta\alpha_0>1\).
Fix \(0\leq\beta<\delta-1/2\) and \(c>0\). Then there are positive constants \(C\) and \(K\) such that, for all
\(k_0\geq K\),
\begin{equation*}
\pr\Bigl(
\lVert\theta_k-\theta^\star\rVert\leq c k^{-\beta}
\text{ for all }k\geq k_0
\,\Bigm|\,\F_0
\Bigr)
\geq
1-C U_2(W_0)(1+\log k_0)^2
k_0^{-(2(\delta-\beta)-1)}.
\end{equation*}
\end{theorem}

Theorem~\ref{thm:trajectory} controls the time at which the trajectory
enters and thereafter remains in the prescribed tube.  To make this interpretation
precise, define the \emph{permanent-entry index}
\begin{equation*}
\tau_{c,\beta}
:=
\inf\left\{
    n\geq1:
    \lVert\theta_k-\theta^\star\rVert
    \leq c k^{-\beta}
    \text{ for all }k\geq n
\right\}.
\end{equation*}
The event in
Theorem~\ref{thm:trajectory} is exactly \(\{\tau_{c,\beta}\leq k_0\}\), which
we call the \emph{tube-containment event}. 
Equivalently, the theorem gives the conditional tail bound
\begin{equation}
\pr\bigl(\tau_{c,\beta}>k_0\mid\F_0\bigr)
\leq
C U_2(W_0)(1+\log k_0)^2
k_0^{-(2(\delta-\beta)-1)}.
\label{eq:permanent-entry-tail}
\end{equation}
It therefore quantifies how long one must wait before the error remains
bounded by \(c k^{-\beta}\) at every subsequent iteration.

This interpretation separates the theorem from finite-time
guarantees.  A fixed-time bound gives an accuracy rate without protection
against later relapse.  Fixed-radius all-future containment gives permanence
at one accuracy level without continued refinement.  Theorem~\ref{thm:trajectory}
combines rate, permanence, and finite-time confidence on a single event \(\{\tau_{c,\beta}\leq k_0\}\).  When
\(\beta=0\), the event gives eventual containment in a fixed neighborhood.  When
\(\beta>0\), it gives containment within a radius that decreases with \(k\).
For \(\beta>0\) and any tolerance \(\varepsilon>0\), the tube-containment event
implies
\begin{equation*}
\lVert\theta_k-\theta^\star\rVert\leq\varepsilon
\qquad
\text{for all }
k\geq
\max\left\{
    k_0,
    \left\lceil(c/\varepsilon)^{1/\beta}\right\rceil
\right\}.
\end{equation*}

The concentration bound also yields a pathwise asymptotic rate.  For each
fixed \(0\leq\beta<\delta-1/2\), the future-exit events decrease as \(k_0\)
increases, and their conditional probabilities converge to zero by
\eqref{eq:permanent-entry-tail}.  Applying this conclusion to \(c=1/m\) for 
\(m=1,2,\ldots\), and taking an intersection gives
\begin{equation*}
\pr\left(
\lim_{k\to\infty}k^\beta\lVert\theta_k-\theta^\star\rVert=0
\,\middle|\,\F_0
\right)=1.
\end{equation*}
For each fixed \(c>0\), the tail bound in
\eqref{eq:permanent-entry-tail} quantifies the permanent-entry index \(\tau_{c,\beta}\) after which
\(\lVert\theta_k-\theta^\star\rVert\leq c k^{-\beta}\) holds permanently.

The exponent in \eqref{eq:permanent-entry-tail} has a second-moment
interpretation.  The recursion multiplies the centered stochastic
contribution to the update direction by the step size \(\alpha_k\).
Squaring this factor gives the polynomial scale \(\alpha_k^2\)
in the corresponding second-moment bound.  To compare this
contribution with the allowed error \(c k^{-\beta}\), normalize
by the squared tube radius.  Since \(\alpha_k\asymp k^{-\delta}\),
the resulting scale is
\(
\alpha_k^2/(c k^{-\beta})^2
\asymp k^{-2(\delta-\beta)}.
\)
Thus, the normalized second-moment scale decays more slowly
when the tube shrinks faster.

Summing these normalized scales over the future gives
\begin{equation}
\sum_{k\geq k_0}
\left(\frac{\alpha_k}{c k^{-\beta}}\right)^2
\asymp
\sum_{k\geq k_0} k^{-2(\delta-\beta)}
\asymp
k_0^{-(2(\delta-\beta)-1)}.
\label{eq:future-normalized-second-moment-budget}
\end{equation}
This series is finite exactly when \(\beta<\delta-1/2\), and
its tail has the polynomial order appearing in
Theorem~\ref{thm:trajectory}, providing a second-moment
interpretation of both the admissible tube-shrinkage rates
and the exit-probability exponent.

The result identifies a tradeoff among tube shrinkage, the
exit-probability bound, and the starting iteration of the
tube-containment guarantee. 
On the one hand, for a fixed starting iteration \(k_0\),
a larger \(\beta\) requires faster decay of the parameter error
but reduces the bound's polynomial decay exponent
\(2(\delta-\beta)-1\).
On the other hand, for a fixed upper bound \(q\in(0,1)\)
on the exit probability, faster tube shrinkage requires
a later starting iteration. At the level of polynomial orders, the required
starting iteration scales as
\(
k_0 \asymp q^{-1/(2(\delta-\beta)-1)},
\)
with constants and logarithmic factors suppressed.
Figure~\ref{fig:tube-failure-tradeoff} illustrates this tradeoff.

\begin{figure}[t]
    \FIGURE{
    \includegraphics[width=0.9\textwidth]{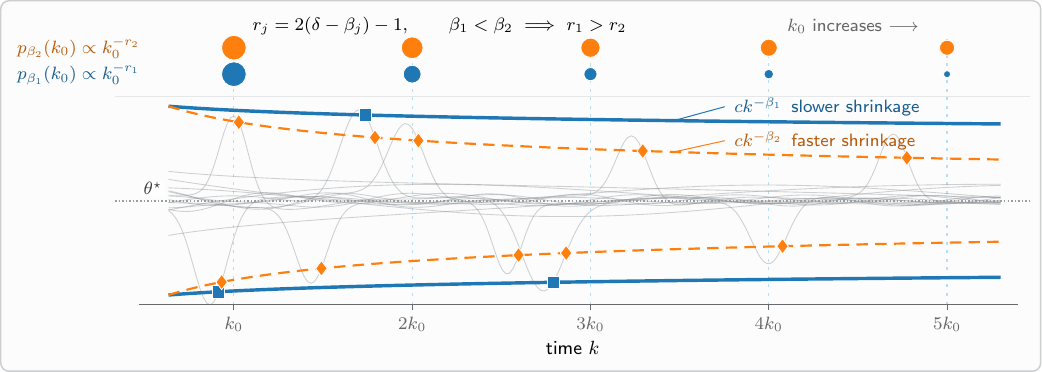}
    }{Nested Shrinking Tubes and Exit-Probability Bounds.
    \label{fig:tube-failure-tradeoff}}
    {The same trajectories are shown against two tubes with
    \(\beta_1<\beta_2\). Squares and diamonds mark final re-entry.
    The slower-shrinking tube has a faster-decaying exit-probability bound
    as \(k_0\) increases, since \(r_1>r_2\) for
    \(r_j=2(\delta-\beta_j)-1\).}
\end{figure}

The polynomial exponent \(2(\delta-\beta)-1\) cannot be improved
in general without additional assumptions.
Section~\ref{sec:sharpness} illustrates the construction using scalar
data-generating processes with independent signed-Pareto observations,
each satisfying the assumptions
of Theorem~\ref{thm:trajectory}.  Their exit-probability lower
bounds have polynomial exponents arbitrarily close to
\(2(\delta-\beta)-1\), ruling out any larger exponent
as a general guarantee under these assumptions.

The time-uniform nature of the bound also allows data-dependent
evaluation: on the containment event,
\(
\lVert\theta_T-\theta^\star\rVert
\leq c T^{-\beta}
\)
holds simultaneously for all \(T\geq k_0\).
Thus, the same probability guarantee applies whether \(T\) is
fixed in advance or selected from data observed while the
algorithm runs.

\subsection{Proof Strategy}
\label{subsec:canonical-proof-architecture}

The proof addresses two gaps.
The first is a \emph{stationarity gap}: QSM quantifies how strongly the
mean field points toward the target, but the algorithm updates using
observations from an adaptive Markov chain. We must relate these
observed updates to the mean field before using QSM to control the MSE.  The second is
a \emph{time-uniformity gap} between a fixed-time estimate and the
tube-containment event.

To close the stationarity gap, set \(t_k:=k-\ell_k\) and compare the
adaptive chain with a homogeneous reference chain starting from
\(W_{t_k}\) and governed by \(P_{\theta_{t_k}}\).
Through \emph{backward kernel replacement}, we replace the adaptive
kernels from the last transition backward, so all transitions following
each replacement use the reference kernel. Geometric ergodicity then reduces the
effect of of each discrepancy over the remaining reference transitions, while kernel continuity relates its magnitude to the parameter change
since \(t_k\). 
Geometric ergodicity also reduces the reference chain's
dependence on its starting state by a factor \(\rho^{\ell_k}\), bringing
its expected update close to the mean field \(\bar F(\theta_{t_k})\).

The look-back window must therefore be long enough for the reference
chain to forget its starting state, but short enough to keep expected
parameter changes small. We choose \(\ell_k\asymp\log k\), with a
sufficiently large proportionality constant, so that
\[
\rho^{\ell_k}\lesssim\alpha_k
\quad\mbox{and}\quad
\sum_{i=t_k}^{k-1}\alpha_i
\lesssim\alpha_k(1+\log k).
\]
The first bound controls the reference chain's initial-state effect
through geometric ergodicity. The second shows that the cumulative
step size over the look-back window tends to zero. Since each parameter
update has expected magnitude bounded by a multiple of its step size,
the expected parameter change across the window also tends to zero. The Lipschitz condition on \(F\) then controls the error from replacing
\(\theta_{t_k}\) by \(\theta_k\). 
After accounting for these errors,
QSM gives an inequality relating the MSE at iteration \(k+1\) to
the MSE at iteration \(k\). 
Iterating this inequality gives the fixed-time bound
\[
\E[\lVert\theta_k-\theta^\star\rVert^2\mid\F_0]
\lesssim
U_2(W_0)(1+\log k)k^{-\delta},
\]
stated in Proposition~\ref{prop:expectation} and proved in
Section~\ref{sec:MSE}.

We now address the \emph{time-uniformity gap}: converting this
fixed-time estimate into a guarantee that all iterates from \(k_0\)
onward remain inside the prescribed shrinking tube. A direct approach
bounds the exit probability at each iteration and sums over future
iterations. Markov's inequality and the MSE estimate yield
\[
\pr\Bigl(
  \lVert\theta_k-\theta^\star\rVert>
  c k^{-\beta}
  \,\Bigm|\,\F_0
\Bigr)
\lesssim
U_2(W_0)c^{-2}
(1+\log k)k^{-\delta+2\beta}.
\]
These pointwise bounds are not summable. 
Indeed, \(\delta-2\beta\leq1\), so
\(\sum_{k\geq k_0}(1+\log k)k^{-\delta+2\beta}=\infty\).
Summing these bounds therefore gives no useful lower bound on the
probability of the tube-containment event. Essentially, this is 
because it counts nearby crossings separately,
although they depend on strongly overlapping portions of the trajectory.
This motivates grouping nearby iterations into blocks.

The block boundaries are most easily understood by comparing two timescales.
The natural timescale is the iteration index \(k\).  The algorithmic
timescale measures accumulated step size:
\(
A_n:=\sum_{k=1}^{n-1}\alpha_k.
\)
When \(1/2<\delta<1\), we have \(A_n\asymp n^{1-\delta}\).
We choose each block to cover a fixed amount of algorithmic time. Thus the
boundaries satisfy \(A_{n_r}\asymp r\) as \(r\to\infty\). Inverting this
relation leads to the choice
\(
n_r:=\bigl\lceil r^{1/(1-\delta)}\bigr\rceil\). 
Let block \(r\) have index set
\(\mathcal I_r:=\{n_r,\ldots,n_{r+1}-1\}\). Accordingly,
\(A_{n_{r+1}}-A_{n_r}
=\sum_{k\in\mathcal I_r}\alpha_k\asymp1\).

On the natural timescale, the blocks become wider. Within block \(r\), the
step sizes are of order \(n_r^{-\delta}\). Covering a fixed amount of
algorithmic time therefore requires \(|\mathcal I_r|\asymp n_r^\delta\).
It follows that
\[
\sum_{k\in\mathcal I_r}\alpha_k^2
\asymp
|\mathcal I_r|n_r^{-2\delta}
\asymp n_r^{-\delta}.
\]
This sum determines the second-moment scale of the stochastic fluctuations accumulated within
one block.  Its polynomial order matches that of the MSE estimate at the
block entrance.  The later maximal estimate uses this balance to control all
possible exit times in the block jointly.

When \(\delta=1\), \(A_n\asymp\log n\).  Fixed increments on the algorithmic
timescale then correspond to geometrically growing block boundaries, leading
to \(n_r=2^r\).

\begin{figure}[t]
    \FIGURE{
    \includegraphics[width=0.9\textwidth]{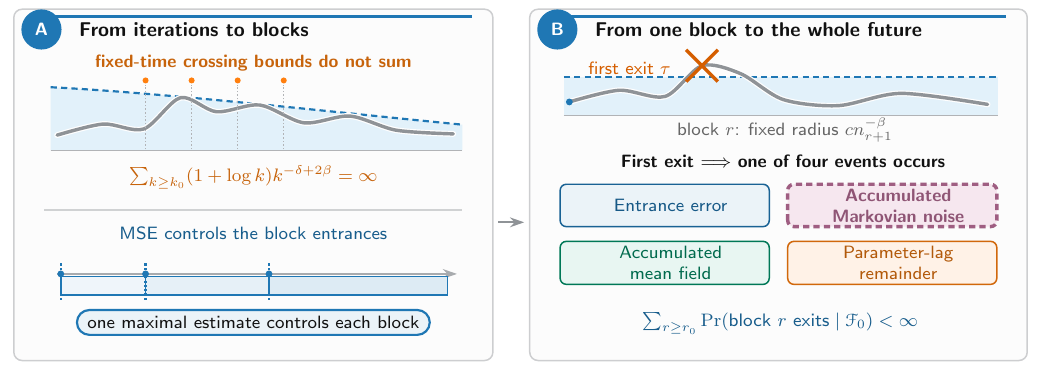}
    }{From Fixed-Time Bounds to Shrinking-Tube Concentration.  \label{fig:blocking}}
    {}\end{figure}

Within each block, we analyze the first exit, which captures exits followed by
a return to the tube, and decompose the parameter error at that exit into
four contributions: entrance error,
accumulated Markovian noise, accumulated mean field, and
parameter-lag remainder. The MSE estimate controls the entrance error.
Together with mean-field regularity, it also bounds the accumulated mean field.
The adaptive-to-reference comparison bound and a dependent maximal estimate
control the accumulated Markovian noise. Lipschitz regularity of \(F\) and bounds on
\(\lVert\theta_k-\theta_{k-\ell}\rVert\) control the parameter-lag remainder.
Section~\ref{sec:concen-bounds} develops
these estimates and shows that an exit is impossible when all four
contributions remain within their assigned bounds.

Finally, the corresponding block-exit probabilities form a series that
converges exactly over the range of
\(\delta\) and \(\beta\) stated in Theorem~\ref{thm:trajectory}.  We apply the union bound beginning with the
block that contains \(k_0\), including the part of that block after \(k_0\),
and then include all later blocks.  The resulting event keeps every iterate
from \(k_0\) onward inside the prescribed tube.

The next three sections carry out this argument.
Section~\ref{sec:backward-telescoping} develops backward kernel replacement technique to control the bias from the adaptive
Markov kernels.
Section~\ref{sec:MSE} incorporates this
control into the MSE recursion and proves the MSE bound.
Section~\ref{sec:concen-bounds} then converts the fixed-time estimate into the
shrinking-tube guarantee through a blockwise first-exit argument.

\section{Backward Kernel Replacement}
\label{sec:backward-telescoping}

The adaptive chain evolves through a sequence of kernels selected by the SA
iterates.  The transition from \(W_s\) to \(W_{s+1}\) uses
\(P_{\theta_s}\), so an expectation at time \(k\) involves the random,
time-inhomogeneous kernel product
\(
P_{\theta_0}P_{\theta_1}\cdots P_{\theta_{k-1}}
\).
The geometric-ergodicity bound in
\eqref{eq:uniform-geometric-ergodicity}, by contrast, applies to powers of a
single fixed kernel.  We therefore need to compare the adaptive product with a
homogeneous reference product.

We obtain this comparison via backward kernel replacement.  Fixing
\(P_{\theta_0}\) as the reference kernel, we replace the adaptive kernels one
at a time, beginning with the transition closest to time \(k\).  Consecutive
hybrid products differ at only one transition and share the same homogeneous
reference tail.  Geometric ergodicity along that tail reduces the effect of
each replacement, while telescoping the hybrid differences collects the local
errors into a single bound.  A final application of geometric ergodicity bounds the
difference between the reference product and its steady-state average.  The resulting
estimate expresses the bias from changing kernels as a geometrically weighted
sum of local parameter changes.

\subsection{Replacing the Kernels Backward in Time}
\label{subsec:kernel-telescoping-comparison}

Fix a horizon \(k\geq1\) and a Borel measurable test function
\(f\) satisfying \(|f|\leq U_1\).  We compare the actual terminal expectation
with the expectation under the homogeneous reference kernel: 
\begin{equation}\label{eq:telescoping-target}
\E[f(W_k)\mid\F_0]-(P_{\theta_0}^{k}f)(W_0).
\end{equation}
To connect these two quantities, define, for \(s=1,\ldots,k\),
\[
\Lambda_s(f)
:=
\E\bigl[
(P_{\theta_0}^{k-s}f)(W_s)
\,\bigm|\,\F_0
\bigr].
\]
The hybrid represented by \(\Lambda_s(f)\) follows the adaptive Markov chain
through time \(s\) and then uses \(P_{\theta_0}\) for the remaining \(k-s\)
transitions.  At the two endpoints,
\[
\Lambda_k(f)=\E[f(W_k)\mid\F_0]
\quad\mbox{and}\quad
\Lambda_1(f)=(P_{\theta_0}^{k}f)(W_0).
\]
The second identity uses the fact that the transition from \(W_0\) to \(W_1\)
is already governed by \(P_{\theta_0}\).  Hence the difference in
\eqref{eq:telescoping-target} is 
\(\Lambda_k(f)-\Lambda_1(f)\).
Figure~\ref{fig:backward-kernel-telescoping-main} illustrates this sequence of
hybrids.

\begin{figure}[t]
    \FIGURE{
    \includegraphics[width=0.9\textwidth]{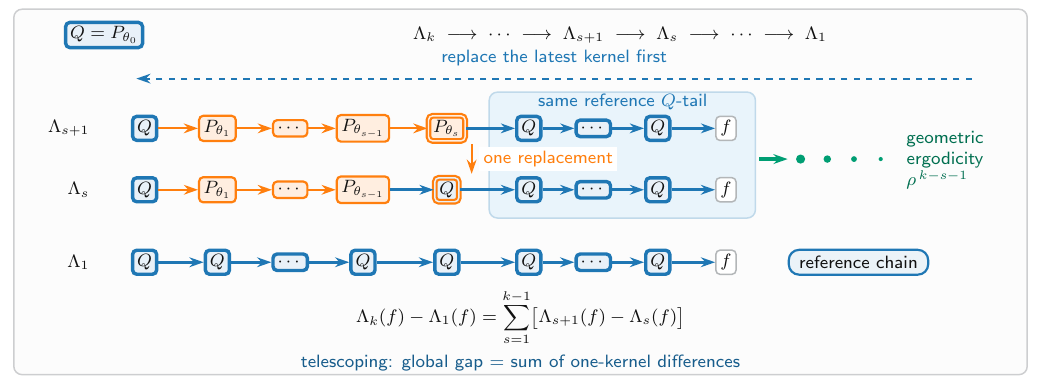}
    }{Backward Kernel Replacement.
      \label{fig:backward-kernel-telescoping-main}}
    {Starting from the adaptive product, the kernels are replaced by
    \(P_{\theta_0}\) from the last transition backward.  Adjacent hybrids share
    the same adaptive prefix and the same homogeneous reference tail.  They
    differ only at the transition from time \(s\) to time \(s+1\).  The
    reference tail has length \(k-s-1\), which yields the decay factor
    \(\rho^{k-s-1}\).}
\end{figure}

The order of replacement matters.  Read from \(\Lambda_k\) back to
\(\Lambda_1\), the sequence replaces the adaptive Markov kernels starting with the
last transition.  For each \(s\), the hybrids \(\Lambda_{s+1}\) and
\(\Lambda_s\) have the same evolution through \(W_s\).  They differ only in
the next transition, which uses \(P_{\theta_s}\) in one hybrid and
\(P_{\theta_0}\) in the other.  Both then use \(P_{\theta_0}\) for the
remaining transitions.  Beginning the replacements at the terminal end makes
each future tail homogeneous, allowing the geometric-ergodicity bound to
apply.

For \(s=1,\ldots,k-1\), define the centered reference tail 
\(
g_s
:=
P_{\theta_0}^{k-s-1}f-\nu_{\theta_0}(f)\).
This centering isolates the part that decays under geometric ergodicity.
Because Markov kernels preserve constants, it leaves the kernel difference
unchanged.  Conditioning first on \(\F_s\) yields 
\begin{align}
\Lambda_{s+1}(f)-\Lambda_s(f)
&=
\E\bigl[
\bigl(
(P_{\theta_s}-P_{\theta_0})
P_{\theta_0}^{k-s-1}f
\bigr)(W_s)
\,\bigm|\,\F_0
\bigr]
\nonumber\\
&=
\E\bigl[
\bigl((P_{\theta_s}-P_{\theta_0})g_s\bigr)(W_s)
\,\bigm|\,\F_0
\bigr].
\label{eq:one-step-diff}
\end{align}

Summing these adjacent differences telescopes from the homogeneous reference
expectation to the actual adaptive expectation:
\begin{align}
\E[f(W_k)\mid\F_0]-(P_{\theta_0}^k f)(W_0)
&=
\sum_{s=1}^{k-1}
\bigl[\Lambda_{s+1}(f)-\Lambda_s(f)\bigr].
\label{eq:gap-decomp}
\end{align}
Each summand now contains one local kernel difference and one centered
homogeneous reference tail.  
Next, we bound these two parts using
weighted kernel continuity and geometric ergodicity.

\subsection{An Adaptive-to-Reference Comparison Bound}

By the \(U_1\)-weighted geometric-ergodicity bound
\eqref{eq:uniform-geometric-ergodicity},
\(
|g_s(w)|
\leq C\rho^{k-s-1}U_1(w).
\)
Weighted kernel continuity then bounds the effect of the single kernel
replacement:
\[
\bigl|
\bigl((P_{\theta_s}-P_{\theta_0})g_s\bigr)(W_s)
\bigr|
\leq
C\rho^{k-s-1}
U_1(W_s)\lVert\theta_s-\theta_0\rVert.
\]
Combining this estimate with
\eqref{eq:one-step-diff} and \eqref{eq:gap-decomp} yields the following
comparison.  

\begin{lemma}
\label{lemma:kernel-telescoping}
Suppose
Assumptions~\ref{assump:compact}--\ref{assump:kernel-Lipschitz} hold.  Let
\(f:\W\to\R\) be Borel measurable with \(|f|\leq U_1\).  Then there is a positive constant \(C\) such that, for all \(k\geq1\),
\begin{align}
\left|
\E[f(W_k)\mid\F_0]-(P_{\theta_0}^k f)(W_0)
\right|
\leq
C\sum_{s=1}^{k-1}
\rho^{k-s-1}
\E\left[
U_1(W_s)\lVert\theta_s-\theta_0\rVert
\,\bigm|\,\F_0
\right].
\label{eq:kernel-telescoping-bound}
\end{align}
\end{lemma}

The bound separates three effects.  The magnitude
\(\lVert\theta_s-\theta_0\rVert\) of the parameter change bounds the
corresponding kernel change.  The weight \(U_1(W_s)\) accounts
for the size of the Markov state.  The factor \(\rho^{k-s-1}\) decreases with
the length of the remaining reference tail.  Kernel changes made earlier in
the product therefore have less effect on the terminal expectation.  The
total difference is bounded by summing these weighted local changes.

For later use, we shift the comparison to a window beginning at time \(t\) and
having length \(\ell\).  We use \(P_{\theta_t}\) as the homogeneous reference kernel and
take \(F(\theta_t,\cdot)\) as the terminal test function.  Applying the
time-shifted adaptive-to-reference comparison bound coordinatewise, followed by
geometric ergodicity of the reference chain, yields, for all \(t\geq0\) and
\(\ell\geq1\), 
\begin{equation}\label{eq:mse-adaptive-to-stationary}
\begin{aligned}
\bigl\lVert
\E[F(\theta_t,W_{t+\ell})\mid\F_t]-\bar F(\theta_t)
\bigr\rVert
\leq
\underbrace{C\rho^\ell U_1(W_t)}
_{\text{reference-chain transient}}
+
\underbrace{
C\sum_{s=1}^{\ell-1}
\rho^{\ell-s-1}
\E\bigl[
U_1(W_{t+s})
\lVert\theta_{t+s}-\theta_t\rVert
\,\bigm|\,\F_t
\bigr]}
_{\text{adaptive-to-reference difference}} .
\end{aligned}
\end{equation}
The first term is the error between the homogeneous reference chain and its
stationary distribution.  The second is the error between the adaptive Markov chain
and the reference chain.  It sums the kernel changes within the window, with
each change reduced according to the length of the remaining reference tail.

\section{Finite-Time MSE Bound with an Adaptive Markov Chain}
\label{sec:MSE}

The backward kernel replacement argument in
Section~\ref{sec:backward-telescoping} isolates the effect of the changing
kernels on the conditional update.  We now show that the resulting bias does
not change the leading polynomial order of the finite-time MSE.  Over a
logarithmic look-back window, this bias contributes only a higher-order term
to the squared-error recursion, while the mean field retains its full
QSM coefficient.  The resulting estimate is useful in its own right
and provides the block-entrance and lagged-error bounds needed for the
trajectory analysis.

\subsection{MSE Bound}
\label{subsec:canonical-mse-result}

Although the observed update is generally not centered at the mean field \(\bar F(\theta_k)\), its 
bias is small enough for the MSE to track the step-size scale.

\begin{proposition}
\label{prop:expectation}
Suppose
Assumptions~\ref{assump:learning-rate}--\ref{assump:qsm} hold, with
\(0<\delta\leq1\).  When \(\delta=1\), also suppose that
\(2\zeta\alpha_0>1\). Then there are positive constants \(C\) and \(K\) such that,
for all \(k\geq K\),
\begin{equation}\label{eq:risk-bound-theta}
\E[\lVert\theta_k-\theta^\star\rVert^2\mid\F_0]
\leq C U_2(W_0)(1+\log k)k^{-\delta}.
\end{equation}
\end{proposition}

The estimate follows the step-size scale: apart from the logarithmic factor,
the MSE decays as \(k^{-\delta}\).  In particular,  step sizes with \(\delta=1\) yield
an order \((1+\log k)/k\) bound under the additional condition on
\(\alpha_0\). The proposition applies throughout \(0<\delta\leq1\), including
step-size schedules outside the range covered by
Theorem~\ref{thm:trajectory}.\footnote{Its proof allows \(\theta^\star\) to lie on
the boundary of \(\Theta\).}  The factor \(U_2(W_0)\) records the dependence on
the realized initial Markov state.

When \(P_\theta\equiv P\), the effect of kernel adaptation disappears and the
model returns to the fixed-kernel Markovian setting studied in finite-time
analyses of temporal-difference learning and contractive SA
\citep{BhandariRussoSingal21,
ChenMaguluriShakkottaiShanmugam24}.  Relative to these fixed-kernel results,
Proposition~\ref{prop:expectation} allows the transition kernel to change with
the iterate and permits a potentially unbounded Markov state under weighted
second-moment conditions.  The update need not be a gradient or arise from a
contractive operator.  

This result is closely related to the analysis of
\citet{CheDongTong26}, who establish finite-time guarantees for
stochastic gradient descent with parameter-dependent Markov sampling.
Under strong convexity and step sizes of order \(1/k\), they obtain
an MSE bound of order \(\tau(\log T)^2/T\), with explicit dependence
on the mixing time \(\tau\).  For the general SA recursion considered
here, Proposition~\ref{prop:expectation} gives an MSE bound of order
\((1+\log k)/k\) under QSM and the same step-size order.
The dependence on the uniform geometric-ergodicity constants is
absorbed into the multiplicative constant.

\subsection{Controlling Bias from Adaptive Markov Kernels}

To prove Proposition~\ref{prop:expectation}, we begin with the one-step
update in the squared distance to the target.  Because
\(\theta^\star\in\Theta\) and the Euclidean projection is nonexpansive,
\begin{align}
\lVert\theta_{k+1}-\theta^\star\rVert^2
\leq{}&
\lVert\theta_k-\theta^\star\rVert^2
+2\alpha_k
\langle\theta_k-\theta^\star,F(\theta_k,W_k)\rangle
+\alpha_k^2\lVert F(\theta_k,W_k)\rVert^2 .
\label{eq:proj-nonexpansive}
\end{align}
The final term is controlled directly.  The linear growth condition on \(F\) implies 
\[
\E[\lVert F(\theta_k,W_k)\rVert^2\mid\F_0]
\lesssim U_2(W_0)
,\]
so its contribution is of order \(\alpha_k^2\).
The main task is therefore to control the inner-product term.

If the observed update direction could be replaced by \(\bar F(\theta_k)\),
Assumption~\ref{assump:qsm} would give
\[
\langle\theta_k-\theta^\star,\bar F(\theta_k)\rangle
\leq
-\zeta\lVert\theta_k-\theta^\star\rVert^2.
\]
The actual update uses \(F(\theta_k,W_k)\), where \(W_k\) reflects the
sequence of kernels selected before time \(k\).  It is therefore generally
not conditionally centered at \(\bar F(\theta_k)\).  We show that this
difference contributes only a higher-order remainder, while the full
QSM coefficient \(\zeta\) is retained.

Fix \(k>\ell\geq1\).  The parameter
\(\theta_{k-\ell}\) is known at the beginning of the interval \([k-\ell,k]\), which 
allows us to compare the adaptive dynamics over this interval with the
reference chain governed by the fixed kernel \(P_{\theta_{k-\ell}}\).
Adding and subtracting the update \(F(\theta_{k-\ell}, W_{k-\ell})\) and mean field \(\bar F(\theta_{k-\ell})\) decomposes
the inner-product term into three parts: the mean-field contribution at
\(\theta_k\), errors caused by the parameter change from \(\theta_{k-\ell}\) to
\(\theta_k\), and the lagged bias 
\[
\langle\theta_{k-\ell}-\theta^\star,
F(\theta_{k-\ell},W_k)-\bar F(\theta_{k-\ell})\rangle.
\]
QSM controls the first part.  The Lipschitz and growth conditions, together
with the bound on parameter changes, control the second.  The
adaptive-to-reference comparison from
Section~\ref{sec:backward-telescoping} controls the lagged bias.

To bound the lagged bias, apply
\eqref{eq:mse-adaptive-to-stationary} with \(t=k-\ell\).  Its first term
measures the remaining transient error of the fixed-kernel reference chain.
Its second term measures the effect of changing kernels within the comparison
window and involves
\(
U_1(W_{t+s})\lVert\theta_{t+s}-\theta_t\rVert,
\)
for \(s=1,\ldots,\ell-1\).
Because \(\theta_t,\theta^\star\in\Theta\) and \(\Theta\) is compact,
\(\lVert\theta_t-\theta^\star\rVert\) is uniformly bounded.  The same
adaptive-to-reference estimate therefore controls the lagged inner product.
The bound on parameter changes gives
\[
\E\bigl[
U_1(W_{t+s})
\lVert\theta_{t+s}-\theta_t\rVert
\,\bigm|\,\F_0
\bigr]
\lesssim
U_2(W_0)
\sum_{i=t}^{t+s-1}\alpha_i.
\]
Applying this estimate to each term in the sum yields
\[
\begin{aligned}
\E\bigl[
\bigl\langle\theta_{k-\ell}-\theta^\star,
F(\theta_{k-\ell},W_k)-\bar F(\theta_{k-\ell})\bigr\rangle
\,\bigm|\,\F_0
\bigr] \lesssim
U_2(W_0)\biggl(
\rho^\ell+\sum_{i=k-\ell}^{k-1}\alpha_i
\biggr).
\end{aligned}
\]
Substituting the lagged-bias bound and the corresponding
bounds on parameter changes into the preceding decomposition yields the
following result.  The complete proof appears in 
Section~\ref{subsec:ec-canonical-mse-proofs}.

\begin{lemma}
\label{lemma:lagged-restoring-drift}
Suppose
Assumptions~\ref{assump:learning-rate}--\ref{assump:qsm} hold.  Then there
is a positive constant \(C\) such that, for all \(k>\ell\geq1\),
\begin{align}
\E[\langle\theta_k-\theta^\star,F(\theta_k,W_k)\rangle
\mid\F_0]
\leq
-\zeta \E[\lVert\theta_k-\theta^\star\rVert^2\mid\F_0]
+C U_2(W_0)\biggl(\rho^\ell
+\sum_{i=k-\ell}^{k-1}\alpha_i\biggr). \label{eq:lagged-restoring-drift}
\end{align}
\end{lemma}

It remains to choose the lag \(\ell\). The two terms in \eqref{eq:lagged-restoring-drift} involving \(\ell\) reflect different timescales. The term \(\rho^\ell\) measures convergence on the reference-chain timescale and decreases as the window grows. The sum \(\sum_{i=k-\ell}^{k-1}\alpha_i\) controls the bound on parameter change over the comparison window and increases with the window length. Balancing these two
effects leads to the logarithmic choice
\begin{equation}\label{eq:main-log-lag}
\ell_k:=
\max\!\biggl\{
1,
\biggl\lceil\frac{\delta\log k}{|\log\rho|}\biggr\rceil
\biggr\}
\quad\mbox{and}\quad
t_k:=k-\ell_k.
\end{equation}
Since \(\ell_k=\mathcal O(\log k)\), the condition \(k>\ell_k\) in Lemma~\ref{lemma:lagged-restoring-drift} holds for all \(k\) large enough. For this choice,
\[
\rho^{\ell_k}\leq k^{-\delta}=\alpha_k/\alpha_0
\quad\mbox{and}\quad
\sum_{i=t_k}^{k-1}\alpha_i
\lesssim
\alpha_k(1+\log k).
\]
Thus, the remainder \(\bigl(\rho^\ell
+\sum_{i=k-\ell}^{k-1}\alpha_i\bigr)\) in \eqref{eq:lagged-restoring-drift}  is of order
\(\alpha_k(1+\log k)\).
In \eqref{eq:proj-nonexpansive}, this remainder is multiplied by \(2\alpha_k\), contributing a term of order \(\alpha_k^2(1+\log k)\) to the MSE recursion. The QSM coefficient \(\zeta\) remains unchanged.

\subsection{MSE Recursion}
\label{subsec:mse-recursion}

Apply Lemma~\ref{lemma:lagged-restoring-drift} with \(\ell=\ell_k\) as defined 
in \eqref{eq:main-log-lag} to \eqref{eq:proj-nonexpansive}, and use the
quadratic-term bound above. 
Then there are positive constants \(C\) and \(K\) such that, for all
\(k\geq K\),
\begin{equation}\label{eq:recursion}
\E[\lVert\theta_{k+1}-\theta^\star\rVert^2\mid\F_0]
\leq
\bigl(1-2\zeta\alpha_0k^{-\delta}\bigr)
\E[\lVert\theta_k-\theta^\star\rVert^2\mid\F_0]
+C U_2(W_0)(1+\log k)k^{-2\delta}.
\end{equation}

For \(0<\delta<1\), iterating \eqref{eq:recursion} yields the MSE bound \eqref{eq:risk-bound-theta}. For \(\delta=1\), the same bound holds under \(2\zeta\alpha_0>1\). The full calculation appears in Section~\ref{subsec:MSE-bound}.  

The MSE bound holds throughout \(0<\delta\leq1\); the stronger condition
\(\delta>1/2\) enters only when fixed-time control is upgraded to an
all-future guarantee.  Proposition~\ref{prop:expectation} controls the iterate
at the entrances to the future blocks, but it does not rule out an excursion
between two such entrances.  Section~\ref{sec:concen-bounds} provides the
required within-block maximal control to prove Theorem~\ref{thm:trajectory}.

\section{From MSE to Shrinking-Tube Concentration}
\label{sec:concen-bounds}

We first reduce an exit within one block to four
component events.  We then develop the maximal estimate for accumulated
Markovian noise, the only component that requires a new time-uniform
argument.  The final subsection combines the component estimates to bound the exit
probability within each block and sums these bounds over the future blocks.

\subsection{A Four-Event Decomposition of Block Exit}

Because \(\theta^\star\in\operatorname{int}(\Theta)\), choose
\(d_\star>0\) such that
\(\{\theta:\lVert\theta-\theta^\star\rVert\leq d_\star\}
\subseteq\operatorname{int}(\Theta)\),
and set
\(\widetilde c:=\min\{c,d_\star\}\).  Containment in the
\(\widetilde c\)-tube implies containment in the prescribed
\(c\)-tube.  We therefore use \(\widetilde c\) as the tube
constant below and write it again as \(c\).

For \(r\geq1\), let block \(r\) have index set
\(\mathcal I_r:=\{n_r,\ldots,n_{r+1}-1\}\), and set
\(\psi_r:=c n_{r+1}^{-\beta}\). Fix \(k_0\),
and let \(r_0\) be the index of the block containing \(k_0\).
If the trajectory
exits the shrinking tube after \(k_0\), then an exit occurs in some block
\(r\geq r_0\). 
Enlarging the relevant part of block \(r_0\) to the full block
and applying the union bound yields
\begin{align}
\pr\Bigl(
\exists k\geq k_0:
\lVert\theta_k-\theta^\star\rVert>c k^{-\beta}
\,\Bigm|\,\F_0
\Bigr)
&\leq
\sum_{r=r_0}^{\infty}
\pr\Bigl(
\exists k\in\mathcal I_r:
\lVert\theta_k-\theta^\star\rVert>c k^{-\beta}
\,\Bigm|\,\F_0
\Bigr) \nonumber\\ 
&\leq
\sum_{r=r_0}^{\infty}
\pr\Bigl(
\exists k\in\mathcal I_r:
\lVert\theta_k-\theta^\star\rVert>c n_{r+1}^{-\beta}
\,\Bigm|\,\F_0
\Bigr)  \nonumber\\
&=\sum_{r=r_0}^{\infty}\pr\Bigl(
\max_{k\in\mathcal I_r}
\lVert\theta_k-\theta^\star\rVert>\psi_r
\,\Bigm|\,\F_0
\Bigr). \label{eq:all-future-block-reduction}
\end{align}
The second inequality holds because \(k<n_{r+1}\) for all
\(k\in\mathcal I_r\).
Thus it suffices to bound the exit probability with radius \(\psi_r\) for each
block \(r\geq r_0\).

Fix a block \(r\geq r_0\). An exit during block \(r\) can arise in two
broad ways. The iterate may enter the block too far from \(\theta^\star\), or
the parameter may change too much after block entrance.
We control the first possibility through the entrance event
\[
\mathcal E_r^{\mathrm{ent}}
:=\bigl\{\lVert\theta_{n_r}-\theta^\star\rVert
>\psi_r/8\bigr\}.
\]
The entrance threshold is smaller than the block radius \(\psi_r\) to
leave room for subsequent parameter changes within the block. 
The particular 
fraction \(1/8\) is a convenient choice.

Even when the entrance event does not occur, subsequent parameter changes
can take the iterate outside the tube. To express these changes as a sum,
first ignore the projection. The induction below handles the projection
formally. In this unprojected picture, the parameter change from block
entrance through update \(m\in\mathcal I_r\) is
\begin{align*}
    \theta_{m+1}-\theta_{n_r}
=
\sum_{k=n_r}^{m}\alpha_kF(\theta_k,W_k) 
=S_r^{\mathrm{markov}}(m)
+S_r^{\mathrm{field}}(m)
+S_r^{\mathrm{lag}}(m).
\end{align*}
The second equality decomposes this parameter change into three contributions:
accumulated Markovian noise, accumulated mean field, and parameter-lag
remainder. 
We use the same lag \(\ell_{n_r}\) throughout block \(r\). 
Specifically,
\begin{gather*}
S_r^{\mathrm{markov}}(m)
:=\sum_{k=n_r}^{m}\alpha_k
[F(\theta_{k-\ell_{n_r}},W_k)-\bar F(\theta_{k-\ell_{n_r}})],\\
S_r^{\mathrm{field}}(m)
:=\sum_{k=n_r}^{m}\alpha_k\bar F(\theta_{k-\ell_{n_r}}), 
\quad\mbox{and}\quad 
S_r^{\mathrm{lag}}(m)
:=\sum_{k=n_r}^{m}\alpha_k
[F(\theta_k,W_k)-F(\theta_{k-\ell_{n_r}},W_k)].
\end{gather*}
Furthermore, define the
three partial-sum events by
\[
\mathcal E_r^{\mathord{\bullet}}
:=
\biggl\{
\max_{m\in\mathcal I_r}
\lVert S_r^{\mathord{\bullet}}(m)\rVert
>\psi_r/8
\biggr\},
\]
where \(\mathord{\bullet}\) stands for
\(\mathrm{markov}\), \(\mathrm{field}\), or \(\mathrm{lag}\).
Together with \(\mathcal E_r^{\mathrm{ent}}\), these are the four component
events. We next show that they cover the block-exit event:
\begin{equation}\label{eq:block-exit-decomp}
\biggl\{
\max_{k\in\mathcal I_r}
\lVert\theta_k-\theta^\star\rVert>\psi_r
\biggr\}
\subseteq
\mathcal E_r^{\mathrm{ent}}
\cup
\mathcal E_r^{\mathrm{markov}}
\cup
\mathcal E_r^{\mathrm{field}}
\cup
\mathcal E_r^{\mathrm{lag}}.    
\end{equation}

To prove this inclusion, fix an outcome outside the four component events.
Then the entrance error and all three partial-sum maxima are at most
\(\psi_r/8\).

We show by induction that
\(
\lVert\theta_k-\theta^\star\rVert\leq \psi_r\) for all \(k\in\mathcal I_r\).
The entrance bound establishes the base case \(k=n_r\):
\(\lVert\theta_{n_r}-\theta^\star\rVert\leq \psi_r/8<\psi_r\).

Now fix \(m\in\{n_r,\ldots,n_{r+1}-2\}\) and suppose that 
\(\lVert\theta_k-\theta^\star\rVert\leq \psi_r\) for all
\(k=n_r,\ldots,m\). 
By construction, \(\psi_r\leq c\leq d_\star\), so
\(\theta_{n_r},\ldots,\theta_m\) are interior points of \(\Theta\). The
projections at updates
\(k=n_r,\ldots,m-1\) therefore have interior outputs. If a projection
changes its input, its output lies on the boundary of \(\Theta\).
Hence these projections are inactive, and we may sum the updates:
\[
\theta_m
=
\theta_{n_r}
+\sum_{k=n_r}^{m-1}\alpha_kF(\theta_k,W_k).
\]

The projection at update \(m\) may still be active. The nonexpansiveness of
\(\Pi_\Theta\), followed by the three-part decomposition, yields
\begin{align*}
\lVert\theta_{m+1}-\theta^\star\rVert
&\leq \bigl\lVert\theta_m+\alpha_mF(\theta_m,W_m)-\theta^\star\bigr\rVert\\
&=
\biggl\lVert
\theta_{n_r}-\theta^\star
+\sum_{k=n_r}^{m}\alpha_kF(\theta_k,W_k)
\biggr\rVert\\
&\leq
\lVert\theta_{n_r}-\theta^\star\rVert
+\lVert S_r^{\mathrm{markov}}(m)\rVert
+\lVert S_r^{\mathrm{field}}(m)\rVert
+\lVert S_r^{\mathrm{lag}}(m)\rVert\\
&\leq 4(\psi_r/8)=\psi_r/2
<\psi_r.
\end{align*}
Thus \(\theta_{m+1}\) also satisfies the induction bound, so the induction
closes. Hence, outside the four component events, every iterate in block \(r\)
remains within \(\psi_r\) of \(\theta^\star\). This proves
\eqref{eq:block-exit-decomp}.

Taking conditional probabilities in \eqref{eq:block-exit-decomp} and applying
the union bound yields
\begin{align}
\pr\Bigl(
\max_{k\in\mathcal I_r}
\lVert\theta_k-\theta^\star\rVert>\psi_r
\,\Bigm|\,\F_0
\Bigr)
\leq
\pr(\mathcal E_r^{\mathrm{ent}}\mid\F_0)
+\pr(\mathcal E_r^{\mathrm{markov}}\mid\F_0)
+\pr(\mathcal E_r^{\mathrm{field}}\mid\F_0)
+\pr(\mathcal E_r^{\mathrm{lag}}\mid\F_0). \label{eq:block-exit-decomp-2}
\end{align}

\subsection{Maximal Control of Accumulated Markovian Noise}

To control \(\mathcal E_r^{\mathrm{markov}}\), we separate the accumulated
Markovian noise \(S_r^{\mathrm{markov}}(m)\) into a conditional-mean component and a centered component.
Fix a block \(r\), and write \(I:=n_r\), \(J:=n_{r+1}-1\), and
\(L:=J-I+1=\lvert\mathcal I_r\rvert\). Recall that the logarithmic lag
\(\ell:=\ell_{n_r}\) is fixed throughout this block; take \(r\) large enough
that \(I\geq2\ell\). For \(k\in\mathcal I_r\), let
\(\xi_{k,\ell}:=F(\theta_{k-\ell},W_k)-\bar F(\theta_{k-\ell})\), and define
\[
\xi_{k,\ell}^{\mathrm{pred}}
:=\E[\xi_{k,\ell}\mid\F_{k-\ell}]
\quad\mbox{and}\quad
\xi_{k,\ell}^{\mathrm{cent}}
:=\xi_{k,\ell}-\xi_{k,\ell}^{\mathrm{pred}}.
\]
This \emph{lagged centering} gives the decomposition: 
\[
S_r^{\mathrm{markov}}(m)
=
\underbrace{\sum_{k=I}^{m}\alpha_k\xi_{k,\ell}^{\mathrm{pred}}}
_{S_r^{\mathrm{pred}}(m)}
+
\underbrace{\sum_{k=I}^{m}\alpha_k\xi_{k,\ell}^{\mathrm{cent}}}
_{S_r^{\mathrm{cent}}(m)},
\]
for \(I\leq m\leq J\).
By the triangle inequality and the union bound, for
any \(x>0\),
\[
\pr\biggl(
\max_{m\in\mathcal I_r}\lVert S_r^{\mathrm{markov}}(m)\rVert\geq x
\,\biggm|\,\F_0\biggr)
\leq
\pr\biggl(
\max_{m\in\mathcal I_r}\lVert S_r^{\mathrm{pred}}(m)\rVert\geq x/2
\,\biggm|\,\F_0\biggr)
+
\pr\biggl(
\max_{m\in\mathcal I_r}\lVert S_r^{\mathrm{cent}}(m)\rVert\geq x/2
\,\biggm|\,\F_0\biggr).
\]

The two components are controlled by different mechanisms. The conditional
mean \(\xi_{k,\ell}^{\mathrm{pred}}\) is small because the reference chain approaches stationarity while the
iterates change little over the lag interval. Its accumulated contribution
can therefore be bounded by summing first moments. The centered component
satisfies \(\E[\xi_{k,\ell}^{\mathrm{cent}}\mid\F_{k-\ell}]=0\), which makes
cross terms between sufficiently separated indices vanish. This cancellation
allows a second-moment analysis, followed by L\'evy's inequality to control
the running maximum.

\subsubsection{Maximal Bound for the Conditional Mean.}

Fix \(k\in\mathcal I_r\). The conditional mean
\(\xi_{k,\ell}^{\mathrm{pred}}
=\E[F(\theta_{k-\ell},W_k)\mid\F_{k-\ell}]-\bar F(\theta_{k-\ell})\)
differs from zero for two reasons. Even if the chain used the fixed kernel
\(P_{\theta_{k-\ell}}\) throughout the interval, its distribution after \(\ell\)
transitions could still depend on \(W_{k-\ell}\). Reference-chain geometric
ergodicity bounds this transient discrepancy by a constant times
\(\rho^\ell U_1(W_{k-\ell})\). The actual chain also changes its kernel as the
iterate changes. Backward kernel replacement from
Section~\ref{sec:backward-telescoping} bounds this additional discrepancy
by weighting the parameter change at time \(k-\ell+s\) by
\(\rho^{\ell-s-1}\), the geometric decay over the remaining reference
transitions.

To bound each weighted parameter change, the SA update and adaptive moment
control give
\[\E[U_1(W_{k-\ell+s})\lVert\theta_{k-\ell+s}-\theta_{k-\ell}\rVert\mid\F_0]
\lesssim U_2(W_0)\sum_{i=k-\ell}^{k-\ell+s-1}\alpha_i
\leq U_2(W_0)s\alpha_{k-\ell},\]
as established in Lemma~\ref{lemma:ec-canonical-movement}.
Since the geometric weights sum to at most \((1-\rho)^{-1}\),
their weighted sum of \(s\) is at most \(\ell/(1-\rho)\).
Taking conditional expectations given \(\F_0\) therefore yields the
first-moment bound in
Lemma~\ref{lemma:ec-canonical-lagged-centering}:
\[
\E[\lVert\xi_{k,\ell}^{\mathrm{pred}}\rVert\mid\F_0]
\lesssim U_2(W_0)\bigl(\rho^\ell+\ell\alpha_{k-\ell}\bigr).
\]

For every partial sum, applying the triangle inequality and Markov's inequality yields 
\begin{equation}\label{eq:conditional-mean-maximal-bound}
\pr\biggl(
\max_{I\leq m\leq J}\lVert S_r^{\mathrm{pred}}(m)\rVert\geq x
\,\biggm|\,\F_0\biggr)
\leq x^{-1}\sum_{k=I}^{J}\alpha_k
\E[\lVert\xi_{k,\ell}^{\mathrm{pred}}\rVert\mid\F_0]
\lesssim U_2(W_0)x^{-1}L\alpha_I
\bigl(\rho^\ell+\ell\alpha_{I-\ell}\bigr).
\end{equation}

\subsubsection{L\'evy's Inequality for the Centered Component.}

We first reduce the running-maximum bound to two moment estimates, which
we establish in the next two subsections.
Write \(S_m:=S_r^{\mathrm{cent}}(m)\) for \(I-1\leq m\leq J\), and let
\(V_m:=\E[\lVert S_J-S_m\rVert^2\mid\F_m]\) be the conditional second
moment of the remaining sum after time \(m\). 
Lemma~\ref{lemma:conditional-levy-remainder} gives, for all \(x>0\),
\begin{equation}\label{eq:centered-conditional-levy}
\pr\biggl(
\max_{I-1\leq m\leq J}\lVert S_m\rVert\geq x
\,\biggm|\,\F_0\biggr)
\leq
8x^{-2}\E[\lVert S_J\rVert^2\mid\F_0]
+
\pr\biggl(
\max_{I-1\leq m\leq J}V_m\geq x^2/8
\,\biggm|\,\F_0\biggr).
\end{equation}

This inequality is the key reduction in the maximal analysis.  
It replaces
the running maximum by two second-moment quantities: the second moment of
the full-block sum \(S_J\), and the conditional second moments \(V_m\) of
the remaining sums. Lagged centering makes both quantities tractable
because cross-product terms between indices at least \(\ell\) apart vanish
in their second-moment expansions.

The condition
\(\E[\xi_{k,\ell}^{\mathrm{cent}}\mid\F_{k-\ell}]=0\), however, does not
generally make \(\{S_m\}\) a martingale with respect to \(\{\F_m\}\).
We therefore cannot obtain \eqref{eq:centered-conditional-levy} from a
standard martingale maximal inequality. Instead, we use the
conditional-median correction underlying the dependent-sequence L\'evy
inequality in Lemma~\ref{lemma:extended-levy-dependent}
\citep[Section~32.1]{LoeveProbBook}.

Let \(q_m\) denote the conditional
median  of \(\lVert S_m\rVert-\lVert S_J\rVert\) given
\(\F_m\). Then,
\[
\pr\bigl(
\lVert S_J\rVert\geq\lVert S_m\rVert-q_m
\,\bigm|\,\F_m\bigr)\geq\frac12.
\]
Fix \(z>0\), and let \(T\) be the first index in
\(\{I-1,\ldots,J\}\) at which
\(\lVert S_m\rVert-q_m\geq z\), with \(T=\infty\) if there is no such
index. Since the event \(\{T=m\}\) is \(\F_m\)-measurable, the conditional-median property implies
that, on this event, the conditional probability of \(\lVert S_J\rVert\geq z\) is at least
\(1/2\). Summing over the disjoint
first-crossing events therefore gives
\[
\begin{aligned}
\frac12\pr(T\leq J\mid\F_0)
=\frac12\sum_{m=I-1}^{J}\pr(T=m\mid\F_0)
\leq\sum_{m=I-1}^{J}
\pr(T=m,\, \lVert S_J\rVert\geq z\mid\F_0)
\leq\pr(\lVert S_J\rVert\geq z\mid\F_0).
\end{aligned}
\]
Because \(\{T\leq J\}\) is exactly the event that \(\lVert S_m\rVert-q_m\) reaches
\(z\) somewhere in the block, Markov's inequality gives
\begin{equation}\label{eq:bound-corrected-value}
\pr\biggl(
\max_{I-1\leq m\leq J}(\lVert S_m\rVert-q_m)\geq z
\,\biggm|\,\F_0\biggr)
\leq
2\pr(\lVert S_J\rVert\geq z\mid\F_0)
\leq2z^{-2}\E[\lVert S_J\rVert^2\mid\F_0].
\end{equation}

It remains to control the conditional median \(q_m\). By the median
property, 
\(\bigl|\lVert S_m\rVert-\lVert S_J\rVert\bigr|\geq\lvert q_m\rvert\)
with conditional probability at least \(1/2\) given \(\F_m\). 
Therefore, 
\[
q_m^2
\leq 2\E\bigl[
(\lVert S_m\rVert-\lVert S_J\rVert)^2\mid\F_m\bigr]
\leq 2\E[\lVert S_J-S_m\rVert^2\mid\F_m]
=2V_m,
\]
where the two inequalities follow from Markov's inequality and the triangle inequality, respectively. 
Hence, if \(\max_{I-1\leq m\leq J}V_m<x^2/8\), then
\(\lvert q_m\rvert<x/2\) for every \(m\). Any crossing
\(\lVert S_m\rVert\geq x\) must then imply
\(\lVert S_m\rVert-q_m\geq x/2\). Consequently,
\[
\Bigl\{\max_{I-1\leq m\leq J}\lVert S_m\rVert\geq x\Bigr\}
\subseteq
\Bigl\{\max_{I-1\leq m\leq J}
(\lVert S_m\rVert-q_m)\geq x/2\Bigr\} \cup
\Bigl\{\max_{I-1\leq m\leq J}V_m\geq x^2/8\Bigr\}.
\]
Applying the bound \eqref{eq:bound-corrected-value} with \(z=x/2\) to the first
event on the right yields \eqref{eq:centered-conditional-levy}.

\subsubsection{Second Moment of the Full Block Sum.}

If \(i<k\) and \(k-i\geq\ell\), then
\(\xi_{i,\ell}^{\mathrm{cent}}\) is known at time \(k-\ell\), whereas
\(\E[\xi_{k,\ell}^{\mathrm{cent}}\mid\F_{k-\ell}]=0\). The
tower property of conditional expectations gives
\[
\E\bigl[
(\xi_{k,\ell}^{\mathrm{cent}})^\intercal\xi_{i,\ell}^{\mathrm{cent}}
\mid\F_t
\bigr]=0,
\qquad I\leq i<k\leq J,\quad 0\leq t\leq i,\quad k-i\geq\ell.
\]
Thus, expanding \(\E[\lVert S_J\rVert^2\mid\F_0]\) leaves only cross
terms between indices less than \(\ell\) apart. Each index belongs to
at most \(2(\ell-1)\) such pairs. Cauchy--Schwarz bounds their contributions by the individual second
moments, which satisfy
\(\E[\lVert\xi_{k,\ell}^{\mathrm{cent}}\rVert^2\mid\F_0]
\lesssim U_2(W_0)\) by
Lemma~\ref{lemma:ec-canonical-lagged-centering}. Consequently,
\begin{equation}\label{eq:centered-full-block-second-moment}
\E[\lVert S_J\rVert^2\mid\F_0]
\lesssim U_2(W_0)\ell L\alpha_I^2.
\end{equation}

\subsubsection{Conditional Second Moments of the Remaining Sums and the Maximal Bound.}

Applying the pair-counting argument above 
conditional on \(\F_m\) gives
\[
V_m\leq(2\ell-1)\sum_{k=m+1}^{J}
\alpha_k^2
\E[\lVert\xi_{k,\ell}^{\mathrm{cent}}\rVert^2\mid\F_m],
\]
for all \(I-1\leq m\leq J\).
For \(k\geq m+\ell\), we have
\(\F_m\subseteq\F_{k-\ell}\), and conditional centering gives
\(\E[\|\xi_{k,\ell}^{\mathrm{cent}}\|^2\mid\F_m]
\leq\E[\|\xi_{k,\ell}\|^2\mid\F_m]\).
Indeed, expanding
\(\|\xi_{k,\ell}\|^2
=\|\xi_{k,\ell}^{\mathrm{cent}}+\xi_{k,\ell}^{\mathrm{pred}}\|^2\)
gives a cross term with conditional mean zero given \(\F_{k-\ell}\),
since \(\xi_{k,\ell}^{\mathrm{pred}}\) is measurable there and
\(\E[\xi_{k,\ell}^{\mathrm{cent}}\mid\F_{k-\ell}]=0\).
The tower property preserves this cancellation given \(\F_m\), while
the squared norm of \(\xi_{k,\ell}^{\mathrm{pred}}\) contributes a
nonnegative term.
For \(m<k<m+\ell\), the conditional mean
\(\xi_{k,\ell}^{\mathrm{pred}}
=\E[\xi_{k,\ell}\mid\F_{k-\ell}]\)
is \(\F_m\)-measurable, so we bound its squared norm separately.
The linear-growth condition \eqref{eq:LinearGrowth-F} and
drift condition \eqref{eq:drift-condition} then yield
\[
\E[\|\xi_{k,\ell}^{\mathrm{cent}}\|^2\mid\F_m]
\lesssim
\begin{cases}
1+\lambda^{k-m}U_2(W_m), & k\geq m+\ell,\\
1+U_2(W_m)+U_2(W_{k-\ell}), & m<k<m+\ell,
\end{cases}
\]
where \(\lambda<1\) is the coefficient in the
drift condition \eqref{eq:drift-condition}.
The proof is given in
Lemma~\ref{lemma:ec-canonical-lagged-centering}.

We sum the bounds on
\(\E[\|\xi_{k,\ell}^{\mathrm{cent}}\|^2\mid\F_m]\)
over \(k=m+1,\ldots,J\), which yields 
\[
V_m\lesssim
\ell L\alpha_I^2
+\ell\alpha_I^2
\biggl[
\ell U_2(W_m)
+\sum_{i=m-\ell+1}^{m-1}U_2(W_i)
+\ell
\biggr].
\]
Iterating the uniform drift condition \eqref{eq:drift-condition} gives
\(\E[U_2(W_i)\mid\F_0]\lesssim U_2(W_0)\), uniformly in \(i\geq0\);
see Lemma~\ref{lemma:adaptive-moment-control}.
Thus, the expression in brackets has conditional expectation at most
a constant times \(\ell U_2(W_0)\), uniformly in \(m\).
Bounding the maximum of these nonnegative expressions by their sum
over \(m\), and then applying Markov's inequality, we obtain that for any \(y>0\), 
\begin{equation}\label{eq:remaining-sum-moment-tail}
\pr\biggl(
\max_{I-1\leq m\leq J}V_m\geq y
\,\biggm|\,\F_0\biggr)
\leq y^{-1} \E\biggl[\max_{I-1\leq m\leq J}V_m\biggm|\F_0\biggr]
\lesssim U_2(W_0)y^{-1}\ell^2L\alpha_I^2.
\end{equation}

Substituting the second-moment bound \eqref{eq:centered-full-block-second-moment}
and the tail bound \eqref{eq:remaining-sum-moment-tail} with \(y=x^2/8\)
into \eqref{eq:centered-conditional-levy}, and using \(\ell\geq1\) yields that for any \(x>0\),
\begin{equation}\label{eq:centered-noise-maximal-bound}
\pr\biggl(
\max_{m\in\mathcal I_r}\lVert S_r^{\mathrm{cent}}(m)\rVert\geq x
\,\biggm|\,\F_0\biggr)
\lesssim U_2(W_0)x^{-2}\ell^2L\alpha_I^2.
\end{equation}
Furthermore, 
combining the two maximum bounds \eqref{eq:conditional-mean-maximal-bound}
and \eqref{eq:centered-noise-maximal-bound}
at threshold \(x/2\) gives the
block-maximal estimate of Lemma~\ref{lemma:ec-canonical-block-maximal}.
Figure~\ref{fig:noise-event} summarizes the argument.

\begin{figure}[t]
    \FIGURE{
    \includegraphics[width=\textwidth]{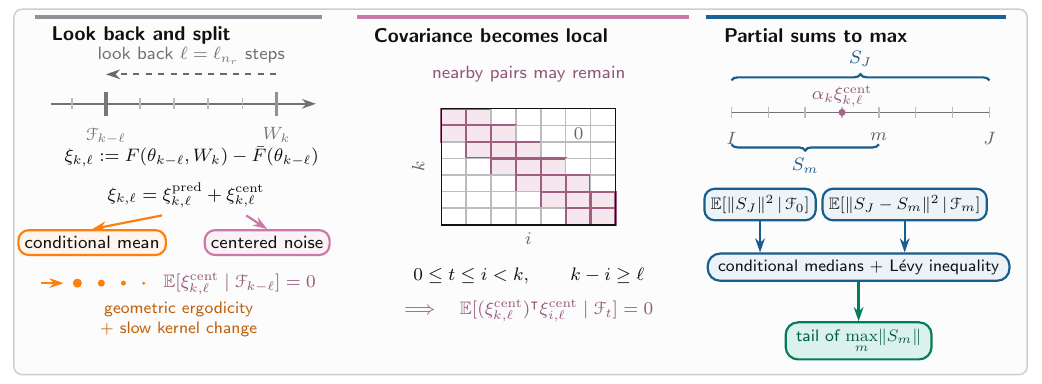}
    }{Lagged Centering and the Block Maximum of Centered Markovian Noise.
    \label{fig:noise-event}}
    {Lagged centering separates the conditional mean from centered noise.
    For \(I:=n_r\), \(J:=n_{r+1}-1\), and
    \(S_m:=\sum_{k=I}^{m}\alpha_k\xi_{k,\ell}^{\mathrm{cent}}\),
    distant cross terms vanish in the second moments of \(S_J\) and \(S_J-S_m\).
    Conditional medians and L\'evy's inequality
    convert these estimates into control of
    \(\pr(\max_{I\leq m\leq J}\lVert S_m\rVert\geq x\mid\F_0)\).}
\end{figure}

We now apply the maximal estimates at the block threshold,
leaving the block length \(L=n_{r+1}-n_r\) unspecified.
Suppose that \(n_{r+1}\leq2n_r\) for all sufficiently large \(r\).
For \(x=\psi_r/8\), combining \eqref{eq:conditional-mean-maximal-bound}
and \eqref{eq:centered-noise-maximal-bound} at threshold \(x/2\), with the logarithmic lag defined in
\eqref{eq:main-log-lag}, gives
\begin{equation}\label{eq:block-markov-bound}
\pr(\mathcal E_r^{\mathrm{markov}}\mid\F_0)
\lesssim
U_2(W_0)\ell_{n_r}^2L n_r^{-2\delta+2\beta},
\end{equation}
for all sufficiently large \(r\).
Section~\ref{subsec:ec-poly-block-noise} gives the complete maximal argument.

\subsection{Block Exit Bounds and Summation}

The remaining three component events in the 
decomposition \eqref{eq:block-exit-decomp-2} require only estimates established
earlier in the analysis. For all \(r\) large enough,
\begin{equation}\label{eq:remaining-block-event-bounds}
\begin{aligned}
\pr(\mathcal E_r^{\mathrm{ent}}\mid\F_0)
&\lesssim
U_2(W_0)(1+\log n_r)n_r^{-\delta+2\beta},\\
\pr(\mathcal E_r^{\mathrm{field}}\mid\F_0)
&\lesssim
U_2(W_0)(1+\log n_r)L^2n_r^{-3\delta+2\beta},\\
\pr(\mathcal E_r^{\mathrm{lag}}\mid\F_0)
&\lesssim
U_2(W_0)\ell_{n_r}L n_r^{-2\delta+\beta}.
\end{aligned}
\end{equation}
The first bound follows from Proposition~\ref{prop:expectation} and
Markov's inequality at the block entrance.  The second uses Lipschitz
continuity of \(\bar F\), the MSE bound at the lagged iterates, and
Cauchy--Schwarz over the block.  The third follows from the
Lipschitz condition on \(F\) and the bound on parameter changes over the
logarithmic lag. These calculations appear in the proofs in
Sections~\ref{subsec:ec-poly-block-entrance},
\ref{subsec:ec-poly-block-drift}, and
\ref{subsec:ec-poly-block-adaptation}, respectively, before the specific
block lengths are substituted.

Combining \eqref{eq:block-markov-bound} and
\eqref{eq:remaining-block-event-bounds} with
\eqref{eq:block-exit-decomp-2}, using
\(\ell_{n_r}\lesssim1+\log n_r\) and \(\beta\geq0\), gives
\[
\pr\biggl(
\max_{k\in\mathcal I_r}
\lVert\theta_k-\theta^\star\rVert>\psi_r
\,\biggm|\,\F_0
\biggr)
\lesssim
U_2(W_0)(1+\log n_r)^2n_r^{-\delta+2\beta}
\left[
1+\frac{L}{n_r^\delta}
+\left(\frac{L}{n_r^\delta}\right)^2
\right].
\]

Longer blocks reduce the number of entrance bounds but increase
the mean-field bound quadratically in \(L\).
Since the number of blocks over a comparable range of iterations
is inversely proportional to \(L\), balancing these contributions
gives \(L\asymp n_r^\delta\).
The Markovian-noise and parameter-lag bounds are linear in
\(L\), so their summed contributions do not affect this balance.

This choice gives a cumulative step size of constant order per block.
We implement it with polynomial blocks (\(
n_r:=
\lceil r^{1/(1-\delta)}\rceil\)) 
 when \(1/2<\delta<1\),
and dyadic blocks (\(n_r:=
2^r\)) when \(\delta=1\).
The combined bound therefore becomes
\begin{equation}\label{eq:one-block-exit-bound}
\pr\biggl(
\max_{k\in\mathcal I_r}
\lVert\theta_k-\theta^\star\rVert
>c n_{r+1}^{-\beta}
\,\biggm|\,\F_0
\biggr)
\lesssim
U_2(W_0)(1+\log n_r)^2n_r^{-\delta+2\beta}.
\end{equation}

Let \(r_0\) denote the block containing \(k_0\).  For \(k_0\) sufficiently
large, \eqref{eq:all-future-block-reduction} and
\eqref{eq:one-block-exit-bound} yield
\begin{equation}\label{eq:block-tail-sum}
\pr\Bigl(
\exists k\geq k_0:
\lVert\theta_k-\theta^\star\rVert>c k^{-\beta}
\,\Bigm|\,\F_0
\Bigr)
\lesssim
U_2(W_0)
\sum_{r=r_0}^{\infty}
(1+\log n_r)^2n_r^{-\delta+2\beta}.
\end{equation}

Suppose first that \(1/2<\delta<1\).  The relation
\(n_r\asymp r^{1/(1-\delta)}\) implies
\begin{equation}\label{eq:block-tail-sum-1}
\sum_{r=r_0}^{\infty}
(1+\log n_r)^2n_r^{-\delta+2\beta}
\lesssim
(1+\log r_0)^2
r_0^{-\{2(\delta-\beta)-1\}/(1-\delta)} \lesssim (1+\log k_0)^2
k_0^{-(2(\delta-\beta)-1)}.
\end{equation}
Here, for the first step,
the series is summable precisely when \(\beta<\delta-1/2\), and 
the second step holds 
because the block containing \(k_0\) satisfies
\(r_0\asymp k_0^{1-\delta}\).

When \(\delta=1\), the dyadic choice \(n_r=2^r\) gives
\begin{equation}\label{eq:block-tail-sum-2}
\sum_{r=r_0}^{\infty}
(1+\log n_r)^2n_r^{-1+2\beta}
\lesssim
(1+r_0)^2 2^{-r_0(1-2\beta)}
\lesssim
(1+\log k_0)^2k_0^{-(1-2\beta)}.
\end{equation}
Substituting \eqref{eq:block-tail-sum-1} and \eqref{eq:block-tail-sum-2} into \eqref{eq:block-tail-sum} proves
Theorem~\ref{thm:trajectory}.

\section{Sharpness of the Polynomial Decay Exponent}
\label{sec:sharpness}

Theorem~\ref{thm:trajectory} bounds the all-future exit probability with
polynomial exponent \(2(\delta-\beta)-1\). The following result shows
that this exponent cannot be increased under the same assumptions.

\begin{theorem}
\label{thm:sharpness}
Fix the step-size and tube parameters as in
Theorem~\ref{thm:trajectory}.There is a class of processes
satisfying recursion~\eqref{eq:update-scheme-projection} and
that theorem's hypotheses with common numerical assumption
constants. For each \(\epsilon>0\), some process in this
class satisfies 
\begin{equation}\label{eq:sharpness-failure-lower-bound}
\pr\Bigl(
\exists k\geq k_0:
\|\theta_k-\theta^\star\|>c k^{-\beta}
\,\Bigm|\,\F_0
\Bigr)
\geq
C_\epsilon
k_0^{-\left(2(\delta-\beta)-1+\epsilon\right)},
\end{equation}
for all \(k_0\ge K_\epsilon\), where 
\(C_\epsilon\) and \(K_\epsilon\) are some positive constants
\end{theorem}

The lower-bound exponent can be made arbitrarily close to
\(2(\delta-\beta)-1\). Consequently, no larger polynomial exponent holds
for all processes satisfying the assumptions, even when the bound's
constant and iteration threshold may depend on the process. Together,
Theorems~\ref{thm:trajectory} and~\ref{thm:sharpness} identify the
worst-case tradeoff at the level of polynomial exponents. Finite second moments control the
average squared size of the noise but still permit rare observations
large enough to cause late exits. Faster tube shrinkage makes smaller
updates sufficient for an exit and reduces the polynomial decay
exponent of the all-future guarantee. This limitation persists even
with fixed noise variance and a uniformly stable mean field.

To illustrate the proof, consider the one-dimensional recursion 
\(\theta_{k+1}
=\Pi_\Theta(
\theta_k-\alpha_k\zeta(\theta_k-\theta^\star)+\alpha_kW_k)\). The observations after initialization are independent and identically distributed (iid) with a symmetric, signed-Pareto distribution which has zero mean and unit variance. 
For each \(\epsilon>0\), let their tail 
index be \(2+\epsilon/(\delta-\beta)\), which approaches the
finite-variance boundary as \(\epsilon\to0\).

Conditional on the trajectory remaining within the tube through iteration \(k\),
the noise contribution to the unprojected update is \(\alpha_kW_k\).  A
sufficiently large noise realization on the scale
\(
k^{-\beta}\alpha_k^{-1}
\asymp
k^{\delta-\beta}
\)
can therefore drive the next iterate outside a tube whose radius is of order
\(k^{-\beta}\).  The probability of a
realization on this scale is of order
\(
k^{-\left(2(\delta-\beta)+\epsilon\right)}.
\)
Iterating the corresponding one-step containment bounds gives a lower
bound proportional to
\(
\sum_{k\geq k_0}k^{-\left(2(\delta-\beta)+\epsilon\right)}
\asymp
k_0^{-\left(2(\delta-\beta)-1+\epsilon\right)}
\). 
The proof makes this argument precise by bounding the probability of
remaining in the tube from one iteration to the next, and extending the one-dimensional example to multiple dimensions. 

Theorem~\ref{thm:ec-sharpness-fixed-class} in
Section~\ref{sec:EC-sharpness-proofs} gives a more formal, worst-case statement over
an admissible class of data-generating processes with fixed numerical bounds and a common
joint initial distribution. The same lower
bound holds for the essential supremum of the conditional exit
probabilities over this class. Moreover, for each \(\epsilon>0\),
a single process in the class satisfies the lower bound for every
sufficiently large \(k_0\). Thus, the lower bound does not rely on
changing the problem parameters or initialization as \(\epsilon\)
decreases.

\section{Extension: Martingale-Difference Noise and Predictable Bias}
\label{sec:structured-extension}

The SA recursion~\eqref{eq:update-scheme-projection} forms its update direction from the adaptive Markovian
noise through \(F(\theta_k,W_k)\).  Many algorithms introduce an
additional update-stage error through simulation, batching, smoothing,
or an auxiliary estimate.  We decompose this error into
martingale-difference noise  and predictable bias.

The added martingale-difference noise has zero mean given the information
available before each update. Allowing its conditional second-moment bound
to grow can make the tube-exit probability bound decay more slowly.
The predictable bias is the conditional mean of the added error and can
repeatedly push the iterate away from the target. The rate at which the
bias bound decreases limits how quickly the tube may shrink in our
concentration bound.

\subsection{Extended SA Recursion}
\label{subsec:extension-recursion-timing}

The extended SA recursion is
\begin{equation}\label{eq:extension-recursion}
\theta_{k+1}
=\Pi_\Theta\!\Bigl(
\theta_k+\alpha_k\bigl(F(\theta_k,W_k)+M_{k+1}+B_k\bigr)
\Bigr),
\end{equation}
where \(M_{k+1}\) is the added martingale-difference noise and
\(B_k\) is the predictable bias. 
For example, in stochastic optimization, a gradient may be
estimated using finite differences of noisy objective evaluations.
The resulting sampling variability and approximation bias
contribute to \(M_{k+1}\) and \(B_k\), respectively.
Reducing the finite-difference increment can decrease the bias
while increasing the variance \citep{Spall03}.

Each iteration of the recursion \eqref{eq:extension-recursion} reveals new information in two stages: first while
forming \(\theta_{k+1}\), and then when drawing \(W_{k+1}\).
We distinguish these stages because the added noise is centered
before the update is observed, whereas the Markov transition
condition must hold after all update information has been revealed.

Let \(\mathscr G_k^-\) denote the information before update \(k\).
The variables \(\theta_k\), \(W_k\), and \(B_k\) depend only on
this information. Let \(\chi_{k+1}\) collect all information
revealed during the update, including \(M_{k+1}\), and let
\(\mathscr G_k^+\) include this new information.
At this point, \(\theta_{k+1}\) is determined, and the next state
\(W_{k+1}\) is drawn using the stored iterate \(\theta_k\).
Formally, define
\begin{equation}\label{eq:extension-interlacing}
\mathscr G_0^-:=\sigma(\theta_0,W_0),\quad 
\mathscr G_k^+:=\mathscr G_k^-\vee\sigma(\chi_{k+1}),\quad\mbox{and}\quad
\mathscr G_{k+1}^-:=\mathscr G_k^+\vee\sigma(W_{k+1}),
\end{equation}
for all \( k\geq0\).
We assume that the next-state transition retains its prescribed
distribution conditional on the complete update information: for all
\(E\in\Borel(\W)\) and \(k\geq0\),
\begin{equation}\label{eq:extension-post-update-transition}
\pr(W_{k+1}\in E\mid\mathscr G_k^+)
=P_{\theta_k}(W_k,E).
\end{equation}
Thus, after all information from update \(k\) has been revealed,
the conditional distribution of \(W_{k+1}\) is still
\(P_{\theta_k}(W_k,\cdot)\). 
Figure~\ref{fig:extension-filtration} summarizes the update sequence
and the available information.

\begin{figure}[ht]
    \FIGURE{
    \includegraphics[width=\textwidth]{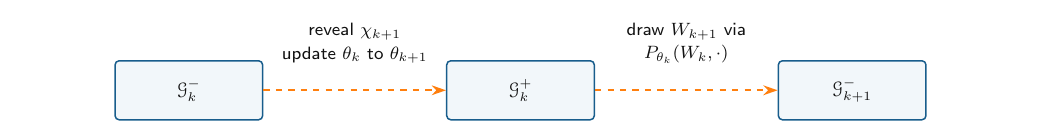}
    }{Update Timing and the Nested Filtration.
    \label{fig:extension-filtration}}
    {}
\end{figure}

The following assumption specifies the magnitude of \(M_{k+1}\) and 
\(B_k\). 

\begin{assumption}
\label{assump:structured-errors}
\begin{enumerate}[label=(\roman*)]
    \item
For all \(k\geq0\),
\(M_{k+1}\) is \(\mathscr G_k^+\)-measurable and conditionally square integrable with 
\begin{equation}\label{eq:extension-martingale-conditions}
\E[M_{k+1}\mid\mathscr G_k^-]=0
\quad\mbox{and}\quad
\E[\lVert M_{k+1}\rVert^2\mid\mathscr G_k^-]
\leq \SFm_k^2U_2(W_k).
\end{equation}

    \item
For all \(k\geq0\), the variable \(B_k\) is
\(\mathscr G_k^-\)-measurable and satisfies
\(\lVert B_k\rVert\leq\SFb_k\).

    \item
There are nonnegative constants \(C_{\mathsf M}\),
\(C_{\mathsf B}\), and \(\omega_{\mathsf M}\), and a
positive constant \(\omega_{\mathsf B}\) such that, for all
\(k\geq0\),
\begin{equation}\label{eq:extension-power-laws}
\SFm_k\leq C_{\mathsf M}(k+1)^{\omega_{\mathsf M}}
\quad\mbox{and}\quad
\SFb_k\leq C_{\mathsf B}(k+1)^{-\omega_{\mathsf B}}.
\end{equation}
\end{enumerate}
\end{assumption}

Let \(\Delta_{k+1}:=M_{k+1}+B_k\) denote the total added direction.
Assumption~\ref{assump:structured-errors} implies 
\(
B_k=\E[\Delta_{k+1}\mid\mathscr G_k^-]\) and
\(
M_{k+1}=\Delta_{k+1}-B_k
\).
Thus, \(B_k\) is the average added direction given the information
before the update, and \(M_{k+1}\) is the fluctuation around that
average.

The process
\(\{
\sum_{i=1}^n\alpha_iM_{i+1},\mathscr G_n^+\}_{n\geq1}
\)
is a martingale by the conditional centering in
Assumption~\ref{assump:structured-errors}\textup{(i)} and the field
inclusions in \eqref{eq:extension-interlacing}.
Equation~\eqref{eq:extension-post-update-transition} also
preserves the adaptive Markov structure, so the backward kernel replacement
argument in Section~\ref{sec:backward-telescoping} remains available under the enlarged information fields.

\subsection{Effects on MSE and Shrinking-Tube Concentration}
\label{subsec:extension-results}

We now give MSE and shrinking-tube concentration bounds for the
extended recursion, conditional on \(\mathscr G_0^+\).
These bounds describe the trajectory after the first update,
when \(\theta_1\) is known and \(W_1\) has yet to be drawn.
Averaging over the first update's randomness gives the same
bounds conditional on the initial information
\(\mathscr G_0^-\).

\begin{proposition}
\label{thm:structured-mse}
Consider the extended SA recursion~\eqref{eq:extension-recursion}.  Suppose
Assumptions~\ref{assump:learning-rate}--\ref{assump:structured-errors} hold,
with \(0<\delta\leq 1\).  When \(\delta=1\), also suppose
\(2\zeta\alpha_0>1\).  Fix
\(
0\leq\omega_{\mathsf M}<\delta/2
\)
and \(\omega_{\mathsf B}>0\).  Then there are positive constants \(C\) and \(K\) such that,
for all \(k\geq K\),
\begin{equation}\label{eq:structured-mse-bound}
\E[\lVert\theta_k-\theta^\star\rVert^2\mid\mathscr G_0^+]
\leq C U_2(W_0)
\bigl[
(1+\log k)k^{-\delta}
+(1+\log k)k^{-\delta+2\omega_{\mathsf M}}
+k^{-2\omega_{\mathsf B}}
\bigr].
\end{equation}
\end{proposition}

The first term in \eqref{eq:structured-mse-bound} retains the order
\((1+\log k)k^{-\delta}\) from Proposition~\ref{prop:expectation} for
the SA recursion with Markovian noise. The second term reflects the
added martingale-difference noise \(M_{k+1}\).
Assumption~\ref{assump:structured-errors} bounds its conditional mean
squared magnitude by \(\SFm_k^2U_2(W_k)\). The power-law bound on
\(\SFm_k\) permits the factor \(k^{2\omega_{\mathsf M}}\) in the second
term; the condition \(\omega_{\mathsf M}<\delta/2\) ensures that this
term tends to zero. The third term, \(k^{-2\omega_{\mathsf B}}\),
accounts for predictable bias. If the bias bound decreases slowly,
this term can decay most slowly and determine the order of the MSE
bound, even when \(\omega_{\mathsf M}=0\).

The following theorem gives the corresponding shrinking-tube bound.

\begin{theorem}
\label{thm:structured-trajectory}
Consider the extended SA recursion~\eqref{eq:extension-recursion}.  Suppose
Assumptions~\ref{assump:learning-rate}--\ref{assump:structured-errors} hold,
with \(1/2<\delta\leq1\).  When \(\delta=1\), also suppose
\(2\zeta\alpha_0>1\).  Fix
\(0\leq\omega_{\mathsf M}<\delta-1/2\) and
\(0\leq\beta<\min\bigl\{\omega_{\mathsf B},
\delta-\omega_{\mathsf M}-1/2\bigr\}\), and \(c>0\).
If \(B_k=0\) almost surely for all \(k\), interpret
\(\omega_{\mathsf B}=\infty\); if
\(M_{k+1}=0\) almost surely for all \(k\), take
\(\omega_{\mathsf M}=0\).  There are positive constants \(C\) and \(K\) such that, for all
\(k_0\geq K\),
\begin{equation}\label{eq:structured-trajectory-bound}
\pr\Bigl(
\lVert\theta_k-\theta^\star\rVert\leq c k^{-\beta}
\text{ for all }k\geq k_0
\,\Bigm|\,\mathscr G_0^+
\Bigr)
\geq
1-C U_2(W_0)(1+\log k_0)^2
k_0^{-(2(\delta-\beta-\omega_{\mathsf M})-1)}.
\end{equation}
\end{theorem}

When \(\omega_{\mathsf B}\geq
\delta-\omega_{\mathsf M}-1/2\), predictable bias neither
restricts the allowable tube rates (\(0\leq \beta < \delta-\omega_{\mathsf M}-1/2\)) nor changes the polynomial
decay exponent in \eqref{eq:structured-trajectory-bound}.
Every allowable \(\beta\) then satisfies
\(\beta<\omega_{\mathsf B}\).
Intuitively, the mean field counteracts the repeated biased
updates, allowing their effect to be accounted for within a
radius that shrinks faster than the prescribed tube.
The bias can still affect the constant \(C\) and threshold \(K\).

The martingale-difference contribution to an update is
\(\alpha_kM_{k+1}\). Its conditional mean squared magnitude is bounded
by \(\alpha_k^2\SFm_k^2U_2(W_k)\). The moment bounds for \(W_k\)
control the factor \(U_2(W_k)\). Dividing the remaining factor by the
squared tube radius and summing over future iterations gives
\[
\sum_{k\geq k_0}
\left(\frac{\alpha_k\SFm_k}{k^{-\beta}}\right)^2
\lesssim
\sum_{k\geq k_0}k^{-2(\delta-\beta-\omega_{\mathsf M})}
\asymp
k_0^{-\left(2(\delta-\beta-\omega_{\mathsf M})-1\right)}.
\]
The polynomial series is finite exactly when
\(\beta<\delta-\omega_{\mathsf M}-1/2\). Its tail explains the polynomial
exponent in \eqref{eq:structured-trajectory-bound}: the factor
\(\SFm_k^2\) contributes \(2\omega_{\mathsf M}\), and summing over
future iterations reduces the decay exponent by one. For a fixed tube
rate \(\beta\), the exponent in the exit-probability bound is therefore
\(2\omega_{\mathsf M}\) smaller than in Theorem~\ref{thm:trajectory}.

The original guarantees are preserved under different requirements on
bias decay. When \(\omega_{\mathsf M}=0\), the MSE order in
Proposition~\ref{prop:expectation} is retained whenever
\(\omega_{\mathsf B}\geq\delta/2\). In the trajectory regime
\(1/2<\delta\leq1\), the exit-probability exponent in
Theorem~\ref{thm:trajectory} is unchanged, and its full range
\(0\leq\beta<\delta-1/2\) remains available whenever
\(\omega_{\mathsf B}\geq\delta-1/2\). For \(1/2<\delta<1\), a bias rate
satisfying \(\delta-1/2\leq\omega_{\mathsf B}<\delta/2\) can therefore
make the MSE bound bias-dominated while preserving every original tube
rate. At \(\delta=1\), the two bias thresholds coincide.

The following result shows that the exponent
\(2(\delta-\beta-\omega_{\mathsf M})-1\), including the reduction
\(2\omega_{\mathsf M}\) associated with the permitted growth of the
martingale-difference second-moment bound, cannot be increased in general.

\begin{theorem}
\label{thm:structured-sharpness}
Fix the step-size, tube parameters, and the exponents
\(\omega_{\mathsf M},\omega_{\mathsf B}\) as in
Theorem~\ref{thm:structured-trajectory}.
There is a class of processes satisfying the extended SA
recursion~\eqref{eq:extension-recursion} and that theorem's hypotheses with
common numerical assumption constants. 
For each \(\epsilon>0\), some process in this class satisfies 
\begin{equation}\label{eq:structured-sharpness-martingale-lower-bound}
\pr\Bigl(
\exists k\geq k_0:
\|\theta_k-\theta^\star\|>c k^{-\beta}
\,\Bigm|\,
\mathscr G_0^+
\Bigr)
\geq
C_\epsilon
k_0^{-\left(2(\delta-\beta-\omega_{\mathsf M})-1+\epsilon\right)},
\end{equation}
for all \(k_0\geq K_\epsilon\), where \(C_\epsilon\) and \(K_\epsilon\) are some positive constants. 
\end{theorem}

The construction adapts the scalar signed-Pareto example for
Theorem~\ref{thm:sharpness} by placing the unit-variance heavy-tailed
variables in the martingale-difference term \(M_{k+1}\) and multiplying them by
\((k+1)^{\omega_{\mathsf M}}\). Their contribution to an update has scale
\(\alpha_k(k+1)^{\omega_{\mathsf M}}\), so an unscaled Pareto variable of
order \(k^{\delta-\beta-\omega_{\mathsf M}}\) can force an exit, in place
of the scale \(k^{\delta-\beta}\) in Section~\ref{sec:sharpness}.
After initialization, the Markov observations are bounded, nondegenerate,
and iid, and the construction retains a nonzero predictable bias.
The exponent loss therefore persists with Markovian noise and predictable
bias present. Section~\ref{sec:EC-structured-sharpness-proofs}
gives the fixed-class statement and complete proof.

\section{Application: Inventory Learning with Fixed Stockout Costs}
\label{sec:inventory-application}

We study stock-level learning when availability affects future demand and
each stockout incurs a fixed cost. The extended theory bounds stock-level
error at a specified update and the probability that any later stock-level
error exceeds a tolerance that decreases over time.

\subsection{Inventory Model and Objective}
\label{subsec:inventory-model}

Stockouts can affect future demand as well as current sales
\citep{AndersonFitzsimonsSimester06}. We use the stockout-damping model
of \citet{CheDongTong26}, with base-stock level
\(\theta_k\in\Theta\), for some interval \(\Theta\) to be specified later. Demand evolves as
\begin{equation}\label{eq:inventory-demand}
D_{k+1}
=\bigl[a\min\{D_k+u_{k+1},\theta_k\}+(1-a)m
+\varepsilon_{k+1}\bigr]^+.
\end{equation}
Here \(0<a<1\) controls demand feedback, \(m>0\) enters the baseline
demand term, and \(\sigma>0\) is the input standard deviation.
The inputs \(u_k\) and \(\varepsilon_k\) are mutually independent
\(N(0,\sigma^2)\) variables. The cap at \(\theta_k\) limits how much
current demand carries into the next period.

In addition to holding and proportional shortage costs, each stockout
incurs a fixed cost \(q>0\), regardless of the shortage size
\citep{BenkheroufSethi10}:
\begin{equation}\label{eq:inventory-fixed-cost}
\mathcal C(\theta,d)
=h(\theta-d)^++b(d-\theta)^+
+q\ind\{d>\theta\},
\end{equation}
for some positive constants \(h\) and \(b\).
For each fixed stock level \(\theta\geq0\), let \(D_\infty(\theta)\)
have the stationary demand distribution. The expected cost when this
stock level is held fixed is
\(f(\theta):=\E[\mathcal C(\theta,D_\infty(\theta))]\).
We seek its minimizer over \(\Theta\).

\subsection{SA Formulation}
\label{subsec:inventory-algorithm}

A natural approach to minimizing \(f\) over \(\Theta\) is projected
stochastic gradient descent, which requires an estimate of \(f'(\theta)\).
Infinitesimal perturbation analysis (IPA) obtains such estimates by
differentiating the sample cost along a demand trajectory
\citep{CheDongTong26}. This derivative must account for the dependence
of demand on the stock level.

For a fixed stock level \(\theta>0\), write \(D_k(\theta)\) for stationary
demand and \(L_k(\theta)=dD_k(\theta)/d\theta\) for its sensitivity,
with the Gaussian inputs held fixed. Differentiating
\eqref{eq:inventory-demand} gives sensitivity zero when the next demand
is zero. When the next demand is positive, the sensitivity is \(a\) if the stock
cap is active and \(aL_k(\theta)\) otherwise. During learning, carry
the same recursion at the current stock level:
\begin{equation}\label{eq:inventory-sensitivity}
L_{k+1}
=a\ind\{D_{k+1}>0\}
\bigl[\ind\{D_k+u_{k+1}>\theta_k\}
+L_k\ind\{D_k+u_{k+1}\leq\theta_k\}\bigr].
\end{equation}
The augmented state is \(W_k=(D_k,L_k)\), with any \(L_0\in[0,1]\).

In the stationary fixed-stock system, the chain rule gives the IPA estimate
\[
\bigl[h\ind\{D_k(\theta)<\theta\}
-b\ind\{D_k(\theta)\geq\theta\}\bigr]\,[1-L_k(\theta)].
\]
The two slopes account for holding and shortage costs; \(1-L_k(\theta)\)
is the derivative of net inventory \(\theta-D_k(\theta)\).
The fixed-cost term is absent because, at a fixed positive \(\theta\),
the stockout indicator is locally constant along almost every sample path.
Yet increasing the stock level reduces the stockout probability and hence
the expected fixed stockout cost. IPA therefore omits the fixed-cost
contribution to the stock-level update and has a persistent bias even in
stationarity \citep{XuZheng23}. If \(q=0\), this estimate would be unbiased in stationarity.

To capture the fixed-cost contribution, we compare costs at two nearby
stock levels. The finite difference detects the change in the fixed stockout
cost when demand lies between them, while retaining the demand-sensitivity factor.

Assume access to the current augmented state and an exact one-step
simulator of \eqref{eq:inventory-demand} and
\eqref{eq:inventory-sensitivity}, using the same fixed model parameters
as the operating process. All operating and auxiliary Gaussian input
pairs are mutually independent and independent of the full initial pair
\((\theta_0,W_0)\), whose coordinates may be dependent.
At update \(k\), generate
\(\widetilde W_{k+1}
=(\widetilde D_{k+1},\widetilde L_{k+1})
\sim P_{\theta_k}(W_k,\cdot)\), where \(P_\theta\) is the
augmented transition kernel. For a positive stock-level increment
\(\eta_k\), form
\begin{equation}\label{eq:inventory-fd-update}
    \begin{aligned}
    \theta_{k+1}&=\Pi_\Theta(\theta_k-\alpha_k\widehat g_k), \\ 
        \widehat g_k
&=(1-\widetilde L_{k+1})
\frac{\mathcal C(\theta_k+\eta_k,\widetilde D_{k+1})
-\mathcal C(\theta_k,\widetilde D_{k+1})}{\eta_k},
    \end{aligned}
\end{equation}
where \(\alpha_k=\alpha_0 k^{-\delta}\) and \(\eta_k=\eta_0 k^{-\gamma}\) for all \(k\geq 1\) and some \(\alpha_0>0\), \(0<\delta\leq1\), \(\eta_0>0\), and \(\gamma>0\).

Both costs use the same simulated demand. The fixed cost contributes
\(-q/\eta_k\) to the quotient when
\(\theta_k<\widetilde D_{k+1}\leq\theta_k+\eta_k\), and zero otherwise.
After the update, generate the operating state
\(W_{k+1}\) from \(W_k\) under \(\theta_k\), using inputs independent
of the simulation. The updated stock level \(\theta_{k+1}\) is used
in the following operating transition.

To express the update as SA, for fixed \(\theta,w\), let
\(W'=(D',L')\sim P_\theta(w,\cdot)\) and define 
\begin{equation}\label{eq:inventory-ideal-direction}
F(\theta,w)
:=-\lim_{\eta\downarrow0}
\E\!\left[
(1-L')\frac{\mathcal C(\theta+\eta,D')
-\mathcal C(\theta,D')}{\eta}
\right].
\end{equation}
Its stationary mean \(\bar F(\theta)=-f'(\theta)\) gives a descent
direction for the stationary expected cost.
Under the convention of Section~\ref{sec:structured-extension},
\(\mathscr G_k^-\) is the information before the auxiliary draw and
\(\mathscr G_k^+\) includes that draw and the update, before the next operating transition. The operating inputs are independent of
\(\mathscr G_k^+\). Define the centered estimation error and numerical
bias by
\begin{equation}\label{eq:inventory-decomposition}
M_{k+1}:=-\bigl(\widehat g_k-\E[\widehat g_k\mid\mathscr G_k^-]\bigr)
\quad\mbox{and}\quad
B_k:=-F(\theta_k,W_k)-\E[\widehat g_k\mid\mathscr G_k^-].
\end{equation}
Here \(B_k\) is the numerical approximation bias relative to the limiting
direction \(F(\theta_k,W_k)\); it tends to zero as \(\eta_k\to0\).
Then \(-\widehat g_k=F(\theta_k,W_k)+M_{k+1}+B_k\), giving the
extended SA recursion \eqref{eq:extension-recursion}.

\subsection{Performance Guarantees}
\label{subsec:inventory-guarantees}

We first state conditions and construct a projection interval containing
a unique interior minimizer of the stationary expected cost.
Define the stockout probability
\(s(\theta):=\pr(D_\infty(\theta)>\theta)\) and the stationary
marginal functions, 
\begin{equation}\label{eq:inventory-marginal-functions}
g_{\mathrm{h}}(\theta):=\frac{d}{d\theta}\E(\theta-D_\infty(\theta))^+,\quad
g_{\mathrm{b}}(\theta):=-\frac{d}{d\theta}\E(D_\infty(\theta)-\theta)^+,\quad\mbox{and}\quad
g_{\mathrm{s}}(\theta):=-s'(\theta),
\end{equation}
for \(\theta>0\).
They measure the marginal increase in holding quantity and reductions
in shortage quantity and stockout probability. Their values at zero
are right limits, and
\(f'=h g_{\mathrm{h}}-b g_{\mathrm{b}}-q g_{\mathrm{s}}\).
The cost derivative is positive for \(\theta\geq\bar\theta\), with
\(\bar\theta\) defined below, so we consider stationary points below this
bound. Write \(\phi,\Phi\) for the standard normal density and distribution
function, and \(\phi_\sigma(x):=\phi(x/\sigma)/\sigma\). Define
\begin{gather}
\bar\theta:=m+\frac{\sigma}{1-a}
\sqrt{2\log\left(1+\frac bh+
                    \frac{2q\phi_\sigma(0)}h\right)}, \label{eq:inventory-upper-stock} \\ 
\Theta_{\det}:=\left\{\theta\in(0,\bar\theta):
\det\begin{pmatrix}
g_{\mathrm{h}}(\theta)&g_{\mathrm{b}}(\theta)&g_{\mathrm{s}}(\theta)\\
g_{\mathrm{h}}'(\theta)&g_{\mathrm{b}}'(\theta)&g_{\mathrm{s}}'(\theta)\\
g_{\mathrm{h}}''(\theta)&g_{\mathrm{b}}''(\theta)&g_{\mathrm{s}}''(\theta)
\end{pmatrix}=0\right\}. \label{eq:inventory-critical-levels}
\end{gather}
We impose the two conditions
\begin{equation}\label{eq:inventory-stability-conditions}
\begin{aligned}
b g_{\mathrm{b}}(0)+q g_{\mathrm{s}}(0)&>h g_{\mathrm{h}}(0),\\
h g_{\mathrm{h}}(\theta)-b g_{\mathrm{b}}(\theta)-q g_{\mathrm{s}}(\theta)&\ne0,
\qquad \theta\in\Theta_{\det}.
\end{aligned}
\end{equation}
The first condition says that, near zero stock, the savings from reducing
shortages and stockouts outweigh the added holding cost. The second
excludes a marginal-cost balance at the levels in \(\Theta_{\det}\),
ruling out an interior minimum with zero curvature. These conditions
allow the interval construction below.

To construct the projection interval, for any \(\theta\geq 0\), define
\[
G(\theta):=
\frac{h g_{\mathrm{h}}(\theta)-q g_{\mathrm{s}}(\theta)}
     {g_{\mathrm{b}}(\theta)},
\]
where \(g_{\mathrm{b}}(\theta)>0\). Since \(f'=g_{\mathrm{b}}(G-b)\),
the cost decreases where \(G<b\) and increases where \(G>b\).
Among the maximal open intervals \((\theta_-,\theta_+)\subset(0,\bar\theta)\)
on which \(G'>0\) and \(G(\theta_-)<b<G(\theta_+)\), choose the leftmost;
endpoint values are defined by continuity.
Define the projection endpoints as the unique solutions in this interval of
\begin{equation}\label{eq:ec-inventory-projection-levels}
G(\theta_{\min})=\frac{G(\theta_-)+b}{2}\quad\mbox{and}\quad
G(\theta_{\max})=\frac{b+G(\theta_+)}{2}.
\end{equation}
Let \(\Theta=[\theta_{\min},\theta_{\max}]\) and \(\theta^\star\) be the unique solution of \(G(\theta)=b\) in
\(\Theta\).
The conditions \eqref{eq:inventory-stability-conditions} make this construction
well-defined, with \(\theta^\star\) the unique interior minimizer on \(\Theta\).
On this interval, the mean field \(-f'\) satisfies QSM with coefficient
\begin{equation}\label{eq:ec-inventory-qsm-coefficient}
\zeta:=\left[\min_{t\in\Theta}G'(t)\right](1-a)
\left[1-\Phi\!\left(\frac{\theta_{\max}}{\sigma}\right)\right]
\left[1-\Phi\!\left(
\frac{(1-a)(\theta_{\max}-m)}{\sigma}\right)\right]>0.
\end{equation}
Together with the stated initialization and schedules, the model and
estimator satisfy the remaining assumptions of the extended SA results.
Section~\ref{sec:EC-inventory-proofs} gives the proof
of the following theorem.

\begin{theorem}
\label{thm:inventory-accuracy}
Consider the inventory model and update
\eqref{eq:inventory-demand}--\eqref{eq:inventory-fd-update}. Suppose
\(\theta_0\in\Theta\),
\((D_0,L_0)\in[0,\infty)\times[0,1]\),
\(\E[U_2(W_0)]<\infty\), and
\eqref{eq:inventory-stability-conditions} holds. When
\(\delta=1\), also suppose \(2\zeta\alpha_0>1\).
\begin{enumerate}[label=\textnormal{(\roman*)}]
\item
If \(0<\gamma<\delta\), there are positive constants \(C,K\)
such that, for all \(k\geq K\),
\begin{equation}\label{eq:inventory-mse}
\E[|\theta_k-\theta^\star|^2\mid\mathscr G_0^+]
\leq C U_2(W_0)
\bigl[(1+\log k)k^{-\delta+\gamma}+k^{-2\gamma}\bigr].
\end{equation}
\item
Suppose \(1/2<\delta\leq1\) and \(0<\gamma<2\delta-1\).
For any \(c>0\) and
\(0\leq\beta<\min\{\gamma,\delta-\gamma/2-1/2\}\), there are
positive constants \(C,K\) such that, for all \(k_0\geq K\),
\begin{equation}\label{eq:inventory-tube}
\pr\Bigl(
|\theta_k-\theta^\star|\leq c k^{-\beta}
\text{ for all }k\geq k_0
\,\Bigm|\,\mathscr G_0^+
\Bigr)
\geq1-C U_2(W_0)(1+\log k_0)^2
k_0^{-(2\delta-2\beta-\gamma-1)}.
\end{equation}
\end{enumerate}
\end{theorem}

Part (i) bounds the stock-level MSE at update \(k\).
Part (ii) controls all later decisions together: it bounds the probability
that any stock level from \(k_0\) onward differs from \(\theta^\star\)
by more than \(c k^{-\beta}\). For \(\beta>0\), this tolerance decreases
over time.

Reducing \(\eta_k\) tightens the approximation-bias bound, but dividing
the jump in the fixed stockout cost by a smaller increment can produce larger gradient
estimates. The two MSE terms reflect random estimation error and numerical
bias. Among the power schedules in these bounds, balancing their polynomial
exponents gives \(\gamma=\delta/3\), with MSE decay exponent \(2\delta/3\)
and the logarithmic factor in \eqref{eq:inventory-mse}.

For part (ii), with \(1/2<\delta\leq1\), the restriction \(\beta<\gamma\)
makes the bias bound decrease faster than the tolerance, while
\(\beta<\delta-\gamma/2-1/2\) makes the bound on the probability of a later
exit decrease. Increasing \(\gamma\) relaxes the first restriction but
tightens the second. The largest permitted upper limit for \(\beta\)
is attained at \(\gamma=(2\delta-1)/3\), allowing any
\(0\leq\beta<(2\delta-1)/3\). This choice decreases \(\eta_k\) more slowly
than the MSE choice when \(1/2<\delta<1\); at \(\delta=1\), both give
\(\gamma=1/3\).

\section{Conclusion}\label{sec:conclusions}

This paper establishes a shrinking-tube concentration guarantee for the entire
future trajectory of projected SA under adaptive
Markovian feedback.  It bounds the probability that, from a chosen time onward,
every iterate remains within a tolerance that tightens over time, even when the
iterate changes the transition kernel and the Markov state is potentially
unbounded. The lower bounds show that the polynomial exponent in the
exit-probability upper bound cannot be improved in general under finite
second moments.  The proof combines backward kernel replacement, a finite-time MSE
bound, and a blockwise maximal first-exit argument.

Practical applications often require SA updates to incorporate additional
simulation or numerical approximation steps. Our theoretical framework
additionally covers the resulting martingale-difference noise and predictable bias.
As demonstrated through an inventory control problem with stockout-dependent
demand feedback and fixed stockout costs, reducing approximation bias can
magnify individual gradient estimates. The resulting bounds show that tuning
for MSE at a chosen iteration can differ from tuning to
control the probability of any later deviation beyond a decreasing tolerance.
Approximation accuracy and step sizes should therefore be selected together
according to the form of accuracy sought.

One direction for future research is to extend shrinking-tube guarantees to
multi-timescale SA, where coupled recursions use different step-size scales,
as in actor--critic reinforcement learning \citep{ZengDoanRomberg24} and
stochastic bilevel optimization \citep{HongWaiWangYang23}.
This requires quantifying how tracking errors in the faster recursion affect
the shrinking tolerances for the slower recursion.
A second direction is to determine the best trajectory guarantees achievable
by robust SA algorithms under finite second moments. Recent results on
clipped SGD provide high-probability accuracy guarantees
\citep{ChezhegovParlettaPaudiceGorbunov26}.
Our lower bounds motivate studying whether clipping or robust aggregation
can improve the tradeoff between shrinking accuracy tolerances and confidence
under adaptive Markov sampling. This requires controlling the bias introduced
by robust updates and establishing matching lower bounds under a common
sampling budget.

\bibliographystyle{informs2014}
\bibliography{RL}

\ECSwitch

\ECHead{Supplementary Materials}

\numberwithin{equation}{section}

Throughout the e-companion, \(C_0,C_1,\ldots\) denote nonnegative
constants.  The symbols \(K_0,K_1,\ldots\) and
\(R_0,R_1,\ldots\) denote integer thresholds for iteration and block
indices, respectively.  These quantities are local to each proof.  They may
depend on fixed model and result parameters and are uniform over running
indices, sample paths, and realizations of the initial variables.

\section{Preliminaries}
\label{sec:EC-preliminaries}

\subsection{\(U_1\)-Drift Condition}
\label{subsec:U1-drift-condition}

\begin{lemma}
\label{lemma:derived-U1-drift}
Suppose the $U_2$-drift condition \eqref{eq:drift-condition} holds.
Define \(\lambda_1:=
    \frac{\lambda+\sqrt{2\lambda}}
    {1+\sqrt{2\lambda}}\),  \(b_1:=1+\sqrt b\), \(\lambda_2:=\lambda\), and \(b_2:=b\).
Then \(0\leq\lambda_p<1\) for \(p\in\{1,2\}\), and, for all
\(\theta\in\Theta\), \(w\in\W\), and \(p\in\{1,2\}\),
\begin{equation*}
    (P_\theta U_p)(w)
    \leq\lambda_pU_p(w)+b_p\ind_{\W_0}(w),
    \qquad 
\end{equation*}
\end{lemma}

\proof{Proof of Lemma~\ref{lemma:derived-U1-drift}.}
The case \(p=2\) is precisely the $U_2$-drift condition \eqref{eq:drift-condition} in
Assumption~\ref{assump:ergodic}\ref{part:ergodic-reference}. 

Consider \(p=1\). Fix $\theta\in\Theta$ and $w\in\mathcal W$. Let $Y$ have distribution
$P_\theta(w,\cdot)$ and write $x:=\|w\|$. By the definition of \(U_1\) and \(U_2\),
\begin{equation}\label{eq:EC-derived-U1-CS}
(P_\theta U_1)(w)=1+\mathbb E\|Y\|
\leq1+\sqrt{\mathbb E\|Y\|^2}
=1+\sqrt{(P_\theta U_2)(w)-1}.
\end{equation}

First suppose that $0<\lambda<1$ and $w\notin\mathcal W_0$.
The $U_2$-drift condition implies
\[
1\leq (P_\theta U_2)(w)
\leq
\lambda(1+x^2).
\]
Hence,
\[
(P_\theta U_1)(w)/U_1(w)
\leq
\frac{1+\sqrt{\lambda(1+x^2)-1}}{1+x}.
\]
Consider the function
\[
g_\lambda(x)
:=
\frac{1+\sqrt{\lambda(1+x^2)-1}}{1+x},
\]
defined on the set where
$\lambda(1+x^2)\geq1$. Direct differentiation shows that its unique
maximizer is
\(
x_\lambda^\star=1+\sqrt{2/\lambda},
\)
and therefore
\begin{equation*}
\sup_{\substack{x\geq0\\\lambda(1+x^2)\geq1}}
g_\lambda(x)
=
\frac{\lambda+\sqrt{2\lambda}}
{1+\sqrt{2\lambda}}
=
\lambda_1 .
\end{equation*}
Thus,
\(
(P_\theta U_1)(w)
\leq
\lambda_1 U_1(w)\),
for all 
\(w\notin\mathcal W_0\).
Moreover,
\(
1-\lambda_1
=
\frac{1-\lambda}{1+\sqrt{2\lambda}}>0,
\)
which implies $\lambda_1<1$. Since $\lambda\geq0$, we also have
$\lambda_1\geq0$.

Now consider $w\in\mathcal W_0$. By
\eqref{eq:EC-derived-U1-CS} and the primitive drift condition,
\[
(P_\theta U_1)(w)\leq1+\sqrt{\lambda(1+x^2)+b}
\leq1+\sqrt{\lambda}(1+x)+\sqrt b
\leq\lambda_1 U_1(w)+1+\sqrt b .
\]
The last inequality follows from
$\lambda_1\geq\sqrt{\lambda}$, which can be verified directly:
\[
\lambda_1-\sqrt{\lambda}
=
\frac{
\sqrt{\lambda}(\sqrt2-1)(1-\sqrt{\lambda})
}
{1+\sqrt{2\lambda}}
\geq0 .
\]

It remains to treat the boundary case $\lambda=0$. In this case, if
$w\notin\mathcal W_0$, the drift condition would imply
\(
(P_\theta U_2)(w)\leq0,
\)
which contradicts $U_2\geq1$. Hence
$\mathcal W_0=\mathcal W$. The drift condition implies
\(
(P_\theta U_2)(w)\leq b,
\)
so necessarily \(b\geq1\).  Therefore, by
\eqref{eq:EC-derived-U1-CS},
\[
(P_\theta U_1)(w)
\leq
1+\sqrt{b-1}
\leq
1+\sqrt b
=
b_1 .
\]

Combining the cases establishes the $U_1$-drift condition. 
\Halmos\endproof

\subsection{Uniform Reference-Chain Moment and Geometric-Ergodicity Bounds}
\label{subsec:ec-proof-exp-ergodic}

\begin{lemma}
\label{lemma:exp-ergodic}
Suppose Assumption~\ref{assump:ergodic}\ref{part:ergodic-reference} holds.
For each \(\theta\in\Theta\), let \(\nu_\theta\) denote the unique stationary
distribution of \(P_\theta\).  
Then the following conclusions hold:
\begin{enumerate}[
label=(\roman*),
ref=\textnormal{(\roman*)},
noitemsep,
nolistsep
]
    \item\label{part:exp-ergodic-reference-moment}
    for all \(\theta\in\Theta\), \(w\in\W\), \(k\geq0\), and
    \(p\in\{1,2\}\),
    \begin{equation}\label{eq:reference-kernel-moment}
    P_\theta^kU_p(w)
    \leq
    \lambda_p^kU_p(w)
    +\frac{b_p}{1-\lambda_p}(1-\lambda_p^k).
    \end{equation}
    \item\label{part:exp-ergodic-invariant-moment}
    for each \(p\in\{1,2\}\),
    \begin{equation}\label{eq:uniform-invariant-moment}
    \sup_{\theta\in\Theta}\nu_\theta(U_p)
    <\infty;
    \end{equation}
    \item\label{part:exp-ergodic-uniform-mixing}
    if, in addition, Assumptions~\ref{assump:compact} and
    \ref{assump:kernel-Lipschitz} hold, then there are positive constants \(C\) and \(\rho\in(0,1)\), independent of \(\theta\), such that for all \(\theta\in\Theta\), \(w\in\W\), \(k\geq0\), and Borel
    measurable functions \(f\) satisfying \(|f|\leq U_1\), we have
    \[
        \bigl|(P_\theta^k f)(w)-\nu_\theta(f)\bigr|
        \leq C\rho^kU_1(w).
    \]
\end{enumerate}
\end{lemma}

\proof{Proof of Lemma~\ref{lemma:exp-ergodic}.} 

For part~\ref{part:exp-ergodic-reference-moment},
\eqref{eq:reference-kernel-moment} is immediate when \(k=0\). Suppose it
holds for some \(k\geq0\).  Define \(\mathbf 1\) as the constant-one
function on \(\W\).  Using Lemma~\ref{lemma:derived-U1-drift} and
\(P_\theta^k\mathbf 1=\mathbf 1\),
\begin{align*}
P_\theta^{k+1}U_p(w)
&=P_\theta^k(P_\theta U_p)(w)
\leq \lambda_p P_\theta^kU_p(w)
    +b_p(P_\theta^k\ind_{\W_0})(w)\\
&\leq \lambda_p P_\theta^kU_p(w)+b_p\\
&\leq \lambda_p^{k+1}U_p(w)+b_p\sum_{j=0}^{k}\lambda_p^j
=\lambda_p^{k+1}U_p(w)
+\frac{b_p}{1-\lambda_p}(1-\lambda_p^{k+1}).
\end{align*}
Thus \eqref{eq:reference-kernel-moment} holds for each \(k\geq0\), proving
part~\ref{part:exp-ergodic-reference-moment}.

For part~\ref{part:exp-ergodic-invariant-moment}, \citet[Theorem 16.0.1]{MeynTweedie09} implies
\(\nu_\theta(U_2)<\infty\), and hence
\(\nu_\theta(U_p)<\infty\) for \(p\in\{1,2\}\).
By Lemma~\ref{lemma:derived-U1-drift} and invariance of \(\nu_\theta\),
\(
(1-\lambda_p)\nu_\theta(U_p)\leq b_p.
\)
Because \(\lambda_p\) and \(b_p\) are uniform in \(\theta\),
\eqref{eq:uniform-invariant-moment} follows.

For part~\ref{part:exp-ergodic-uniform-mixing}, let \(f\) be a Borel
measurable scalar function and write
\[
\lVert f\rVert_{U_1}
:=\sup_{w\in\W}|f(w)|/U_1(w),
\qquad
\mathsf B:=\{f:\lVert f\rVert_{U_1}<\infty\}.
\]
The map \(f\mapsto f/U_1\) is an isometry from \(\mathsf B\) onto the
bounded Borel functions, so \(\mathsf B\) is a Banach space.
By 
\citet[Theorem 16.0.1]{MeynTweedie09}, there are positive constants \(C_0(\theta)\) and \(\rho(\theta)\in(0,1)\) such that for all \(k\geq 0\),
\begin{equation}\label{eq:fixed-kernel-geometric}
    \lVert P_\theta^k f-\nu_\theta(f)\rVert_{U_1}
    \leq C_0(\theta)\rho(\theta)^k
    \lVert f\rVert_{U_1}.
\end{equation}

For a bounded linear operator \(\mathcal L\) on a normed space
\((\mathsf X,\lVert\cdot\rVert_{\mathsf X})\), write
\(
\opnorm{\mathcal L}_{\mathsf X}
:=\sup_{\lVert f\rVert_{\mathsf X}\leq1}
\lVert \mathcal Lf\rVert_{\mathsf X}.
\)
Let
\(
\mathsf C:=\{t\mathbf 1:t\in\R\}\), and
\(\mathsf Q:=\mathsf B/\mathsf C.
\)
The space \(\mathsf C\) is closed in \(\mathsf B\), so
\(\mathsf Q\) is a Banach space.  Equip it with the quotient norm
\(\lVert[f]\rVert_{\mathsf Q}:=\inf_{t\in\R}\lVert f-t\rVert_{U_1}\).
Each Markov kernel preserves constants and hence induces a bounded operator
on \(\mathsf Q\).  We use the same notation \(P_\theta\) for this
operator, with \(P_\theta[f]:=[P_\theta f]\); the norm subscript identifies
the space on which the operator acts.
The drift condition \eqref{eq:drift-condition} implies
\[
    |P_\theta f|\leq \lVert f\rVert_{U_1}P_\theta U_1
    \leq(\lambda_1+b_1)\lVert f\rVert_{U_1}U_1.
\]
Thus we may fix \(M\geq1\) such that for all \(\theta\in\Theta\),
\begin{equation}\label{eq:uniform-operator-bound}
    \opnorm{P_\theta}_{U_1}\leq M 
    \quad\mbox{and}\quad \opnorm{P_\theta}_{\mathsf Q}\leq M.
\end{equation}

For any constant \(t\),
\(P_\theta^k(f-t)-\nu_\theta(f-t)=P_\theta^k f-\nu_\theta(f)\).
Applying \eqref{eq:fixed-kernel-geometric} to \(f-t\) and taking the infimum over
\(t\) yields
\begin{equation}\label{eq:fixed-quotient-contraction}
    \opnorm{P_\theta^k}_{\mathsf Q}
    \leq C_0(\theta)\rho(\theta)^k.
\end{equation}
For a fixed \(\theta^{\circ}\in\Theta\), choose an integer
\(n(\theta^{\circ})\geq1\) so that the right-hand side of
\eqref{eq:fixed-quotient-contraction} is at most \(1/4\).

Assumption~\ref{assump:kernel-Lipschitz} and homogeneity imply
\begin{equation*}
\opnorm{P_\eta-P_{\theta^{\circ}}}_{\mathsf Q}
\leq\opnorm{P_\eta-P_{\theta^{\circ}}}_{U_1}
\leq L_P\lVert\eta-\theta^{\circ}\rVert.
\end{equation*}
For bounded operators \(A\) and \(B\) and every \(n\geq1\),
\[
A^n-B^n=\sum_{i=0}^{n-1}A^i(A-B)B^{n-i-1}.
\]
Applying this identity with \(A=P_\eta\) and
\(B=P_{\theta^{\circ}}\), and using
\eqref{eq:uniform-operator-bound}, gives
\begin{align*}
\opnorm{P_\eta^n-
P_{\theta^{\circ}}^n}_{\mathsf Q}
&\leq \sum_{i=0}^{n-1}
\opnorm{P_\eta}_{\mathsf Q}^{i}
\opnorm{P_\eta-P_{\theta^{\circ}}}_{\mathsf Q}
\opnorm{P_{\theta^{\circ}}}_{\mathsf Q}^{n-i-1}
\\
&\leq n M^{n-1}L_P\lVert\eta-\theta^{\circ}\rVert.
\end{align*}
Choose a relative neighborhood \(\mathcal N(\theta^{\circ})\) of
\(\theta^{\circ}\) in \(\Theta\) such that for all \(\eta\in\mathcal N(\theta^{\circ})\),
\[
n(\theta^{\circ})M^{n(\theta^{\circ})-1}L_P
\lVert\eta-\theta^{\circ}\rVert\leq\frac14.
\]
The triangle inequality, the choice of \(n(\theta^{\circ})\), and
the preceding perturbation bound now give
\begin{align}
\opnorm{P_\eta^{n(\theta^{\circ})}}_{\mathsf Q}
&\leq
\opnorm{P_{\theta^{\circ}}^{n(\theta^{\circ})}}_{\mathsf Q}
+\opnorm{P_\eta^{n(\theta^{\circ})}-
P_{\theta^{\circ}}^{n(\theta^{\circ})}}_{\mathsf Q}
\leq\frac14+\frac14=\frac12,
\label{eq:local-block-contraction}
\end{align}
for all \(\eta\in\mathcal N(\theta^{\circ})\).
This proves the required local block contraction.

By compactness of \(\Theta\), choose a finite subcover
\(\{\mathcal N_i:1\leq i\leq m\}\), with corresponding block lengths
\(n_i\).  Set \(N:=\max_i n_i\).  Given any \(\eta\in\Theta\), choose a
covering neighborhood \(\mathcal N_i\).  If \(k=a n_i+h\), where
\(0\leq h<n_i\), then \eqref{eq:uniform-operator-bound} and
\eqref{eq:local-block-contraction} imply
\begin{align}
    \opnorm{P_\eta^k}_{\mathsf Q}
    &\leq M^h2^{-a}
    \leq M^{N-1}2^{-\lfloor k/N\rfloor}
    \leq 2M^{N-1}\rho^k,
    \label{eq:global-quotient-contraction}
\end{align}
where we used
\(a=\lfloor k/n_i\rfloor\geq\lfloor k/N\rfloor\),
\(h<n_i\leq N\), and
\(\rho:=2^{-1/N}\in(0,1)\).

It remains to return from the quotient norm to centering at the invariant measure.
For any \(t\in\R\), invariance and \eqref{eq:uniform-invariant-moment} imply
\begin{align*}
\lVert P_\eta^k f-\nu_\eta(f)\rVert_{U_1}
&\leq \lVert P_\eta^k f-t\rVert_{U_1}
  +|\nu_\eta(P_\eta^k f-t)|\\
&\leq(1+\nu_\eta(U_1))\lVert P_\eta^k f-t\rVert_{U_1}
\leq \bigl(1+\sup_{\theta\in\Theta}\nu_\theta(U_1)\bigr)
\lVert P_\eta^k f-t\rVert_{U_1}.
\end{align*}
Taking the infimum over \(t\), applying
\eqref{eq:global-quotient-contraction}, and using
\(\lVert[f]\rVert_{\mathsf Q}\leq\lVert f\rVert_{U_1}\), define
\[
C:=2M^{N-1}
\bigl(1+\sup_{\theta\in\Theta}\nu_\theta(U_1)\bigr).
\]
Then
\begin{equation}\label{eq:p-specific-uniform-bound}
\lVert P_\eta^k f-\nu_\eta(f)\rVert_{U_1}
\leq C\rho^k\lVert f\rVert_{U_1},
\end{equation}
where the fixed prefactor in \eqref{eq:global-quotient-contraction} is included
in \(C\).  Once the finite cover is fixed, \(N\), \(\rho\), and \(C\)
are determined by the fixed kernel family,
the compact parameter set, and the primitive drift and kernel-continuity
constants, and are uniform in \(\eta\).  Since
\(|f|\leq U_1\) implies
\(\lVert f\rVert_{U_1}\leq1\),
\eqref{eq:p-specific-uniform-bound} implies
\eqref{eq:uniform-geometric-ergodicity}, proving
part~\ref{part:exp-ergodic-uniform-mixing}.
\Halmos\endproof

\subsection{Adaptive Moment Control}

\begin{lemma}
    \label{lemma:adaptive-moment-control}
Suppose Assumption~\ref{assump:ergodic}\ref{part:ergodic-reference} holds.
For each \(p\in\{1,2\}\) and all \(k,s\geq0\),
\begin{equation*}
\E[U_p(W_{k+s})\mid\F_k]
\leq \lambda_p^sU_p(W_k)
+\frac{b_p}{1-\lambda_p}(1-\lambda_p^s).
\end{equation*}
\end{lemma}

\proof{Proof of Lemma~\ref{lemma:adaptive-moment-control}.}
By Lemma~\ref{lemma:derived-U1-drift}, for each \(p\in\{1,2\}\) and all \(k\geq 0\),
\[
    \E[U_p(W_{k+1})\mid\F_k]
    \leq \lambda_pU_p(W_k)+b_p.
\]
Iteration yields, for \(s\geq0\),
\[
\E[U_p(W_{k+s})\mid\F_k]
\leq
\lambda_p^sU_p(W_k)+b_p\sum_{j=0}^{s-1}\lambda_p^j
=
\lambda_p^sU_p(W_k)
+\frac{b_p}{1-\lambda_p}(1-\lambda_p^s). \Halmos
\]
\endproof

\subsection{Stationary Distribution Continuity and Vector Consequences}

\begin{lemma}
    \label{lemma:stationary-Lipschitz}
Suppose Assumptions~\ref{assump:compact}, \ref{assump:ergodic}\ref{part:ergodic-reference}, and
\ref{assump:kernel-Lipschitz} hold. Then there is a positive constant \(C\) such that, for all \(\theta,\theta'\in\Theta\),
\begin{equation*}
    \sup_{\lVert f\rVert_{U_1}\leq1}
    |\nu_\theta(f)-\nu_{\theta'}(f)|
    \leq C\lVert\theta-\theta'\rVert.
\end{equation*}
\end{lemma}

\proof{Proof of Lemma~\ref{lemma:stationary-Lipschitz}.}
Fix \(\theta,\theta'\in\Theta\) and
\(f\in\mathsf B\) with \(\lVert f\rVert_{U_1}\leq1\).
Let \(C_0\) be the constant in
Lemma~\ref{lemma:exp-ergodic}\ref{part:exp-ergodic-uniform-mixing}.
Define
\[
    h_{\theta',f}
    :=\sum_{n=0}^{\infty}
    \bigl(P_{\theta'}^n f-\nu_{\theta'}(f)\bigr).
\]
By Lemma~\ref{lemma:exp-ergodic}\ref{part:exp-ergodic-uniform-mixing},
\[
\sum_{n=0}^{\infty}
\Bigl\lVert P_{\theta'}^n f-\nu_{\theta'}(f)\Bigr\rVert_{U_1}
\leq \sum_{n=0}^{\infty}C_0\rho^n
=C_0/(1-\rho).
\]
Hence the series converges absolutely in the Banach space
\(\mathsf B\), and
\begin{equation}\label{eq:poisson-series-norm}
    \lVert h_{\theta',f}\rVert_{U_1}
    \leq C_0/(1-\rho).
\end{equation}

For \(N\geq0\), let
\[
    h_{\theta',f}^{(N)}
    :=\sum_{n=0}^{N}\bigl(P_{\theta'}^n f-\nu_{\theta'}(f)\bigr).
\]
A finite telescoping calculation yields the exact identity
\begin{equation}\label{eq:poisson-partial-sum-identity}
    (I-P_{\theta'})h_{\theta',f}^{(N)}
    =f-P_{\theta'}^{N+1}f.
\end{equation}
Equation~\eqref{eq:uniform-operator-bound} shows that \(P_{\theta'}\) is a
bounded operator on \(\mathsf B\).  Thus it is continuous in the
\(U_1\)-norm and may be passed through the norm-convergent series.  Moreover,
Lemma~\ref{lemma:exp-ergodic}\ref{part:exp-ergodic-uniform-mixing} shows that
\(P_{\theta'}^{N+1}f\to\nu_{\theta'}(f)\) in \(\mathsf B\).
Passing to the limit in \eqref{eq:poisson-partial-sum-identity} yields
\((I-P_{\theta'})h_{\theta',f}=f-\nu_{\theta'}(f)\).

The \(U_1\) bounds make
\(h_{\theta',f}\), \(P_\theta h_{\theta',f}\), and
\(P_{\theta'}h_{\theta',f}\) all \(\nu_\theta\)-integrable.  Invariance
therefore implies
\(\nu_\theta(h_{\theta',f})=\nu_\theta(P_\theta h_{\theta',f})\), and
\[
\nu_\theta(f)-\nu_{\theta'}(f)
=\nu_\theta\!\bigl[(I-P_{\theta'})h_{\theta',f}\bigr]
=\nu_\theta\!\bigl[(P_\theta-P_{\theta'})h_{\theta',f}\bigr].
\]

Assumption~\ref{assump:kernel-Lipschitz}, applied by homogeneity to
\(h_{\theta',f}\), together with
Lemma~\ref{lemma:exp-ergodic}\ref{part:exp-ergodic-invariant-moment} and
\eqref{eq:poisson-series-norm}, allows
us to set
\[
C:=\frac{L_PC_0}{1-\rho}
\sup_{\eta\in\Theta}\nu_\eta(U_1).
\]
Then
\[
|\nu_\theta(f)-\nu_{\theta'}(f)|
\leq L_P\lVert h_{\theta',f}\rVert_{U_1}
\nu_\theta(U_1)\lVert\theta-\theta'\rVert
\leq C\lVert\theta-\theta'\rVert.
\]
Taking the supremum over \(\lVert f\rVert_{U_1}\leq1\) proves the lemma.
\Halmos\endproof

\begin{lemma}
    \label{lemma:vector-duality}
Fix \(d_H\in\mathbb N\).  Let
\(H:\W\to\R^{d_H}\) be Borel measurable and satisfy
\(\lVert H(w)\rVert\leq U_1(w)\).

\textup{(i)} Under Assumption~\ref{assump:kernel-Lipschitz},
\begin{equation}
\lVert(P_\theta H)(w)-(P_{\theta'}H)(w)\rVert
\leq L_PU_1(w)\lVert\theta-\theta'\rVert.
\label{eq:vector-kernel-duality}
\end{equation}

\textup{(ii)} If, in addition, Assumptions~\ref{assump:compact} and
\ref{assump:ergodic}\ref{part:ergodic-reference} hold, then there is a positive constant \(C\) such that
\begin{equation}
\lVert\nu_\theta(H)-\nu_{\theta'}(H)\rVert
\leq C\lVert\theta-\theta'\rVert.
\label{eq:vector-stationary-duality}
\end{equation}
\end{lemma}

\proof{Proof of Lemma~\ref{lemma:vector-duality}.}

First fix \(\theta,\theta'\in\Theta\) and \(w\in\W\).  For any
\(u\in\R^{d_H}\) with \(\lVert u\rVert\leq1\), define 
\(
    h_u(w):=u^\intercal H(w).
\)
Then
\(
    |h_u(w)|
    \leq
    U_1(w)
\), for all \(w\in\W\).
Hence Assumption~\ref{assump:kernel-Lipschitz} implies
\[
\biggl|
(P_\theta h_u)(w)-(P_{\theta'}h_u)(w)
\biggr|
\leq
L_PU_1(w)\lVert\theta-\theta'\rVert .
\]
By linearity, for each \(\lVert u\rVert\leq1\),
\[
\biggl|
u^\intercal \bigl((P_\theta H)(w)-(P_{\theta'}H)(w)\bigr)
\biggr|
\leq
L_PU_1(w)\lVert\theta-\theta'\rVert .
\]
The right-hand side is independent of \(u\).  Taking the supremum over
\(\lVert u\rVert\leq1\) and using the Euclidean duality identity
\(
    \lVert x\rVert
    =
    \sup_{\lVert u\rVert\leq1}|u^\intercal x|
\)
proves \eqref{eq:vector-kernel-duality}.

For part~\textup{(ii)},
let \(C_0\) be the constant in
Lemma~\ref{lemma:stationary-Lipschitz}.  Then
\(
\biggl|
\nu_\theta(h_u)-\nu_{\theta'}(h_u)
\biggr|
\leq
C_0\lVert\theta-\theta'\rVert .
\)
Taking \(C:=C_0\), the same scalarization and Euclidean-duality argument
proves \eqref{eq:vector-stationary-duality}.
\Halmos\endproof

\section{Backward Kernel Replacement}
\label{sec:EC-backward-telescoping}

\subsection{Proof of Lemma~\ref{lemma:kernel-telescoping}}
\label{subsec:ec-canonical-telescoping-proof}

\proof{Proof of Lemma~\ref{lemma:kernel-telescoping}.}
The case of \(k=1\) is trivial, with both sides of the inequality \eqref{eq:kernel-telescoping-bound} being zero. 
Assume henceforth that \(k\geq2\), and for each \(s=1,\ldots,k\), define
\(\Lambda_s(f):=\E[(P_{\theta_0}^{k-s}f)(W_s)\mid\F_0]\).
Lemmas~\ref{lemma:exp-ergodic}\ref{part:exp-ergodic-reference-moment} and
\ref{lemma:adaptive-moment-control} ensure that these conditional
expectations are finite.

Note that 
\begin{align}
   \bigl|\E[f(W_k)\mid\F_0] -  (P_{\theta_0}^kf)(W_0) \bigr| 
   = \bigl| \Lambda_k(f) - \Lambda_1(f) \bigr| 
   \leq & \sum_{s=1}^{k-1}  |\Lambda_{s+1}(f)-\Lambda_s(f)|.  \label{eq:Lambda-endpoints}
\end{align}

For each \(s\in\{1,\ldots,k-1\}\), conditioning through \(\F_s\) yields
\begin{align}
\Lambda_{s+1}(f)-\Lambda_s(f)
&=\E\biggl[
\bigl(P_{\theta_0}^{k-s-1}f\bigr)(W_{s+1})
\,\biggm|\,\F_0
\biggr]
-\E\biggl[
\bigl(P_{\theta_0}^{k-s}f\bigr)(W_s)
\,\biggm|\,\F_0
\biggr]
\nonumber\\
&=\E\biggl[
\E\biggl[
\bigl(P_{\theta_0}^{k-s-1}f\bigr)(W_{s+1})
\,\biggm|\,\F_s
\biggr]
-\bigl(P_{\theta_0}^{k-s}f\bigr)(W_s)
\,\biggm|\,\F_0
\biggr]
\nonumber\\
&=\E\biggl[
\bigl(P_{\theta_s}P_{\theta_0}^{k-s-1}f\bigr)(W_s)
-\bigl(P_{\theta_0}P_{\theta_0}^{k-s-1}f\bigr)(W_s)
\,\biggm|\,\F_0
\biggr]
\nonumber\\
&=\E\biggl[
\biggl(
\bigl(P_{\theta_s}-P_{\theta_0}\bigr)
P_{\theta_0}^{k-s-1}f
\biggr)(W_s)
\,\biggm|\,\F_0
\biggr].
\label{eq:single-replacement-identity}
\end{align}

Define 
\(
g_s:=P_{\theta_0}^{k-s-1}f-\nu_{\theta_0}(f)
\). Then, 
\[
\biggl(
\bigl(P_{\theta_s}-P_{\theta_0}\bigr)g_s
\biggr)(W_s) = \biggl(\bigl(P_{\theta_s}-P_{\theta_0}\bigr)
P_{\theta_0}^{k-s-1}f
\biggr)(W_s), 
\]
since \(\nu_{\theta_0}(f)\) is a constant and both kernels preserve constants.
Thus, it follows from 
\eqref{eq:single-replacement-identity} that
\begin{align}
|\Lambda_{s+1}(f)-\Lambda_s(f)|
= \biggl| \E\biggl[
\bigl(P_{\theta_s}-P_{\theta_0}\bigr)g_s(W_s)
\,\biggm|\,\F_0
\biggr] \biggr| 
 \leq \E\biggl[
\biggl|\bigl(P_{\theta_s}-P_{\theta_0}\bigr)
g_s(W_s)\biggr|
\,\biggm|\,\F_0
\biggr].\label{eq:single-replacement-centered}
\end{align}

Let \(C_0\) be the constant in
Lemma~\ref{lemma:exp-ergodic}\ref{part:exp-ergodic-uniform-mixing}. Then,
\(
|g_s(w)| \leq C_0\rho^{k-s-1} U_1(w)
\) for all \(w\in\W\). 
Set \(C:=L_PC_0\).  Thus, by homogeneity,
Assumption~\ref{assump:kernel-Lipschitz} implies
\begin{align*}
\biggl|\bigl(P_{\theta_s}-P_{\theta_0}\bigr)
g_s(w) \biggr| \leq C\rho^{k-s-1} U_1(w)\lVert\theta_s-\theta_0\rVert,
\end{align*}
which, together with \eqref{eq:single-replacement-centered}, yields 
\begin{align}
|\Lambda_{s+1}(f)-\Lambda_s(f)|
&\leq C\rho^{k-s-1}
\E\biggl[
U_1(W_s)\lVert\theta_s-\theta_0\rVert
\,\biggm|\,\F_0
\biggr].
\label{eq:single-replacement-bound}
\end{align}

Finally,
combining 
\eqref{eq:Lambda-endpoints} and \eqref{eq:single-replacement-bound} completes
the proof.
\Halmos\endproof

\subsection{Lagged Update-Field Comparison}
\label{subsec:ec-lagged-update-comparison}

\begin{corollary}
\label{corollary:lagged-update-field-comparison}
Suppose Assumptions~\ref{assump:compact}--\ref{assump:growth} hold.
There is a positive constant \(C\) such that, for all \(t\geq0\) and
\(\ell\geq1\),
\begin{align*}
\biggl\lVert
\E[F(\theta_t,W_{t+\ell})\mid\F_t]
-\bar F(\theta_t)
\biggr\rVert
\leq
C\rho^\ell U_1(W_t)
+C
\sum_{s=1}^{\ell-1}
\rho^{\ell-s-1}
\E\biggl[
U_1(W_{t+s})
\lVert\theta_{t+s}-\theta_t\rVert
\,\biggm|\,\F_t
\biggr].
\end{align*}
\end{corollary}

\proof{Proof of Corollary~\ref{corollary:lagged-update-field-comparison}.}
Let \(\mathsf D\) be a countable dense subset of
\(\{u\in\R^d:\lVert u\rVert=1\}\).  Define
\(
F_\theta(w):=F(\theta,w)\) and 
\(f_u(w):=u^\intercal F_{\theta_t}(w)\) for \(u\in\mathsf D\).
By Assumption~\ref{assump:growth}, 
\(
|f_u(w)|
\leq \lVert F(\theta_t,w)\rVert
\leq C_FU_1(w)
\) for all \(u\in\mathsf D\) and \(w\in\W\). 
Conditional on \(\F_t\), \(\theta_t\) and hence each \(f_u\) are fixed.
Thus, by homogeneity, the time-shifted proof of
Lemma~\ref{lemma:kernel-telescoping} yields, for each
\(u\in\mathsf D\), with a deterministic constant \(C_0\) uniform in
\(t\), \(\theta_t\), and \(u\),
\begin{align*}
\biggl|
u^\intercal
\biggl(
\E[F(\theta_t,W_{t+\ell})\mid\F_t]
-(P_{\theta_t}^{\ell}F_{\theta_t})(W_t)
\biggr)
\biggr|
={}&
\biggl|
\E[f_u(W_{t+\ell})\mid\F_t]
-(P_{\theta_t}^{\ell}f_u)(W_t)
\biggr|\\
\leq{}&
C_0
\sum_{s=1}^{\ell-1}\rho^{\ell-s-1}
\E\biggl[
U_1(W_{t+s})\lVert\theta_{t+s}-\theta_t\rVert
\,\biggm|\,\F_t
\biggr].
\end{align*}

Because \(\mathsf D\) is countable, these inequalities may be taken to
hold simultaneously.  Taking the supremum over \(u\in\mathsf D\) and
using the Euclidean duality identity
\(
    \lVert x\rVert
    =
    \sup_{u\in\mathsf D}|u^\intercal x|
\)
yields
\begin{align}
\biggl\lVert
\E[F(\theta_t,W_{t+\ell})\mid\F_t]
-(P_{\theta_t}^{\ell}F_{\theta_t})(W_t)
\biggr\rVert 
\leq
C_0
\sum_{s=1}^{\ell-1}\rho^{\ell-s-1}
\E\biggl[
U_1(W_{t+s})\lVert\theta_{t+s}-\theta_t\rVert
\,\biggm|\,\F_t
\biggr]. \label{eq:lagged-update-field-comparison-part-1}
\end{align}

Lemma~\ref{lemma:exp-ergodic}\ref{part:exp-ergodic-uniform-mixing}, again applied by homogeneity to
\(f_u\), implies the following bound for a deterministic constant \(C_1\)
uniform in \(t\), \(\theta_t\), and \(u\):
\begin{align*}
\biggl|
u^\intercal
\biggl(
(P_{\theta_t}^{\ell}F_{\theta_t})(W_t)
-\bar F(\theta_t)
\biggr)
\biggr|
=
\biggl|
(P_{\theta_t}^{\ell}f_u)(W_t)
-\nu_{\theta_t}(f_u)
\biggr|
\leq
C_1\rho^\ell U_1(W_t),
\end{align*}
where
\(
\nu_{\theta_t}(f_u)
=u^\intercal\nu_{\theta_t}(F_{\theta_t})
=u^\intercal\bar F(\theta_t).
\)
Taking the supremum over \(u\in\mathsf D\) therefore yields
\begin{align}
    \biggl\lVert
(P_{\theta_t}^{\ell}F_{\theta_t})(W_t)-\bar F(\theta_t)
\biggr\rVert
\leq
C_1\rho^\ell U_1(W_t).
    \label{eq:lagged-update-field-comparison-part-2}
\end{align}
Set \(C:=\max\{C_0,C_1\}\).  The triangle inequality, combined with
\eqref{eq:lagged-update-field-comparison-part-1} and
\eqref{eq:lagged-update-field-comparison-part-2}, proves
Corollary~\ref{corollary:lagged-update-field-comparison} with \(C\).
\Halmos\endproof

\section{Finite-Time MSE Bound}
\label{sec:EC-MSE-proofs}

\subsection{Basic Estimates}
\label{subsec:ec-canonical-mse-recursion}

\begin{lemma}
\label{lemma:ec-canonical-movement}
Suppose Assumptions~\ref{assump:learning-rate}, \ref{assump:compact},
\ref{assump:ergodic}\ref{part:ergodic-reference}, and
\ref{assump:growth} hold.  Then, for some \(C\) and all \(k,s\geq0\),
\begin{align*} \E[U_1(W_{k+s})\lVert\theta_{k+s}-\theta_k\rVert\mid\F_0]
\leq C U_2(W_0)\sum_{i=k}^{k+s-1}\alpha_i.
\end{align*}
\end{lemma}

\proof{Proof of Lemma~\ref{lemma:ec-canonical-movement}.}
For each \(i\geq0\), since \(\theta_i\in\Theta\), 
\begin{equation}\label{eq:canonical-increment-bound}
\lVert\theta_{i+1}-\theta_i\rVert
=\biggl\lVert
\Pi_\Theta\!\bigl(\theta_i+\alpha_iF(\theta_i,W_i)\bigr)
-\Pi_\Theta(\theta_i)
\biggr\rVert
\leq\alpha_i\lVert F(\theta_i,W_i)\rVert
\leq\alpha_iC_F(1+\lVert W_i\rVert),
\end{equation}
where the two inequalities follow from the projection
nonexpansiveness and Assumption~\ref{assump:growth}, respectively.  
Summing \eqref{eq:canonical-increment-bound} from \(i=k\) through \(k+s-1\) yields
\begin{align}
\E[U_1(W_{k+s})
\lVert\theta_{k+s}-\theta_k\rVert\mid\F_0]
\leq{}&\sum_{i=k}^{k+s-1}\E\biggl[
U_1(W_{k+s})
\alpha_{i}C_F(1+\lVert W_{i}\rVert)
\,\biggm|\,\F_0\biggr] \nonumber \\ 
={}& \sum_{i=k}^{k+s-1} \alpha_{i}C_F\E\biggl[
(1+\lVert W_{k+s}\rVert)
(1+\lVert W_{i}\rVert)
\,\biggm|\,\F_0\biggr].
\label{eq:W-times-theta-diff-bound}
\end{align}

By Lemma~\ref{lemma:adaptive-moment-control}, there is a positive constant \(C_0\)
such that
\(\E[(1+\lVert W_i\rVert)^2\mid\F_0]\leq C_0U_2(W_0)\) for all \(i\geq0\).
Applying Cauchy--Schwarz gives, for all \(i_1,i_2\geq0\),
\begin{align}
\E[(1+\lVert W_{i_1}\rVert)(1+\lVert W_{i_2}\rVert)\mid\F_0]
&\leq\sqrt{
\E[(1+\lVert W_{i_1}\rVert)^2\mid\F_0]
\E[(1+\lVert W_{i_2}\rVert)^2\mid\F_0]}\nonumber\\
&\leq \sqrt{\E[2(1+\lVert W_{i_1}\rVert^2)\mid\F_0]
\E[2(1+\lVert W_{i_2}\rVert^2)\mid\F_0]} \nonumber\\
&\leq C_0 U_2(W_0).
\label{eq:W-times-theta-diff-bound-1}
\end{align}
Set \(C:=C_FC_0\).
Combining \eqref{eq:W-times-theta-diff-bound} and
\eqref{eq:W-times-theta-diff-bound-1} finishes the proof with \(C\).
\Halmos\endproof

\begin{lemma}
\label{lemma:ec-stationary-mean-field-lipschitz}
Suppose Assumptions~\ref{assump:compact},
\ref{assump:ergodic}\ref{part:ergodic-reference},
\ref{assump:kernel-Lipschitz}, and \ref{assump:growth} hold.  Then, for some
\(C\) and all \(\theta,\theta'\in\Theta\),
\[
\lVert\bar{F}(\theta)-\bar{F}(\theta')\rVert
\leq C\lVert\theta-\theta'\rVert.
\]
\end{lemma}

\proof{Proof of Lemma~\ref{lemma:ec-stationary-mean-field-lipschitz}.}
By the definition of \(\bar F(\theta)\) in \eqref{eq:mean-field} and the triangle inequality,
\begin{align}
\lVert\bar{F}(\theta)-\bar{F}(\theta')\rVert
\leq
\biggl\lVert\E_{\nu_\theta}\!
\bigl[F(\theta,W_0)-F(\theta',W_0)\bigr]\biggr\rVert 
+\biggl\lVert\E_{\nu_\theta}[F(\theta',W_0)]
-\E_{\nu_{\theta'}}[F(\theta',W_0)]\biggr\rVert. \label{eq:bar-F-Lipschitz}
\end{align}
Let \(C_0\) be a constant obtained from
Assumption~\ref{assump:growth} and
Lemma~\ref{lemma:exp-ergodic}\ref{part:exp-ergodic-invariant-moment}.
For the first term on the right-hand side of \eqref{eq:bar-F-Lipschitz},
\begin{align}
    \biggl\lVert\E_{\nu_\theta}\!
\bigl[F(\theta,W_0)-F(\theta',W_0)\bigr]\biggr\rVert  \leq{}& 
\E_{\nu_\theta}\! \Bigl[L_F (1+\Bigl\lVert W_0\Bigr\rVert ) \Bigl\lVert \theta-\theta'\Bigr\rVert  \Bigr] \nonumber \\ 
={}& L_F\nu_\theta(U_1)\lVert\theta-\theta'\rVert \nonumber \\
\leq{}& C_0\lVert\theta-\theta'\rVert, \label{eq:bar-F-Lipschitz-1st-term}
\end{align}
where the first inequality uses the triangle inequality for the integral and
Assumption~\ref{assump:growth}, and the second uses
Lemma~\ref{lemma:exp-ergodic}\ref{part:exp-ergodic-invariant-moment}.

For the second term, set \(H(w):=F(\theta',w)\).  Since
\(\lVert H(w)\rVert\leq C_FU_1(w)\) by Assumption~\ref{assump:growth},
Lemma~\ref{lemma:vector-duality}\textup{(ii)}, applied by homogeneity, yields
a constant \(C_1\) such that
\begin{align}
\biggl\lVert\E_{\nu_\theta}[F(\theta',W_0)]
-\E_{\nu_{\theta'}}[F(\theta',W_0)]\biggr\rVert = 
\Bigl\lVert\nu_\theta(H)-\nu_{\theta'}(H)\Bigr\rVert
\leq C_1\lVert\theta-\theta'\rVert. \label{eq:bar-F-Lipschitz-2nd-term}
\end{align}
Set \(C:=C_0+C_1\).  Combining \eqref{eq:bar-F-Lipschitz},
\eqref{eq:bar-F-Lipschitz-1st-term}, and
\eqref{eq:bar-F-Lipschitz-2nd-term} completes the proof with \(C\).
\Halmos\endproof

\begin{lemma}\label{lemma:update-field-bounds}
Suppose Assumptions~\ref{assump:ergodic}\ref{part:ergodic-reference}
and \ref{assump:growth} hold.  Then there is a positive constant \(C\) such that, for all
\(k\geq0\) and \(\theta\in\Theta\),
\begin{gather}
\E[\lVert F(\theta_k,W_k)\rVert^2\mid\F_0]
\leq C U_2(W_0), \label{eq:update-field-second-moment-bound} \\ 
\lVert\bar F(\theta)\rVert
\leq C. \label{eq:F_bar_bound}
\end{gather}
\end{lemma}

\proof{Proof of Lemma~\ref{lemma:update-field-bounds}.}
The condition \eqref{eq:LinearGrowth-F} implies that 
\[
\lVert F(\theta_k,W_k)\rVert^2
\leq 
C_F^2 (1 + \lVert W_k\rVert)^2
\leq
2C_F^2 (1 + \lVert W_k\rVert^2) .
\]
By Lemma~\ref{lemma:adaptive-moment-control}, there is a positive constant \(C_0\) for which
\eqref{eq:update-field-second-moment-bound} holds.
Likewise, for each \(\theta\in\Theta\), by the definition of \(\bar F(\theta)\) in \eqref{eq:mean-field},
\[
\lVert\bar F(\theta)\rVert
\leq
\int_\W\lVert F(\theta,w)\rVert\,\nu_\theta(dw)
\leq
\int_\W C_F (1+ \lVert w \rVert)\,\nu_\theta(dw) = 
C_F \nu_\theta(U_1).
\]
Applying Lemma~\ref{lemma:exp-ergodic}\ref{part:exp-ergodic-invariant-moment} proves
\eqref{eq:F_bar_bound} with a constant \(C_1\).  Set
\(C:=\max\{C_0,C_1\}\).  Then both bounds in the lemma hold with \(C\).
\Halmos\endproof

\subsection{Bias from Adaptive Markov Kernels}
\label{subsec:ec-canonical-mse-proofs}

\proof{Proof of Lemma~\ref{lemma:lagged-restoring-drift}.}
Fix integers \(k>\ell\geq1\).  Decompose
\begin{align}
\langle\theta_k-\theta^\star,F(\theta_k,W_k)\rangle
={}&
\underbrace{\langle\theta_{k-\ell}-\theta^\star,
F(\theta_{k-\ell},W_k)-\bar F(\theta_{k-\ell})\rangle}_{\Psi_1}
+\underbrace{\langle\theta_{k-\ell}-\theta^\star,
F(\theta_k,W_k)-F(\theta_{k-\ell},W_k)\rangle}_{\Psi_2}
\nonumber\\
&+\underbrace{\langle\theta_{k-\ell}-\theta^\star,
\bar F(\theta_{k-\ell})-\bar F(\theta_k)\rangle}_{\Psi_3}
+\underbrace{\langle\theta_k-\theta_{k-\ell},
F(\theta_k,W_k)-\bar F(\theta_k)\rangle}_{\Psi_4}
\nonumber\\
&+\underbrace{\langle\theta_k-\theta^\star,
\bar F(\theta_k)\rangle}_{\Psi_5}.
\label{eq:inner-product-decomp}
\end{align}

We first control \(\Psi_1\). Define \(t:=k-\ell\).  Let \(C_0\) denote
\(d_\Theta\) times the constant in
Corollary~\ref{corollary:lagged-update-field-comparison}.  Conditioning on
\(\F_t\) yields
\begin{align}
\E[\Psi_1\mid\F_t]
&=\biggl\langle \theta_t-\theta^\star,
\E[F(\theta_t,W_k)\mid\F_t]-\bar F(\theta_t)
\biggr\rangle \nonumber \\
& \leq \lVert \theta_t-\theta^\star \rVert
\biggl\lVert
\E[F(\theta_t,W_k)\mid\F_t]-\bar F(\theta_t)
\biggr\rVert \nonumber \\ 
&\leq C_0
\biggl[\rho^\ell U_1(W_t)
+
\sum_{s=1}^{\ell-1}\rho^{\ell-s-1}
\E\biggl[
U_1(W_{t+s})\lVert\theta_{t+s}-\theta_t\rVert
\,\biggm|\,\F_t
\biggr]\biggr], \label{eq:canonical-fixed-state-comparison}
\end{align}
where the last step uses
\(\lVert\theta_t-\theta^\star\rVert\leq d_\Theta\) and
Corollary~\ref{corollary:lagged-update-field-comparison}.

For \(1\leq s\leq\ell-1\), the tower property and
Lemma~\ref{lemma:ec-canonical-movement} yield a constant \(C_1\) such that
\begin{align}
    \label{eq:canonical-fixed-state-comparison-1}
\E\biggl[
\E\biggl[
U_1(W_{t+s})\lVert\theta_{t+s}-\theta_t\rVert
\,\biggm|\,\F_t
\biggr]
\,\biggm|\,\F_0
\biggr]
=
\E[U_1(W_{t+s})\lVert\theta_{t+s}-\theta_t\rVert\mid\F_0]
\leq C_1 U_2(W_0)\sum_{i=t}^{t+s-1}\alpha_i
\end{align}
Moreover,
\begin{align}
    \label{eq:canonical-fixed-state-comparison-2}
\sum_{s=1}^{\ell-1}\rho^{\ell-s-1}
\sum_{i=t}^{t+s-1}\alpha_i
\leq\bigl[1/(1-\rho)\bigr]\sum_{i=t}^{k-1}\alpha_i.
\end{align}

By Lemma~\ref{lemma:adaptive-moment-control}, there is a positive constant \(C_2\)
such that \(\E[U_1(W_t)\mid\F_0]\leq C_2U_2(W_0)\) for all \(t\geq0\).
Let \(C_3\) be a constant
determined by \(C_0\), \(C_1\), \(C_2\), and \((1-\rho)^{-1}\).  Combining
\eqref{eq:canonical-fixed-state-comparison},
\eqref{eq:canonical-fixed-state-comparison-1}, and
\eqref{eq:canonical-fixed-state-comparison-2} with the tower property yields
\begin{align}
\E[\Psi_1\mid\F_0]
\leq{}&
C_3 U_2(W_0)
\biggl(\rho^\ell+\sum_{i=t}^{k-1}\alpha_i\biggr).
\label{eq:lagged-bias-bound}
\end{align}

For \(\Psi_2\), Assumption~\ref{assump:growth} and compactness of
\(\Theta\) imply
\[
|\Psi_2|\leq\lVert\theta_t-\theta^\star\rVert
\lVert F(\theta_k,W_k)-F(\theta_t,W_k)\rVert
\leq d_\Theta L_FU_1(W_k)\lVert\theta_k-\theta_t\rVert.
\]
Lemma~\ref{lemma:ec-canonical-movement} therefore yields a constant
\(C_4\) such that
\begin{align}
|\E[\Psi_2\mid\F_0]|
&\leq C_4 U_2(W_0)
\sum_{i=t}^{k-1}\alpha_i. \label{eq:lag-bound}
\end{align}

For \(\Psi_3\), let \(C_5\) denote the constant in
Lemma~\ref{lemma:ec-stationary-mean-field-lipschitz} multiplied by the fixed
diameter factor.  That lemma and \(U_1\geq1\) imply
\[
|\Psi_3|\leq\lVert\theta_t-\theta^\star\rVert
\lVert\bar F(\theta_t)-\bar F(\theta_k)\rVert
\leq C_5 U_1(W_k)\lVert\theta_k-\theta_t\rVert.
\]
Hence Lemma~\ref{lemma:ec-canonical-movement} yields a constant
\(C_6\) determined by \(C_1\) and \(C_5\),
\begin{align}
|\E[\Psi_3\mid\F_0]|
&\leq C_6 U_2(W_0)
\sum_{i=t}^{k-1}\alpha_i. \label{eq:window-bound}
\end{align}

For \(\Psi_4\), choose \(C_7\) so that
Assumption~\ref{assump:growth}, Lemma~\ref{lemma:update-field-bounds}, the triangle
inequality, and \(U_1\geq1\) imply
\[
|\Psi_4|\leq\lVert\theta_k-\theta_t\rVert
\bigl(\lVert F(\theta_k,W_k)\rVert+\lVert\bar F(\theta_k)\rVert\bigr)
\leq C_7 U_1(W_k)\lVert\theta_k-\theta_t\rVert.
\]
Another application of Lemma~\ref{lemma:ec-canonical-movement} yields, for a
constant \(C_8\) determined by \(C_1\) and \(C_7\),
\begin{align}
|\E[\Psi_4\mid\F_0]|
&\leq C_8 U_2(W_0)
\sum_{i=t}^{k-1}\alpha_i. \label{eq:correction-bound}
\end{align}

For \(\Psi_5\), Assumption~\ref{assump:qsm} implies
\begin{equation} \label{eq:drift-bound}
\E[\Psi_5\mid\F_0]
\leq-\zeta
\E[\lVert\theta_k-\theta^\star\rVert^2\mid\F_0].
\end{equation} 

Set \(C:=C_3+C_4+C_6+C_8\).  Combining
\eqref{eq:inner-product-decomp} and
\eqref{eq:lagged-bias-bound}--\eqref{eq:drift-bound} yields
\begin{align*}
\E[\langle\theta_k-\theta^\star,F(\theta_k,W_k)\rangle
\mid\F_0]
\leq
-\zeta \E[\lVert\theta_k-\theta^\star\rVert^2\mid\F_0]
+C U_2(W_0)
\biggl(\rho^\ell+\sum_{i=k-\ell}^{k-1}\alpha_i\biggr).
\Halmos
\end{align*}
\endproof

\subsection{MSE Recursion}
\label{subsec:ec-canonical-recursion-constants}

\begin{lemma}[MSE Recursion]
\label{lemma:recursion}
Suppose Assumptions~\ref{assump:learning-rate}--\ref{assump:qsm} hold.
Then there are positive constants \(C\) and \(K\) such that for all \(k\geq K\),
\[
\E[\lVert\theta_{k+1}-\theta^\star\rVert^2\mid\F_0]
\leq
\bigl(1-2\zeta\alpha_0k^{-\delta}\bigr)
\E[\lVert\theta_k-\theta^\star\rVert^2\mid\F_0]
+C U_2(W_0)(1+\log k)k^{-2\delta}.
\]

\end{lemma}

\proof{Proof of Lemma~\ref{lemma:recursion}.}

Recall the logarithmic lag
\begin{equation}\label{eq:log-lag}
\ell_k:=
\max\!\biggl\{1,
\Bigl\lceil\frac{\delta\log k}{|\log\rho|}\Bigr\rceil\biggr\},
\end{equation}
and choose an integer \(K\geq2\) so that \(n\geq2\ell_n\) for all \(n\geq K\).
Such a \(K\) can be chosen because \(\ell_n=\mathcal O(\!\log n)\).

Let \(C_0\) be the constant in
Lemma~\ref{lemma:lagged-restoring-drift}.  Applying that lemma with
\(\ell=\ell_k\) yields
\begin{align}
\E[\langle\theta_k-\theta^\star,F(\theta_k,W_k)\rangle
\mid\F_0]
\leq
-\zeta \E[\lVert\theta_k-\theta^\star\rVert^2\mid\F_0]
&+C_0 U_2(W_0)
\biggl(\rho^{\ell_k}+\sum_{i=k-\ell_k}^{k-1}\alpha_i\biggr).
\label{eq:restoring-drift-log-specialization}
\end{align}

Fix \(k\geq K\).  The definition of
\(K\) ensures that the look-back window lies within the most
recent half of the trajectory: \(k\geq2\ell_k\), so
\(k-\ell_k\geq k/2\) and
\(
\alpha_{k-\ell_k}
\leq\alpha_0(k/2)^{-\delta}
=2^\delta\alpha_k.
\)
Monotonicity of the step sizes therefore implies
\begin{equation*}
\sum_{i=k-\ell_k}^{k-1}\alpha_i
\leq\ell_k\alpha_{k-\ell_k}
\leq2^\delta\ell_k\alpha_k\leq 2^\delta \alpha_k
\bigl(1+\delta\log k/|\log\rho|\bigr).
\end{equation*}
Moreover, \(\rho^{\ell_k}
\leq k^{-\delta}=\alpha_k/\alpha_0\). 
Hence there is a positive constant \(C_1\) such that
\begin{align}
U_2(W_0)\biggl(\rho^{\ell_k}
+\sum_{i=k-\ell_k}^{k-1}\alpha_i\biggr)
\leq C_1 U_2(W_0)\alpha_k(1+\log k).
\label{eq:restoring-drift-log-specialization-2}
\end{align}

Combining \eqref{eq:restoring-drift-log-specialization} and \eqref{eq:restoring-drift-log-specialization-2}
yields, with \(C_2:=C_0C_1\),
\begin{align}
\E[\langle\theta_k-\theta^\star,F(\theta_k,W_k)\rangle\mid\F_0]
\leq C_2 U_2(W_0)\alpha_k(1+\log k)
-\zeta\E[\lVert\theta_k-\theta^\star\rVert^2\mid\F_0].
\label{eq:log-lag-restoring-bound}
\end{align}

We now insert \eqref{eq:log-lag-restoring-bound} into
\eqref{eq:proj-nonexpansive}, reproduced below:
\begin{align*}
    \lVert\theta_{k+1}-\theta^\star\rVert^2 
\leq
\lVert\theta_k-\theta^\star\rVert^2
+
2\alpha_k
\langle\theta_k-\theta^\star,F(\theta_k,W_k)\rangle
+
\alpha_k^2\lVert F(\theta_k,W_k)\rVert^2.
\end{align*}
Let 
\(
x_k:=\E[\lVert\theta_k-\theta^\star\rVert^2\mid\F_0].
\)
Let \(C_3\) be the constant in
Lemma~\ref{lemma:update-field-bounds}.
Set \(C_4:=C_3+2C_2\) and \(C:=\alpha_0^2C_4\).  Then
\begin{align*}
x_{k+1}
&\leq x_k+C_3 U_2(W_0)\alpha_k^2
+2\alpha_k\biggl[
C_2 U_2(W_0)\alpha_k(1+\log k)-\zeta x_k
\biggr] \\
&\leq(1-2\zeta\alpha_k)x_k
+C_4 U_2(W_0)\alpha_k^2(1+\log k)\\
&\leq \bigl(1-2\zeta\alpha_0k^{-\delta}\bigr)
x_k
+C U_2(W_0)(1+\log k)k^{-2\delta}
. \Halmos
\end{align*}
\endproof

\subsection{Proof of Proposition~\ref{prop:expectation}}
\label{subsec:MSE-bound}

We treat the cases \(0<\delta<1\) and \(\delta=1\) separately. The corresponding
MSE bounds are given in Propositions~\ref{prop:theta_upper_bound} and
\ref{prop:theta_upper_bound2}, respectively.

\begin{proposition}
    \label{prop:theta_upper_bound}
Suppose Assumptions~\ref{assump:learning-rate}--\ref{assump:qsm} hold with 
\(0<\delta<1\). Let \(\Gamma:=2\zeta\alpha_0\). Then there are positive constants \(C\) and \(K\)
such that, for all \(k\geq K\),
\begin{equation}\label{eq:theta_t_rate}
\E[\lVert\theta_k-\theta^\star\rVert^2\mid\F_0]
\leq d_\Theta^2\exp\biggl(
\frac{\Gamma(K^{1-\delta}-k^{1-\delta})}{1-\delta}
\biggr)
+C U_2(W_0)(1+\log k)k^{-\delta}.
\end{equation}
\end{proposition}

\proof{Proof of Proposition~\ref{prop:theta_upper_bound}.}
Let \(C_0\) and \(K_0\) be constants for which
Lemma~\ref{lemma:recursion} holds.  Choose an integer \(K\geq K_0\) so that
\(K^\delta>\Gamma\) and
\(K^{1-\delta}\geq2\delta/\Gamma\).  We verify
\eqref{eq:theta_t_rate} with \(K\) and a constant \(C\) chosen below.
When \(k=K\),
\eqref{eq:theta_t_rate} holds 
because 
\(\theta_K,\theta^\star\in\Theta\) and the definition of
\(d_\Theta\) imply
\(
\lVert\theta_K-\theta^\star\rVert^2
\leq d_\Theta^2
\).

Now assume \eqref{eq:theta_t_rate} holds for some \(k\geq K\) and
consider \(k+1\).
Lemma~\ref{lemma:recursion} yields
\begin{align}
    \E[\lVert\theta_{k+1}-\theta^\star\rVert^2\mid\F_0] 
    \leq{}&  \bigl(1 - \Gamma k^{-\delta}\bigr) \E [\lVert\theta_{k}-\theta^\star\rVert^2\mid\F_0] + C_0 U_2(W_0)(1+\log k)k^{-2\delta}\nonumber \\    
    \leq{}& \underbrace{\bigl(1 - \Gamma k^{-\delta} \bigr)d_\Theta^2 \exp\biggl(\frac{\Gamma (K^{1-\delta} - k^{1-\delta})}{1-\delta}\biggr)}_{A_1} \nonumber \\ 
    & +
    \underbrace{\bigl(1 - \Gamma k^{-\delta} \bigr) C U_2(W_0)(1+\log k)k^{-\delta} + C_0 U_2(W_0)(1+\log k)k^{-2\delta}}_{A_2}, \label{eq:bound-decomp}
\end{align}
where the second inequality follows from the induction hypothesis and
\(1-\Gamma/k^\delta\geq1-\Gamma/K^\delta>0\).

For \(A_1\), we have
\[
    1-\Gamma k^{-\delta} \leq  \exp\bigl(-\Gamma k^{-\delta}\bigr) \leq \exp\biggl(-\int_k^{k+1} \Gamma u^{-\delta}\dd{u}\biggr) = \exp\biggl(\frac{\Gamma (k^{1-\delta} - (k+1)^{1-\delta})}{1-\delta}\biggr),
\]
where the first inequality follows from the basic relation that $1+u\leq e^u$ for all $u\in\R$.
Thus,
\begin{align}
    A_1 \leq d_\Theta^2 \exp\Bigl(\frac{\Gamma (K^{1-\delta} - (k+1)^{1-\delta})}{1-\delta}\Bigr). \label{eq:D1-bound}
\end{align}

For \(A_2\), direct algebra yields
\begin{align}
    & A_2 - C U_2(W_0)(1+\log k)(k+1)^{-\delta} \notag\\ 
    ={}& U_2(W_0)(1+\log k)k^{-2\delta}
    \biggl[C_0-\Gamma C+Ck^\delta
    \bigl(1-k^\delta(k+1)^{-\delta}\bigr)\biggr].
    \label{eq:D2-difference-bound}
\end{align}
Using \((1+1/k)^k\leq e\) for positive integers \(k\) and
\(1+u\leq e^u\) for each \(u\in\R\), we have
\[
    k^\delta \bigl(1-k^\delta(k+1)^{-\delta}\bigr)
    = k^\delta \bigl[1-(1+k^{-1})^{-\delta}\bigr]
    \leq k^\delta \bigl(1- e^{-\delta/k}\bigr)
    \leq \delta k^{\delta-1} \leq \delta K^{\delta-1}\leq \Gamma/2,
\]
where the last inequality follows from the choice of \(K\).  Thus,
the right-hand side of \eqref{eq:D2-difference-bound} is bounded by 
\[
U_2(W_0)(1+\log k)k^{-2\delta}
\bigl(C_0-\Gamma C/2\bigr).
\]
Choose \(C\geq2C_0/\Gamma\).  The last display is then nonpositive, and
\begin{equation}\label{eq:D2-bound}
    A_2 \leq C U_2(W_0)(1+\log k)(k+1)^{-\delta}.
\end{equation}

Finally, \eqref{eq:bound-decomp}, \eqref{eq:D1-bound},
\eqref{eq:D2-bound}, and \(\log k\leq\log(k+1)\) yield
\begin{align*}
\E[\lVert\theta_{k+1}-\theta^\star\rVert^2\mid\F_0]
&\leq d_\Theta^2
\exp\biggl(\frac{\Gamma (K^{1-\delta}-(k+1)^{1-\delta})}{1-\delta}\biggr)
+C U_2(W_0)(1+\log(k+1))(k+1)^{-\delta}.
\end{align*}
Thus, \eqref{eq:theta_t_rate} holds for \(k+1\), completing the induction.
\Halmos\endproof

\begin{proposition}
    \label{prop:theta_upper_bound2}
Suppose Assumptions~\ref{assump:learning-rate}--\ref{assump:qsm} hold with \(\delta=1\) and \(\Gamma:=2\zeta\alpha_0>1\). Then there are positive constants \(C\) and \(K\) such that for each \(k\geq K\),
\begin{equation}\label{eq:theta_t_rate2}
\E[\lVert\theta_k-\theta^\star\rVert^2\mid\F_0]
\leq d_\Theta^2\bigl(K/k\bigr)^{\Gamma}
+C U_2(W_0)(1+\log k)k^{-1}.
\end{equation}
\end{proposition}

\proof{Proof of Proposition~\ref{prop:theta_upper_bound2}.} 
Let \(C_0\) and \(K_0\) be constants for which
Lemma~\ref{lemma:recursion} holds.  Choose an integer \(K\geq K_0\) with
\(K>\Gamma\).  When \(k=K\),
\eqref{eq:theta_t_rate2} holds because
\(\theta_K,\theta^\star\in\Theta\) and the definition of
\(d_\Theta\) imply
\(
\lVert\theta_K-\theta^\star\rVert^2
\leq d_\Theta^2
\).

For \(k\geq K+1\), applying Lemma~\ref{lemma:recursion} at index \(k-1\)
first yields
\begin{align*}
\E[\lVert\theta_k-\theta^\star\rVert^2\mid\F_0]
&\leq\bigl(1-\Gamma(k-1)^{-1}\bigr)
\E[\lVert\theta_{k-1}-\theta^\star\rVert^2\mid\F_0]
+C_0 U_2(W_0)(1+\log(k-1))(k-1)^{-2}\\
&\leq\bigl(1-\Gamma(k-1)^{-1}\bigr)
\E[\lVert\theta_{k-1}-\theta^\star\rVert^2\mid\F_0]
+C_0 U_2(W_0)(1+\log k)(k-1)^{-2}.
\end{align*}
Iterating this bound yields
\begin{align}
\E[\lVert\theta_k-\theta^\star\rVert^2\mid\F_0]
&\leq d_\Theta^2\prod_{i=K}^{k-1}\bigl(1-\Gamma i^{-1}\bigr)
+C_0 U_2(W_0)(1+\log k)
\sum_{n=K}^{k-1}n^{-2}
\prod_{i=n+1}^{k-1}\bigl(1-\Gamma i^{-1}\bigr).
\label{eq:iterative_inequality}
\end{align}
All factors are positive because \(K>\Gamma\).

For each \(n\geq K\),
\begin{equation}\label{eq:prod_upper_bound}
\prod_{i=n}^{k-1}\bigl(1-\Gamma i^{-1}\bigr)
\leq\exp\biggl(-\Gamma\sum_{i=n}^{k-1}i^{-1}\biggr)
\leq\exp\biggl(-\Gamma\int_n^k u^{-1}\dd{u}\biggr)
=\bigl(n/k\bigr)^\Gamma.
\end{equation}
Hence,
\begin{align*}
\sum_{n=K}^{k-1}n^{-2}
\prod_{i=n+1}^{k-1}\bigl(1-\Gamma i^{-1}\bigr) 
\leq\bigl((1+1/K)k^{-1}\bigr)^\Gamma
\sum_{n=K}^{k-1}n^{\Gamma-2},
\end{align*}
because \(n+1\leq(1+1/K)n\).  Applying the left-end integral comparison
when \(1<\Gamma<2\) and the right-end comparison when \(\Gamma\geq2\)
yields
\[
\sum_{n=K}^{k-1}n^{\Gamma-2}\leq k^{\Gamma-1}/(\Gamma-1).
\]
Thus, in both cases,
\begin{equation}\label{eq:prod_upper_bound-3}
\sum_{n=K}^{k-1}n^{-2}
\prod_{i=n+1}^{k-1}\bigl(1-\Gamma i^{-1}\bigr)
\leq\bigl((1+1/K)^\Gamma/(\Gamma-1)\bigr)k^{-1}.
\end{equation}
Set
\(
C:=C_0(1+1/K)^\Gamma/(\Gamma-1).
\)
Substituting \eqref{eq:prod_upper_bound} and  \eqref{eq:prod_upper_bound-3} into
\eqref{eq:iterative_inequality} proves \eqref{eq:theta_t_rate2}.
\Halmos\endproof 

\proof{Proof of Proposition~\ref{prop:expectation}.}
For \(0<\delta<1\), let \(C_0\) and \(K_0\) be the constant and iteration
threshold in Proposition~\ref{prop:theta_upper_bound}.  For \(\delta=1\), let
\(C_0\) and \(K_0\) be the constant and iteration threshold in
Proposition~\ref{prop:theta_upper_bound2}.  The first term on the right-hand
side of \eqref{eq:theta_t_rate} is \(o(k^{-\delta})\), while the first term on
the right-hand side of \eqref{eq:theta_t_rate2} is
\(\mathcal O(k^{-\Gamma})=o(k^{-1})\).  In either case, choose an integer
\(K\geq K_0\) so that the corresponding first term is at most
\((1+\log k)k^{-\delta}\) for all \(k\geq K\), and set \(C:=C_0+1\).  Since
\(U_2(W_0)\geq1\),
\begin{equation*}
\E[\lVert\theta_k-\theta^\star\rVert^2\mid\F_0]
\leq
 C U_2(W_0)(1+\log k)k^{-\delta},
\end{equation*}
for all \(k\geq K\).
\Halmos\endproof

\section{Shrinking-Tube Concentration}
\label{sec:EC-trajectory-proofs}

\subsection{Block Decomposition}
\label{subsec:ec-canonical-block-decomposition}

We divide the trajectory of SA iterates into blocks. 
For each \(r\geq 1\), block \(r\) consists of iterates at times  \(k\in 
\mathcal I_r:=\{n_r,\ldots,n_{r+1}-1\}\). 
If \(1/2<\delta<1\), we use polynomial blocks: 
\begin{equation}\label{eq:block-definition}
n_r:=\Bigl\lceil r^{1/(1-\delta)}\Bigr\rceil;
\end{equation}
if \(\delta=1\), we use dyadic blocks: 
\begin{equation}\label{eq:dyadic-block-definition}
n_r:=2^r.
\end{equation}
For either block formulation, let
\(\psi_r:=c n_{r+1}^{-\beta}.
\)
For an integer \(k_0\geq n_1\), let \(r_0=r_0(k_0)\) be the index of the
block containing \(k_0\), so \(n_{r_0}\leq k_0<n_{r_0+1}\).

The following lemma reduces a shrinking-tube exit to a block exit.

\begin{lemma}
\label{lemma:blockwise-tube-reduction}
For either block construction and every integer \(k_0\geq n_1\),
\begin{equation*}
\biggl\{
\lVert\theta_k-\theta^\star\rVert>c k^{-\beta}
\text{ for some }k\geq k_0
\biggr\}
\subseteq
\bigcup_{r\geq r_0}
\biggl\{
\max_{k\in\mathcal I_r}
\lVert\theta_k-\theta^\star\rVert>\psi_r
\biggr\}.
\end{equation*}
\end{lemma}

\proof{Proof of Lemma~\ref{lemma:blockwise-tube-reduction}.}
Choose \(k\geq k_0\) at which  \(\lVert\theta_k-\theta^\star\rVert>c k^{-\beta}\).
Let \(r\) be the block containing \(k\).  Then
\(r\geq r_0\).  Since \(k<n_{r+1}\) and \(\beta\geq0\),
\(
k^{-\beta}\geq n_{r+1}^{-\beta},
\)
and hence
\[
\lVert\theta_k-\theta^\star\rVert
>c k^{-\beta}
\geq c n_{r+1}^{-\beta}=\psi_r.
\]
Thus the maximum over block \(r\) exceeds \(\psi_r\), which completes the proof.
\Halmos\endproof

Fix a block \(r\) satisfying
\(n_r\geq\ell_{n_r}\), where \(\ell_{n_r}\) is the logarithmic lag defined in \eqref{eq:log-lag} with \(k=n_r\). 
For
\(m\in\mathcal I_r\), define
\begin{align}
S_r^{\mathrm{markov}}(m)
&:=
\sum_{k=n_r}^{m}
\alpha_k\bigl(
F(\theta_{k-\ell_{n_r}},W_k)
-\bar F(\theta_{k-\ell_{n_r}})
\bigr),
\label{eq:def-S-rm}\\
S_r^{\mathrm{field}}(m)
&:=
\sum_{k=n_r}^{m}
\alpha_k\bar F(\theta_{k-\ell_{n_r}}),
\\
S_r^{\mathrm{lag}}(m)
&:=
\sum_{k=n_r}^{m}
\alpha_k\bigl(
F(\theta_k,W_k)-F(\theta_{k-\ell_{n_r}},W_k)
\bigr).
\label{eq:def-S-adapt-rm}
\end{align}
For the same block, define the four component events
\begin{align*}
\mathcal E_r^{\mathrm{ent}}
:=
\biggl\{
\lVert\theta_{n_r}-\theta^\star\rVert
>\psi_r/8
\biggr\},\qquad 
\mathcal E_r^{\mathord{\bullet}}
:=
\biggl\{
\max_{m\in\mathcal I_r}
\lVert S_r^{\mathord{\bullet}}(m)\rVert
>\psi_r/8
\biggr\}.
\end{align*}
Here \(\mathord{\bullet}\) stands for
\(\mathrm{markov}\), \(\mathrm{field}\), or \(\mathrm{lag}\).  The following lemma reduces a block exit to the four component events.

Because \(\theta^\star\in\operatorname{int}(\Theta)\), let
\(d_\star
:=
\inf_{\theta\in\Theta^{\complement}}
\Bigl\lVert \theta^\star-\theta\Bigr\rVert
>0\).
By monotonicity of the tube radius in \(c\), it suffices to prove
Theorem~\ref{thm:trajectory} for \(0<c\leq d_\star/2\).

\begin{lemma}
\label{lemma:canonical-block-containment}
Suppose Assumptions~\ref{assump:compact} and
\ref{assump:interior-target} hold. Fix
\(0<c\leq d_\star/2\), \(\beta\geq0\), and \(r\). Then
\begin{equation}
\biggl\{
\max_{k\in\mathcal I_r}
\lVert\theta_k-\theta^\star\rVert>\psi_r
\biggr\}
\subseteq
\mathcal E_r^{\mathrm{ent}}
\cup\mathcal E_r^{\mathrm{markov}}
\cup\mathcal E_r^{\mathrm{field}}
\cup\mathcal E_r^{\mathrm{lag}}.
\label{eq:full-block-failure-decomposition}
\end{equation}
\end{lemma}

\proof{Proof of Lemma~\ref{lemma:canonical-block-containment}.}
Fix an outcome outside all four events.  If
\(\lVert\theta_k-\theta^\star\rVert\leq \psi_r\) for all
\(k\in\mathcal I_r\), then the event on the left-hand side of
\eqref{eq:full-block-failure-decomposition} does not occur.  Otherwise, define
the first exit time
\[
\tau:=\min\biggl\{k\in\{n_r,\ldots,n_{r+1}-1\}:
\lVert\theta_k-\theta^\star\rVert>\psi_r\biggr\}.
\]
Because the outcome is outside \(\mathcal E_r^{\mathrm{ent}}\),
\(\lVert\theta_{n_r}-\theta^\star\rVert\leq \psi_r/8<\psi_r\), and hence
\(\tau>n_r\).  Since \(\tau\) is the first exit time,
\[
\lVert\theta_j-\theta^\star\rVert
\leq \psi_r
\leq c
\leq d_\star/2
<d_\star,
\]
for all \(j=n_r,\ldots,\tau-1\).
Thus, all these iterates belong to \(\operatorname{int}(\Theta)\).

For \(k=n_r,\ldots,\tau-1\), define the preprojection point
\[
u_{k+1}:=\theta_k+\alpha_kF(\theta_k,W_k),\]
so \(
\theta_{k+1}=\Pi_\Theta(u_{k+1}).
\)
For all \(k=n_r,\ldots,\tau-2\), the projected output
\(\theta_{k+1}\) is in the interior.  If \(u_{k+1}\neq\theta_{k+1}\), then, for
sufficiently small \(\varepsilon\in(0,1)\),
\(
z:=\theta_{k+1}+\varepsilon(u_{k+1}-\theta_{k+1})\in\Theta
\)
and
\[
\lVert u_{k+1}-z\rVert
=(1-\varepsilon)\lVert u_{k+1}-\theta_{k+1}\rVert
<\lVert u_{k+1}-\theta_{k+1}\rVert.
\]
This contradicts the fact that
\(\theta_{k+1}=\Pi_\Theta(u_{k+1})\) is the closest point in \(\Theta\) to
\(u_{k+1}\).
Hence these preceding projections are inactive and
\(\theta_{\tau-1}=\theta_{n_r}
+\sum_{k=n_r}^{\tau-2}\alpha_kF(\theta_k,W_k)\).
At the possibly exiting update, \(\theta_\tau=\Pi_\Theta(u_\tau)\), while
\(\theta^\star=\Pi_\Theta(\theta^\star)\).  Projection nonexpansiveness yields
\begin{align*}
\lVert\theta_\tau-\theta^\star\rVert
\leq
\lVert u_\tau-\theta^\star\rVert
=
\biggl\lVert
\theta_{n_r}-\theta^\star
+\sum_{k=n_r}^{\tau-1}\alpha_kF(\theta_k,W_k)
\biggr\rVert.
\end{align*}
Adding and subtracting \(F(\theta_{k-\ell_{n_r}},W_k)\) and
\(\bar F(\theta_{k-\ell_{n_r}})\) in each summand yields
\[
\sum_{k=n_r}^{\tau-1}\alpha_kF(\theta_k,W_k)
=S_r^{\mathrm{markov}}(\tau-1)
+S_r^{\mathrm{field}}(\tau-1)
+S_r^{\mathrm{lag}}(\tau-1).
\]
Because \(\tau>n_r\) and \(\tau\leq n_{r+1}-1\), one has
\(\tau-1\in\mathcal I_r\).  The complement of the four component events and
the triangle inequality now imply
\begin{align*}
\lVert\theta_\tau-\theta^\star\rVert
&\leq
\lVert\theta_{n_r}-\theta^\star\rVert
+\lVert S_r^{\mathrm{markov}}(\tau-1)\rVert
+\lVert S_r^{\mathrm{field}}(\tau-1)\rVert
+\lVert S_r^{\mathrm{lag}}(\tau-1)\rVert\\
&\leq
4\cdot \psi_r/8 <\psi_r,
\end{align*}
contradicting the definition of \(\tau\).  This completes the proof.
\Halmos\endproof

\subsection{Maximal Bound for Accumulated Markovian Noise}
\label{subsec:ec-poly-block-noise}

\begin{lemma}
\label{lemma:ec-canonical-lagged-centering}
Suppose Assumptions~\ref{assump:learning-rate}--\ref{assump:growth} hold. Fix
integers \(\ell\geq1\) and \(I\leq J\) with \(I\geq\ell+1\).  For
\(k=I,\ldots,J\), define the lagged Markovian noise and its
lagged-centering decomposition by
\begin{equation}\label{eq:lag-centering-main}
\xi_{k,\ell}:=F(\theta_{k-\ell},W_k)-\bar F(\theta_{k-\ell}),\quad
\xi_{k,\ell}^{\mathrm{pred}}:=\E[\xi_{k,\ell}\mid\F_{k-\ell}],\quad\mbox{and}\quad
\xi_{k,\ell}^{\mathrm{cent}}:=\xi_{k,\ell}-\xi_{k,\ell}^{\mathrm{pred}}.
\end{equation}
These random vectors are square integrable.  There is a positive constant \(C\) such that, for all \(k=I,\ldots,J\),
\begin{equation}
\E[\lVert\xi_{k,\ell}^{\mathrm{pred}}\rVert\mid\F_0]
\leq C U_2(W_0)\bigl(\rho^\ell+\ell\alpha_{k-\ell}\bigr)
\label{eq:lag-predictable-bias}.
\end{equation}
The same \(C\) also satisfies, for \(0\leq t<k\),
\begin{equation}
\E[\lVert\xi_{k,\ell}^{\mathrm{cent}}\rVert^2\mid\F_t]
\leq
C
\begin{cases}
\lambda_2^{k-t}U_2(W_t)+1,
&0\leq t\leq k-\ell,\\
U_2(W_t)+U_2(W_{k-\ell})+1,
&k-\ell<t<k
\end{cases}.
\label{eq:lag-centered-conditional-moment}
\end{equation}
If \(I\leq i<k\leq J\), \(0\leq t\leq i\), and
\(k-i\geq\ell\), then
\begin{equation}\label{eq:lag-covariance-cancellation}
\E[(\xi_{k,\ell}^{\mathrm{cent}})^\intercal
\xi_{i,\ell}^{\mathrm{cent}}\mid\F_t]=0.
\end{equation}
\end{lemma}

\proof{Proof of Lemma~\ref{lemma:ec-canonical-lagged-centering}.}
By Assumption~\ref{assump:growth} and
Lemma~\ref{lemma:update-field-bounds}, choose a constant \(C_0\) such that
\(\lVert\xi_{k,\ell}\rVert\leq C_0 U_1(W_k)\).  Set
\(C_1:=2C_0^2\).  Since
\(U_1(W_k)^2\leq2U_2(W_k)\), we have
\begin{equation}\label{eq:lag-xi-growth}
    \lVert\xi_{k,\ell}\rVert^2
    \leq C_1 U_2(W_k).
\end{equation}
Set \(C_2:=b_2/(1-\lambda_2)\).  Lemma~\ref{lemma:adaptive-moment-control},
the tower property,
and Assumption~\ref{assump:ergodic}\ref{part:ergodic-initial} imply
\(\E[U_2(W_k)]\leq\E[U_2(W_0)]+C_2<\infty\).
Thus \(\xi_{k,\ell}\in L^2\) by \eqref{eq:lag-xi-growth}, and Jensen's
inequality extends this conclusion to
\(\xi_{k,\ell}^{\mathrm{pred}},\xi_{k,\ell}^{\mathrm{cent}}\).  By definition,
\(\E[\xi_{k,\ell}^{\mathrm{cent}}\mid\F_{k-\ell}]=0\).

To bound \(\xi_{k,\ell}^{\mathrm{pred}}\), fix \(k\) and set
\(t=k-\ell\).  The identity
\[
\xi_{k,\ell}^{\mathrm{pred}}
=\E[F(\theta_t,W_k)\mid\F_t]-\bar F(\theta_t)
\]
holds.  Let \(C_3\) be the constant in
Corollary~\ref{corollary:lagged-update-field-comparison}. Then,
\begin{align*}
\lVert\xi_{k,\ell}^{\mathrm{pred}}\rVert
\leq C_3\rho^\ell U_1(W_t)
+C_3
\sum_{s=1}^{\ell-1}\rho^{\ell-s-1}
\E[U_1(W_{t+s})\lVert\theta_{t+s}-\theta_t\rVert
\mid\F_t].
\end{align*}
By Lemmas~\ref{lemma:adaptive-moment-control}
and~\ref{lemma:ec-canonical-movement}, together with the tower property, there is a positive constant \(C_4\) such that
\begin{align*}
\E[\lVert\xi_{k,\ell}^{\mathrm{pred}}\rVert\mid\F_0]
&\leq C_4 U_2(W_0)\rho^\ell
+C_4 U_2(W_0)
\sum_{s=1}^{\ell-1}\rho^{\ell-s-1}
\sum_{i=t}^{t+s-1}\alpha_i\\
&\leq C_4 U_2(W_0)\bigl(\rho^\ell+\ell\alpha_t\bigr),
\end{align*}
where monotonicity implies
\(\sum_{i=t}^{t+s-1}\alpha_i\leq s\alpha_t\) and
\(\sum_{s=1}^{\ell-1}s\rho^{\ell-s-1}\leq
\ell/(1-\rho)\).  This is \eqref{eq:lag-predictable-bias}.

For the second-moment claims, apply Jensen's inequality:
\begin{equation}\label{eq:lag-centered-projection}
\E[\lVert\xi_{k,\ell}^{\mathrm{cent}}\rVert^2\mid\F_{k-\ell}]
\leq\E[\lVert\xi_{k,\ell}\rVert^2\mid\F_{k-\ell}].
\end{equation}
If \(t\leq k-\ell\), take the conditional expectation of
\eqref{eq:lag-centered-projection} with respect to \(\F_t\), and
apply Lemma~\ref{lemma:adaptive-moment-control}.  There is a positive constant \(C_5\) for which the first case of
\eqref{eq:lag-centered-conditional-moment} holds.  If
\(k-\ell<t<k\), Jensen's inequality and
\eqref{eq:lag-xi-growth} imply, for a constant \(C_6\),
\begin{align*}
\E[\lVert\xi_{k,\ell}^{\mathrm{cent}}\rVert^2\mid\F_t]
&\leq2\E[\lVert\xi_{k,\ell}\rVert^2\mid\F_t]
+2\lVert\xi_{k,\ell}^{\mathrm{pred}}\rVert^2\\
&\leq C_6
\bigl(\lambda_2^{k-t}U_2(W_t)+1\bigr)+C_6
\bigl(\lambda_2^\ell U_2(W_{k-\ell})+1\bigr),
\end{align*}
which implies the second case of
\eqref{eq:lag-centered-conditional-moment}.

Finally, suppose \(t\leq i<k\) and \(k-i\geq\ell\).  By exact centering and
the tower property,
\begin{align*}
\E[(\xi_{k,\ell}^{\mathrm{cent}})^\intercal
\xi_{i,\ell}^{\mathrm{cent}}\mid\F_t]
&=\E\biggl[
\E[(\xi_{k,\ell}^{\mathrm{cent}})^\intercal
\xi_{i,\ell}^{\mathrm{cent}}
\mid\F_{k-\ell}]
\,\biggm|\,\F_t
\biggr]\\
&=\E\biggl[
(\xi_{i,\ell}^{\mathrm{cent}})^\intercal
\E[\xi_{k,\ell}^{\mathrm{cent}}\mid\F_{k-\ell}]
\,\biggm|\,\F_t
\biggr]=0,
\end{align*}
which proves
\eqref{eq:lag-covariance-cancellation}.  Taking
\(C:=\max\{C_4,C_5,C_6\}\) proves the lemma.
\Halmos\endproof

\begin{lemma}[L\'evy Inequality for Dependent Sequences]
\label{lemma:extended-levy-dependent}
Let \(\{\mathcal G_0,\ldots,\mathcal G_n\}\) be a filtration, and suppose
\(\mathcal H\subseteq\mathcal G_0\). For each \(m=0,\ldots,n\), let \(Y_m\) be a
finite, real-valued, \(\mathcal G_m\)-measurable random variable, and let \(q_m\) be a
\(\mathcal G_m\)-measurable conditional median of \(Y_m-Y_n\) given
\(\mathcal G_m\), that is,
\begin{equation}\label{eq:conditional-median-definition}
\pr(Y_m-Y_n\leq q_m\mid\mathcal G_m)\geq1/2 \quad\mbox{and}\quad 
\pr(Y_m-Y_n\geq q_m\mid\mathcal G_m)\geq1/2,
\end{equation}
with \(q_n=0\). Then, for all \(z>0\),
\begin{equation}\label{eq:extended-levy-dependent}
\pr\biggl(
\max_{0\leq m\leq n}|Y_m-q_m|\geq z
\,\biggm|\,\mathcal H
\biggr)
\leq
2\pr(|Y_n|\geq z\mid\mathcal H).
\end{equation}
\end{lemma}

\proof{Proof of Lemma~\ref{lemma:extended-levy-dependent}.}
The proof follows the extended L\'evy inequality in
\citet[Section~32.1]{LoeveProbBook}.
Define \(T:=\inf\{m\geq 0: |Y_m-q_m|\geq z\}\), 
\[
A_m^+:=\{T=m,\ Y_m-q_m\geq z\}\quad\mbox{and}\quad
A_m^-:=\{T=m,\ Y_m-q_m\leq-z\}.
\]
On
\(A_m^+\cap\{Y_m-Y_n\leq q_m\}\), one has \(Y_n\geq z\).  Hence, by
\eqref{eq:conditional-median-definition} and conditioning through
\(\mathcal G_m\),
\begin{align*}
\pr(A_m^+\cap\{Y_n\geq z\}\mid\mathcal H)
&\geq
\E\biggl[
\ind(A_m^+)\pr(Y_m-Y_n\leq q_m\mid\mathcal G_m)
\,\biggm|\,\mathcal H
\biggr]\\
&\geq1/2\pr(A_m^+\mid\mathcal H).
\end{align*}
Summing over \(m\) yields
\(
\pr(\cup_m A_m^+\mid\mathcal H)
\leq2\pr(Y_n\geq z\mid\mathcal H)
\).
Likewise, 
\(
\pr(\cup_m A_m^-\mid\mathcal H)
\leq2\pr(Y_n\leq-z\mid\mathcal H)
\).
Adding the two bounds proves \eqref{eq:extended-levy-dependent}.
\Halmos\endproof

\begin{lemma}
\label{lemma:conditional-levy-remainder}
Let \(\{\mathcal G_0,\ldots,\mathcal G_n\}\) be a filtration, and suppose
\(\mathcal H\subseteq\mathcal G_0\). For each \(
m=1,\ldots,n\), let \(X_m\) be a 
square-integrable, \(\R^d\)-valued, \(\mathcal G_m\)-measurable random vector. Let \(S_m:=\sum_{k=1}^{m}X_k\) for \(m=1,\ldots,n\) and 
\(
S_0:=0\). 
Then, for any \(x>0\),
\begin{align*}
\pr\biggl(
\max_{0\leq m\leq n}\lVert S_m\rVert\geq x
\,\biggm|\,\mathcal H
\biggr) \leq
8x^{-2}\E[\lVert S_n\rVert^2\mid\mathcal H]
+
\pr\biggl(
\max_{0\leq m\leq n}
\E[\lVert S_n-S_m\rVert^2\mid\mathcal G_m]
\geq x^2/8
\,\biggm|\,\mathcal H
\biggr).
\end{align*}
\end{lemma}

\proof{Proof of Lemma~\ref{lemma:conditional-levy-remainder}.}
For each \(m<n\), let 
\(q_m\) denote the conditional median of \(\lVert S_m\rVert-\lVert S_n\rVert\) given
\(\mathcal G_m\), and set \(q_n:=0\).
Markov's inequality gives
\[
\E[(\lVert S_m\rVert-\lVert S_n\rVert)^2\mid\mathcal G_m] \geq
q_m^2  \pr\bigl(
\bigl|\lVert S_m\rVert-\lVert S_n\rVert\bigr|\geq|q_m|
\mid\mathcal G_m\bigr)\geq \frac{q_m^2}{2},
\]
where the second inequality follows from the definition of a conditional median. 
Using the triangle inequality yields
\begin{align*}
q_m^2
\leq
2\E[(\lVert S_m\rVert-\lVert S_n\rVert)^2\mid\mathcal G_m]
\leq
2\E[\lVert S_n-S_m\rVert^2\mid\mathcal G_m].
\end{align*}
This bound also holds when \(q_m=0\).
Consequently,
\begin{align*}
\Bigl\{\max_{0\leq m\leq n}\lVert S_m\rVert\geq x\Bigr\}
& \subseteq \Bigl\{ \max_{0\leq m\leq n}\Bigl|\lVert S_m\rVert - q_m\Bigr|\geq x/2 \Bigr\} 
\cup 
 \bigl\{ \max_{0\leq m\leq n}\bigl| q_m \bigr|\geq x/2 \bigr\}  \\
& \subseteq 
\Bigl\{\max_{0\leq m\leq n}\bigl|\lVert S_m\rVert-q_m\bigr|\geq x/2\Bigr\}
\cup
\Bigl\{\max_{0\leq m\leq n} \E[\lVert S_n-S_m\rVert^2\mid\mathcal G_m] \geq x^2/8\Bigr\}    
\end{align*}
Applying Lemma~\ref{lemma:extended-levy-dependent} and Markov's inequality completes the proof.
\Halmos\endproof

\begin{lemma}[Block-Maximal Bound for Accumulated Markovian Noise]
\label{lemma:ec-canonical-block-maximal}
Suppose 
Assumptions~\ref{assump:learning-rate}--\ref{assump:growth} hold.
Fix integers \(\ell\geq1\) and \(I\leq J\) with \(I\geq2\ell\). 
Set \(L:=J-I+1\) and define 
\[
S^{\mathrm{markov}}(m):=\sum_{k=I}^{m}\alpha_k
\bigl(F(\theta_{k-\ell},W_k)-\bar F(\theta_{k-\ell})\bigr),
\]
for \(I\leq m\leq J\). Then
there is a positive constant \(C\) such that, for all \(x>0\),
\begin{align*}
\pr\biggl(
\max_{I\leq m\leq J}\lVert S^{\mathrm{markov}}(m)\rVert\geq x
\,\biggm|\,\F_0
\biggr)\leq
C U_2(W_0)
\biggl[
x^{-2}\ell^2L\alpha_I^2
+x^{-1}L\alpha_I
\biggl(
\rho^\ell+\ell\alpha_{I-\ell}
\biggr)
\biggr].
\end{align*}
\end{lemma}

\proof{Proof of Lemma~\ref{lemma:ec-canonical-block-maximal}.}
Let \(C_0\) be the constant in Lemma~\ref{lemma:ec-canonical-lagged-centering}.
Fix a block as in the lemma.  By hypothesis, \(I\geq2\ell\).  Use
\((\xi_{k,\ell},\xi_{k,\ell}^{\mathrm{pred}},
\xi_{k,\ell}^{\mathrm{cent}})\) from
\eqref{eq:lag-centering-main} and define 
\begin{align*}
S^{\mathrm{pred}}(m):=\sum_{k=I}^{m}\alpha_k\xi_{k,\ell}^{\mathrm{pred}} 
\quad \mbox{and} \quad
S^{\mathrm{cent}}(m):=\sum_{k=I}^{m}\alpha_k\xi_{k,\ell}^{\mathrm{cent}},
\end{align*} 
for \(I\leq m \leq J\), 
with \(S^{\mathrm{cent}}(I-1)=S^{\mathrm{pred}}(I-1)=0\).  Thus
\(S^{\mathrm{markov}}(m)=S^{\mathrm{pred}}(m)+S^{\mathrm{cent}}(m)\).

We first
expand \(\E[\lVert S^{\mathrm{cent}}(J)\rVert^2\mid\F_0]\).  By
Lemma~\ref{lemma:ec-canonical-lagged-centering}, all cross terms with
\(|k-i|\geq \ell\) vanish.  For the remaining pairs, Cauchy--Schwarz and
\(2ab\leq a^2+b^2\) imply
\begin{align*}
2\alpha_k\alpha_i
\biggl|\E[(\xi_{k,\ell}^{\mathrm{cent}})^\intercal
\xi_{i,\ell}^{\mathrm{cent}}
\mid\F_0]\biggr|
\leq
\alpha_k^2\E[\lVert\xi_{k,\ell}^{\mathrm{cent}}\rVert^2\mid\F_0]
+\alpha_i^2\E[\lVert\xi_{i,\ell}^{\mathrm{cent}}\rVert^2\mid\F_0].
\end{align*}
Each index occurs in at most \(2(\ell-1)\) such near pairs.  Therefore, using
Lemma~\ref{lemma:ec-canonical-lagged-centering} at \(t=0\),
\(U_2(W_0)\geq1\), and monotonicity of the step sizes,
\begin{equation}\label{eq:cen-full-block-second-moment}
\E[\lVert S^{\mathrm{cent}}(J)\rVert^2\mid\F_0]
\leq 4C_0 U_2(W_0)\ell L\alpha_I^2.
\end{equation}

For \(m=I-1,\ldots,J\), let
\begin{equation*}
V_m:=\E[\lVert S^{\mathrm{cent}}(J)-S^{\mathrm{cent}}(m)\rVert^2\mid\F_m].
\end{equation*}
In particular, \(V_J=0\).
The same covariance cancellation, now conditional on \(\F_m\),
implies
\begin{equation}\label{eq:V-m-near-pairs}
V_m\leq(2\ell-1)\sum_{k=m+1}^{J}
\alpha_k^2\E[\lVert\xi_{k,\ell}^{\mathrm{cent}}\rVert^2\mid\F_m].
\end{equation}
For both \(k\geq m+\ell\) and \(m<k<m+\ell\), use
Lemma~\ref{lemma:ec-canonical-lagged-centering}.  For the near terms, the
valid indices \(k-\ell\) form a
subset of \(\{(m-\ell+1)\vee0,\ldots,m-1\}\).  Because \(U_2\geq0\), their
contribution is bounded by the sum over this full set.  Since
\(2\ell-1\leq2\ell\),
\(\alpha_k\leq\alpha_I\), and
\(\sum_{h=\ell}^{\infty}\lambda_2^h\leq1/(1-\lambda_2)\),
\eqref{eq:V-m-near-pairs} yields
\begin{align*}
V_m
& \leq 2C_0\ell\alpha_I^2
\Bigl(\frac{U_2(W_m)}{1-\lambda_2}+L\Bigr)+2C_0\ell\alpha_I^2
\biggl(\underbrace{\ell U_2(W_m)
+\sum_{i=(m-\ell+1)\vee0}^{m-1}U_2(W_i)
+\ell}_{:=Z_\ell(m)}\biggr)\\ 
& \leq C_1\ell(L+\ell)\alpha_I^2+C_2\ell\alpha_I^2 Z_\ell(m),
\end{align*}
for suitable \(C_1\) and \(C_2\).
By Lemma~\ref{lemma:adaptive-moment-control}, \(U_2(W_0)\geq1\), and the
definition of \(Z_\ell(m)\), there is a positive constant \(C_3\) such that
\[
\E[Z_\ell(m)\mid\F_0]
\leq C_3\ell U_2(W_0),
\qquad I-1\leq m\leq J-1.
\]
Let \(C_4\) be a
constant determined by \(C_2\) and \(C_3\).  Suppose first that
\(C_1\ell(L+\ell)\alpha_I^2\leq y/2\). Markov's inequality and a
union bound over \(m\) then yield
\begin{align*}
\pr\biggl(\max_{I-1\leq m\leq J}V_m\geq y
\,\biggm|\,\F_0\biggr)
&\leq
2C_2 y^{-1}\ell\alpha_I^2
\sum_{m=I-1}^{J-1}\E[Z_\ell(m)\mid\F_0]\\
&\leq
C_4 U_2(W_0)y^{-1}\ell^2L\alpha_I^2.
\end{align*}
If instead \(C_1\ell(L+\ell)\alpha_I^2>y/2\), then
\[
2y^{-1}\ell^2L\alpha_I^2
\geq y^{-1}\ell(L+\ell)\alpha_I^2
>\frac{1}{2C_1},
\]
where the first inequality uses \(L,\ell\geq1\). Since
\(U_2(W_0)\geq1\), set \(C_5:=\max\{C_4,4C_1\}\). 
Thus, for all \(y>0\),
\begin{equation}
\pr\biggl(\max_{I-1\leq m\leq J}V_m\geq y
\,\biggm|\,\F_0\biggr)
\leq
C_5 U_2(W_0)y^{-1}\ell^2L\alpha_I^2.
\label{eq:V-max-probability}
\end{equation}

Apply Lemma~\ref{lemma:conditional-levy-remainder} with
\(n=L\), \(\mathcal G_q:=\F_{I+q-1}\) for \(q=0,\ldots,L\),
\(\mathcal H:=\F_0\), and
\(X_q:=\alpha_{I+q-1}\xi_{I+q-1,\ell}^{\mathrm{cent}}\) for
\(q=1,\ldots,L\). Its partial sum at index \(q\) is
\(S^{\mathrm{cent}}(I+q-1)\). Hence, for any \(x>0\),
\begin{align}
&\pr\biggl(
\max_{I-1\leq m\leq J}\lVert S^{\mathrm{cent}}(m)\rVert\geq x
\,\biggm|\,\F_0
\biggr) \leq
8x^{-2}
\E[\lVert S^{\mathrm{cent}}(J)\rVert^2\mid\F_0]
+
\pr\biggl(
\max_{I-1\leq m\leq J}V_m\geq x^2/8
\,\biggm|\,\F_0
\biggr).
\label{eq:cen-max-Levy-bound}
\end{align}
Equations \eqref{eq:cen-full-block-second-moment},
\eqref{eq:V-max-probability}, and \eqref{eq:cen-max-Levy-bound} imply that there is a positive constant \(C_6\) such that 
\begin{equation}\label{eq:cen-max-final}
\pr\biggl(
\max_{I\leq m\leq J}\lVert S^{\mathrm{cent}}(m)\rVert\geq x
\,\biggm|\,\F_0
\biggr)
\leq
C_6 U_2(W_0)x^{-2}\ell^2L\alpha_I^2.
\end{equation}

Let \(C_7\) be the constant obtained by applying Markov's inequality, the
triangle inequality, and Lemma~\ref{lemma:ec-canonical-lagged-centering}.  Then
\begin{align}
\pr\biggl(
\max_{I\leq m\leq J}\lVert S^{\mathrm{pred}}(m)\rVert\geq x
\,\biggm|\,\F_0
\biggr)
\leq{}&
x^{-1}\sum_{k=I}^{J}\alpha_k
\E[\lVert\xi_{k,\ell}^{\mathrm{pred}}\rVert\mid\F_0]\nonumber\\
\leq{}&
C_7 U_2(W_0)x^{-1}L\alpha_I
\bigl(\rho^\ell+\ell\alpha_{I-\ell}\bigr).
\label{eq:pred-max-final}
\end{align}

Choose \(C\) large enough to dominate \(4C_6\) and \(2C_7\).
Combining \eqref{eq:cen-max-final} and
\eqref{eq:pred-max-final} with threshold \(x/2\) in each term completes the proof.
\Halmos\endproof

\subsection{Block Arithmetic}
\label{subsec:ec-block-arithmetic}

\begin{lemma}
\label{lemma:block-arithmetic}
Suppose Assumption~\ref{assump:learning-rate} holds with
\(1/2<\delta\leq1\), and fix \(\rho\in(0,1)\).
Then there are positive constants \(C\) and \(R\) such that, in either block formulation, for all
\(r\geq R\),
\begin{gather*}
\frac{n_{r+1}}{n_r}\leq2,
\quad n_{r+1}-n_r\leq Cn_r^\delta,
\quad n_r\geq2\ell_{n_r}, \\
\quad \alpha_{n_r-\ell_{n_r}}\leq2^\delta\alpha_{n_r},
\quad \ell_{n_r}\leq C(1+\log n_r),
\quad \rho^{\ell_{n_r}}\leq n_r^{-\delta}.
\end{gather*}
\end{lemma}

\proof{Proof of Lemma~\ref{lemma:block-arithmetic}.}
If \(1/2<\delta<1\), then \(n_r=\lceil r^{1/(1-\delta)}\rceil\) as defined in \eqref{eq:block-definition}. 
Thus, 
\(r^{1/(1-\delta)}\leq n_r\leq r^{1/(1-\delta)}+1\), and the mean-value
theorem gives
\[
n_{r+1}-n_r
\leq (r+1)^{1/(1-\delta)}-r^{1/(1-\delta)}+1
\leq (r+1)^{\delta/(1-\delta)}/(1-\delta)+1
\leq \Bigl(\frac{2^{\delta/(1-\delta)}}{1-\delta}+1\Bigr)n_r^\delta.
\]
Dividing by \(n_r\) also shows that \(n_{r+1}/n_r\to1\).

If \(\delta=1\), then \(n_r=2^r\) as defined in \eqref{eq:dyadic-block-definition}. Thus, 
\(n_{r+1}=2n_r\) and
\(n_{r+1}-n_r=n_r=n_r^\delta\).
Thus, for the applicable construction, choose \(C_0\) and \(R_0\) so that
the first two comparisons in the lemma hold for all \(r\geq R_0\).

In both cases of \(\delta\), the definition of \(\ell_k\) yields a constant
\(C_1\) such that
\[
\ell_{n_r}
\leq1+\delta\log n_r/|\log\rho|
\leq C_1(1+\log n_r),
\qquad
\rho^{\ell_{n_r}}
\leq n_r^{-\delta}.
\]
Since \((\log n_r)n_r^{-1}\to0\), there is a positive block-index threshold \(R_1\) such that, for
\(r\geq R_1\),
\(n_r\geq2\ell_{n_r}\).  Hence
\(n_r-\ell_{n_r}\geq n_r/2\), and monotonicity yields
\[
\alpha_{n_r-\ell_{n_r}}
\leq\alpha_0(n_r/2)^{-\delta}
=2^\delta\alpha_{n_r}.
\]
Taking \(C:=\max\{1,C_0,C_1\}\) and
\(R:=\max\{R_0,R_1\}\) proves the lemma.
\Halmos\endproof

\subsection{Accumulated Markovian Noise}
\label{subsec:ec-block-noise}

\begin{lemma}
\label{lemma:block-noise}
Suppose Assumptions~\ref{assump:learning-rate}--\ref{assump:growth} hold.
Fix \(1/2<\delta\leq1\), \(0\leq\beta<\delta-1/2\), and
\(c>0\). Then there are positive constants \(C\) and \(R\) such that, for all \(r\geq R\),
\begin{equation*}
\pr(\mathcal E_r^{\mathrm{markov}}\mid\F_0)
\leq C U_2(W_0)\ell_{n_r}^2n_r^{-\delta+2\beta}.
\end{equation*}
\end{lemma}

\proof{Proof of Lemma~\ref{lemma:block-noise}.}
Let 
\(
I:=n_r\), \(J:=n_{r+1}-1\),
\(L:=J-I+1=n_{r+1}-n_r\),
 \(\ell:=\ell_{n_r}\), and 
\(x:=\psi_r/8=(c/8)n_{r+1}^{-\beta}\).
Let \(C_0\) and \(R_0\) be the constants from
Lemma~\ref{lemma:block-arithmetic}.  Then, for all \(r\geq R_0\),  
\[
L\leq C_0 I^\delta,\quad
n_{r+1}\leq2I,\quad
I\geq2\ell,\quad
\rho^\ell\leq I^{-\delta},\quad
\alpha_{I-\ell}\leq2^\delta\alpha_I,
\quad
\ell\leq C_0(1+\log I).
\]
Let \(C_1\) be the constant in
Lemma~\ref{lemma:ec-canonical-block-maximal}. Then,
\[
\pr\biggl(
\max_{I\leq m\leq J}\lVert S_r^{\mathrm{markov}}(m)\rVert\geq x
\,\biggm|\,\F_0
\biggr)
\leq C_1 U_2(W_0)
\biggl[
x^{-2}\ell^2L\alpha_I^2
+x^{-1}L\alpha_I
\bigl(\rho^\ell+\ell\alpha_{I-\ell}\bigr)
\biggr].
\]
We bound the two terms separately.  Since \(n_{r+1}\leq2I\),
\(L\leq C_0I^\delta\), and
\(\alpha_I\leq\alpha_0I^{-\delta}\).  Then
\begin{align*}
x^{-2}\ell^2L\alpha_I^2
=\frac{64}{c^2}n_{r+1}^{2\beta}
\ell^2L\alpha_I^2
\leq C_2\ell^2 I^{2\beta}I^\delta I^{-2\delta}
 = C_2\ell^2 I^{-\delta+2\beta},
\end{align*}
where \(C_2:=2^{2\beta+6}c^{-2}C_0\alpha_0^2\).
Similarly, using \(\rho^\ell\leq I^{-\delta}\) and
\(\alpha_{I-\ell}\leq2^\delta\alpha_I\),
\begin{align*}
x^{-1}L\alpha_I
\bigl(\rho^\ell+\ell\alpha_{I-\ell}\bigr)
&=\frac{8}{c}n_{r+1}^\beta L\alpha_I
\bigl(\rho^\ell+\ell\alpha_{I-\ell}\bigr)\\
&\leq \frac{2^{\beta+3}}{c} C_0\alpha_0
\bigl(1+2^\delta\alpha_0\ell\bigr)I^{-\delta+\beta}\leq C_3\ell I^{-\delta+\beta}
 \leq C_3\ell^2I^{-\delta+2\beta},
\end{align*}
where \(C_3:=(2^{\beta+3}/c)C_0\alpha_0
(1+2^\delta\alpha_0)\), and  the last two inequalities use \(\ell\geq1\) and \(I^\beta\geq1\).
Combining the two estimates and setting \(C:=C_1(C_2+C_3)\) and \(R:=R_0\) complete the proof.
\Halmos\endproof

\subsection{Block Entrance}
\label{subsec:ec-poly-block-entrance}

\begin{lemma}
\label{lemma:block-entrance}
Suppose Assumptions~\ref{assump:learning-rate}--\ref{assump:qsm} hold.
Fix \(1/2<\delta\leq1\), \(0\leq\beta<\delta-1/2\), and
\(c>0\). When \(\delta=1\), also suppose
\(2\zeta\alpha_0>1\). Then there are positive constants \(C\) and \(R\) such that, for all
\(r\geq R\),
\begin{equation*}
\pr(\mathcal E_r^{\mathrm{ent}}\mid\F_0)
\leq C U_2(W_0)(1+\log n_r)n_r^{-\delta+2\beta}.
\end{equation*}
\end{lemma}

\proof{Proof of Lemma~\ref{lemma:block-entrance}.}
Let \(C_0,K_0\) be the constants from Proposition~\ref{prop:expectation}, and let \(R_0\) be a threshold
from Lemma~\ref{lemma:block-arithmetic}.  Choose \(R\geq R_0\) so that \(n_r\geq K_0\) for all
\(r\geq R\).
In particular, \(n_{r+1}\leq2n_r\).  Applying Markov's inequality and Proposition~\ref{prop:expectation} yields 
\begin{align*}
\pr(\mathcal E_r^{\mathrm{ent}}\mid\F_0)
= \pr(\lVert\theta_{n_r}-\theta^\star\rVert \geq \psi_r/8)
\leq
64c^{-2}n_{r+1}^{2\beta}
\E[\lVert\theta_{n_r}-\theta^\star\rVert^2\mid\F_0]
\leq
C U_2(W_0)(1+\log n_r)n_r^{-\delta+2\beta},
\end{align*}
where \(C:=2^{2\beta+6}c^{-2}C_0.\)
\Halmos\endproof

\subsection{Accumulated Mean Field}
\label{subsec:ec-poly-block-drift}

\begin{lemma}
\label{lemma:block-drift}
Suppose Assumptions~\ref{assump:learning-rate}--\ref{assump:qsm} hold.
Fix \(1/2<\delta\leq1\), \(0\leq\beta<\delta-1/2\), and
\(c>0\). When \(\delta=1\), also suppose
\(2\zeta\alpha_0>1\). Then there are positive constants \(C\) and 
\(R\) such that, for all \(r\geq R\),
\begin{equation*}
\pr(\mathcal E_r^{\mathrm{field}}\mid\F_0)
\leq C U_2(W_0)(1+\log n_r)n_r^{-\delta+2\beta}.
\end{equation*}
\end{lemma}
\proof{Proof of Lemma~\ref{lemma:block-drift}.}
Let 
\(
I:=n_r\), \(J:=n_{r+1}-1\),
\(L:=J-I+1=n_{r+1}-n_r\),
 \(\ell:=\ell_{n_r}\), and 
\(x:=\psi_r/8=(c/8)n_{r+1}^{-\beta}\).
Let \(C_0,K_0\) be the constants in
Proposition~\ref{prop:expectation}.
Let \(C_1,R_0\) be the constants in 
Lemma~\ref{lemma:block-arithmetic}.
Choose \(R\geq R_0\) so that \(I-\ell\geq K_0\) for all \(r\geq R\).

Fix \(r\geq R\).  For \(I\leq k\leq J\), 
Lemma~\ref{lemma:block-arithmetic} implies 
\( 
I/2\leq k-\ell \leq2I\). 
Consequently,
\(
(k-\ell)^{-\delta}\leq2^\delta I^{-\delta}\), and 
\(
1+\log(k-\ell)\leq2(1+\log I).
\)
Proposition~\ref{prop:expectation} therefore implies, uniformly over
\(I\leq k\leq J\),
\begin{equation}\label{eq:MSE-unif-bound}
\E[\lVert\theta_{k-\ell}-\theta^\star\rVert^2\mid\F_0]
\leq
C_2U_2(W_0)(1+\log I)I^{-\delta},
\end{equation}
where \(C_2:=2^{\delta+1}C_0\).

Let \(C_3\) be the constant in
Lemma~\ref{lemma:ec-stationary-mean-field-lipschitz}.
Since \(\bar F(\theta^\star)=0\), we have 
\(
\lVert\bar F(\theta_{k-\ell})\rVert
\leq C_3\lVert\theta_{k-\ell}-\theta^\star\rVert.
\)
Hence, by Cauchy--Schwarz,
\begin{align*}
\max_{I\leq m\leq J}
\lVert S_r^{\mathrm{field}}(m)\rVert^2
& = \max_{I\leq m\leq J} \biggl\lVert \sum_{k=n_r}^{m}
\alpha_k\bar F(\theta_{k-\ell_{n_r}}) \biggr\rVert \\ 
& \leq
C_3^2
\biggl(
\sum_{k=I}^{J}
\alpha_k\lVert\theta_{k-\ell}-\theta^\star\rVert
\biggr)^2
\leq
C_3^2L\sum_{k=I}^{J}
\alpha_k^2\lVert\theta_{k-\ell}-\theta^\star\rVert^2.
\end{align*}
Taking conditional expectations and using the uniform MSE bound \eqref{eq:MSE-unif-bound} gives
\begin{align*}
\E\biggl[
\max_{I\leq m\leq J}
\lVert S_r^{\mathrm{field}}(m)\rVert^2
\,\biggm|\,\F_0
\biggr]
&\leq
C_2C_3^2U_2(W_0)(1+\log I)I^{-\delta}
L\sum_{k=I}^{J}\alpha_k^2\\
&\leq
C_4U_2(W_0)(1+\log I)I^{-\delta},
\end{align*}
where
\(
C_4:=C_2C_3^2C_1^2\alpha_0^2,
\)
and the second inequality holds because 
\(
L\sum_{k=I}^{J}\alpha_k^2
\leq L^2\alpha_I^2
\leq C_1^2\alpha_0^2.
\)

Finally, Markov's inequality,
\(x^{-2}=64c^{-2}n_{r+1}^{2\beta}\), and
\(n_{r+1}\leq2I\) yield
\begin{align*}
\pr(\mathcal E_r^{\mathrm{field}}\mid\F_0)
\leq
x^{-2}
\E\biggl[
\max_{I\leq m\leq J}
\lVert S_r^{\mathrm{field}}(m)\rVert^2
\,\biggm|\,\F_0
\biggr]\leq
2^{2\beta+6}c^{-2}C_4
U_2(W_0)(1+\log I)I^{-\delta+2\beta}.
\end{align*}
Setting \(C:=2^{2\beta+6}c^{-2}C_4\) and recalling
\(I=n_r\) completes the proof.
\Halmos\endproof

\subsection{Parameter-Lag Remainder}
\label{subsec:ec-poly-block-adaptation}

\begin{lemma}
\label{lemma:block-adaptation}
Suppose Assumptions~\ref{assump:learning-rate}, \ref{assump:compact},
\ref{assump:ergodic}\ref{part:ergodic-reference}, \ref{assump:growth}, and
\ref{assump:interior-target} hold. Fix \(1/2<\delta\leq1\),
\(0\leq\beta<\delta-1/2\), and \(c>0\). Then there are positive constants \(C\) and
\(R\) such that, for all \(r\geq R\),
\begin{equation*}
\pr(\mathcal E_r^{\mathrm{lag}}\mid\F_0)
\leq C U_2(W_0)\ell_{n_r}n_r^{-\delta+\beta}.
\end{equation*}
\end{lemma}

\proof{Proof of Lemma~\ref{lemma:block-adaptation}.}
Write \(I:=n_r\), \(J:=n_{r+1}-1\),
\(L:=J-I+1=n_{r+1}-n_r\), and \(\ell:=\ell_{n_r}\). Set
\(x:=\psi_r/8=(c/8)n_{r+1}^{-\beta}\).
Let \(C_0,R_0\) be constants from
Lemma~\ref{lemma:block-arithmetic}.  For \(r\geq R_0\), in particular,
\(I\geq2\ell\).  By the triangle inequality and
\eqref{eq:Lipschitz-F},
\begin{equation*}
\max_{I\leq m\leq J}\lVert S_r^{\mathrm{lag}}(m)\rVert
= 
\max_{I\leq m\leq J}\biggl\lVert 
\sum_{k=n_r}^{m}
\alpha_k\bigl(
F(\theta_k,W_k)-F(\theta_{k-\ell_{n_r}},W_k)\bigr) \biggr\rVert
\leq L_F\sum_{k=I}^{J}
\alpha_kU_1(W_k)\lVert\theta_k-\theta_{k-\ell}\rVert.
\end{equation*}
Let \(C_1\) be the constant in
Lemma~\ref{lemma:ec-canonical-movement}.  Taking
\(\F_0\)-conditional expectations gives
\begin{align*}
\E\biggl[
\max_{I\leq m\leq J}\lVert S_r^{\mathrm{lag}}(m)\rVert
\,\biggm|\,\F_0
\biggr]
\leq L_FC_1U_2(W_0)
\sum_{k=I}^{J}\alpha_k
\sum_{i=k-\ell}^{k-1}\alpha_i
\leq C_2U_2(W_0)\ell L\alpha_I^2
\leq C_3U_2(W_0)\ell I^{-\delta}.
\end{align*}
The first inequality applies Lemma~\ref{lemma:ec-canonical-movement} to
each pair \((k-\ell,k)\).  For the second inequality, monotonicity and
Lemma~\ref{lemma:block-arithmetic} give
\(\sum_{i=k-\ell}^{k-1}\alpha_i
\leq\ell\alpha_{I-\ell}\leq2^\delta\ell\alpha_I\) and
\(\sum_{k=I}^{J}\alpha_k\leq L\alpha_I\); thus, set
\(C_2:=2^\delta L_FC_1\).  The third inequality uses
\(L\leq C_0I^\delta\) and
\(\alpha_I\leq\alpha_0I^{-\delta}\); thus, set
\(C_3:=C_0\alpha_0^2C_2\).

Finally, Markov's inequality and
\(x^{-1}=8c^{-1}n_{r+1}^\beta\) give
\begin{align*}
\pr(\mathcal E_r^{\mathrm{lag}}\mid\F_0)
\leq x^{-1}
\E\biggl[
\max_{I\leq m\leq J}\lVert S_r^{\mathrm{lag}}(m)\rVert
\,\biggm|\,\F_0
\biggr]
\leq 8c^{-1}n_{r+1}^\beta C_3
U_2(W_0)\ell I^{-\delta}
\leq 2^{\beta+3}c^{-1}C_3
U_2(W_0)\ell I^{-\delta+\beta},
\end{align*}
where the last inequality uses \(n_{r+1}\leq2I\). Setting
 \(C:=2^{\beta+3}c^{-1}C_3\) and \(R:=R_0\) completes the proof.
\Halmos\endproof

\subsection{Block Exit Bounds and Summation}
\label{subsec:ec-one-block-aggregation}

\begin{corollary}
\label{cor:block-failure}
Suppose Assumptions~\ref{assump:learning-rate}--\ref{assump:qsm} hold. Fix
\(1/2<\delta\leq1\), \(0\leq\beta<\delta-1/2\), and
\(c>0\). When \(\delta=1\), also suppose
\(2\zeta\alpha_0>1\). Then there are positive constants \(C\) and \(R\) such that, for all
\(r\geq R\),
\begin{equation*}
\pr\biggl(
\max_{k\in\mathcal I_r}\lVert\theta_k-\theta^\star\rVert
>c n_{r+1}^{-\beta}
\,\biggm|\,\F_0\biggr)
\leq C U_2(W_0)(1+\log n_r)^2n_r^{-\delta+2\beta}.
\end{equation*}
\end{corollary}

\proof{Proof of Corollary~\ref{cor:block-failure}.} 
Combining  Lemmas~\ref{lemma:canonical-block-containment} and \ref{lemma:block-noise}--\ref{lemma:block-adaptation} finishes the proof. 
\Halmos\endproof

\proof{Proof of Theorem~\ref{thm:trajectory}.}
Use the polynomial blocks \eqref{eq:block-definition} when
\(1/2<\delta<1\) and the dyadic blocks 
\eqref{eq:dyadic-block-definition} when \(\delta=1\).  For an integer
\(k_0\geq n_1\), let \(r_0=r_0(k_0)\) be the index of the block containing
\(k_0\).  Let \(C_0\) and \(R_0\) be the bound constant and block-index
threshold, respectively, from Corollary~\ref{cor:block-failure}.  Whenever
\(r_0\geq R_0\),
Lemma~\ref{lemma:blockwise-tube-reduction}, Corollary~\ref{cor:block-failure}, and the union
bound yield
\begin{align}
\pr\biggl(
\lVert\theta_k-\theta^\star\rVert>c k^{-\beta}
\text{ for some }k\geq k_0
\,\biggm|\,\F_0
\biggr)
\leq
C_0 U_2(W_0)\sum_{r\geq r_0}
n_r^{-\delta+2\beta}(1+\log n_r)^2.
\label{eq:trajectory-common-block-tail}
\end{align}

It remains to bound this deterministic series in terms of \(k_0\).

Suppose first that \(1/2<\delta<1\).  Define
\[
p:=\frac{\delta-2\beta}{1-\delta}
=1+\frac{2(\delta-\beta)-1}{1-\delta}>1.
\]
Because \(r^{1/(1-\delta)}\leq n_r\leq
2r^{1/(1-\delta)}\) and \(\delta-2\beta>0\), there is a positive constant \(C_1\) such that, for all \(r\geq1\),
\[
n_r^{-\delta+2\beta}(1+\log n_r)^2
\leq C_1 r^{-p}(1+\log r)^2.
\]
Choose a block-index threshold \(R_1\geq\max\{R_0,1\}\) so that
\(r^{-p}(1+\log r)^2\) is decreasing for \(r\geq R_1\).

The relations \(n_{r_0}\leq k_0<n_{r_0+1}\) and the polynomial block definition \eqref{eq:block-definition} imply
\[
k_0\leq n_{r_0+1}-1<(r_0+1)^{1/(1-\delta)} \quad \mbox{and}\quad 
r_0^{1/(1-\delta)}\leq n_{r_0}\leq k_0.
\]
Thus, \(k_0^{1-\delta}-1<r_0\leq k_0^{1-\delta}\). 
In particular, \(r_0(k_0)\to\infty\).  Choose an iteration threshold
\(K_0\geq n_1\) so that 
\(r_0(k_0)\geq R_1\) and \(k_0^{1-\delta}\geq2\) for all \(k_0\geq K_0\).  Therefore,
\[
\frac12k_0^{1-\delta}<r_0\leq k_0^{1-\delta}\leq k_0.
\]

The integral comparison gives a constant \(C_2\), independent of \(k_0\) and
absorbing \(C_1\), for which the following holds whenever \(k_0\geq K_0\)
(and hence \(r_0(k_0)\geq R_1\)):
\begin{align*}
\sum_{r\geq r_0}
n_r^{-\delta+2\beta}(1+\log n_r)^2
&\leq C_2 r_0^{1-p}(1+\log r_0)^2\\
&\leq
2^{p-1}C_2
k_0^{(1-\delta)(1-p)}(1+\log k_0)^2\\
&=
2^{p-1}C_2
k_0^{1-2(\delta-\beta)}(1+\log k_0)^2.
\end{align*}
The second inequality uses \(1-p<0\): the lower bound on \(r_0\) controls
\(r_0^{1-p}\), while \(r_0\leq k_0^{1-\delta}\leq k_0\) gives
\(1+\log r_0\leq1+\log k_0\).

Now suppose \(\delta=1\), so \(0\leq\beta<1/2\).  Set
\(q:=1-2\beta>0\) and
\[
A_q:=\sum_{i\geq0}(1+i)^2 2^{-qi}<\infty.
\]
Since \(n_r = 2^r\),  \(n_{r_0}\leq k_0<n_{r_0+1}=2n_{r_0}\)  for all 
\(k_0\geq2\). Hence, 
\[
\frac{k_0}{2}<n_{r_0}\leq k_0 
\quad\mbox{and}\quad
\log_2(k_0/2)<r_0\leq\log_2 k_0.
\]
Thus \(r_0(k_0)\to\infty\).  Choose an iteration
threshold \(K_1\geq2\) so that \(r_0(k_0)\geq R_0\) whenever
\(k_0\geq K_1\).
Since \(n_{r_0+i}=2^i n_{r_0}\) and
\(1+\log n_{r_0+i}\leq(1+i)(1+\log n_{r_0})\), writing
\(r=r_0+i\) gives
\begin{align*}
\sum_{r\geq r_0}(1+\log n_r)^2n_r^{-q}
&=\sum_{i\geq0}(1+\log n_{r_0+i})^2n_{r_0+i}^{-q}\\
&\leq
(1+\log n_{r_0})^2n_{r_0}^{-q}
 A_q\\
&\leq
2^q A_q(1+\log k_0)^2k_0^{-q},
\end{align*}
where \(q=2(\delta-\beta)-1\) and the second inequality uses \(k_0/2<n_{r_0}\leq k_0\).

Set \((C,K):= (2^{p-1}C_0C_2,K_0)\) if \(1/2<\delta<1\) and set \((C,K):=(2^qC_0A_q,K_1)\) if \(\delta=1\). 
Inserting the corresponding deterministic tail estimate into
\eqref{eq:trajectory-common-block-tail} completes the proof.
\Halmos\endproof

\section{Sharpness of the Polynomial Exponent}
\label{sec:EC-sharpness-proofs}

Fix the step-size parameters \(\alpha_0,\delta\) and tube parameters
\(\beta,c\) as in Theorem~\ref{thm:trajectory}. Also fix dimensions
\(d,d'\geq1\), a target \(\theta^\star\in\R^d\), a projection-ball radius
\(\kappa>c\), and a QSM coefficient
\(\zeta>0\), with \(2\zeta\alpha_0>1\) when \(\delta=1\). Set
\(
\Theta:=\{\theta\in\R^d:\|\theta-\theta^\star\|\leq\kappa\}\) and 
\(\W:=\R^{d'}\). 

For Assumption~\ref{assump:ergodic}, fix a positive integer \(N\), a
minorization coefficient in \((0,1]\), and constants \(0\leq\lambda<1\)
and \(b\geq2\).
For Assumptions~\ref{assump:kernel-Lipschitz} and~\ref{assump:growth},
fix
\(L_P>0\), \(L_F\geq\zeta\), and 
\(C_F\geq\max\{1,\zeta\kappa\}\).

Let \(\mu_{\mathrm{init}}\) be any prescribed probability distribution
on \(\Theta\times\W\) satisfying the moment condition
\(\int_{\Theta\times\W}U_2(w)\,
\mu_{\mathrm{init}}(\dd\theta,\dd w)<\infty\).

The class \(\mathfrak M\) consists of all data-generating processes
following the SA recursion~\eqref{eq:update-scheme-projection} with these
spaces, target, step
sizes, and initial distribution, satisfying
Assumptions~\ref{assump:learning-rate}--\ref{assump:qsm} with the prescribed
numerical constants.
The update map \(F\), the kernel family \(\{P_\theta\}\), and the
minorization set and measure may vary across processes. The prescribed parameters,
numerical constants, and joint initial distribution do not depend on
\(\epsilon\) or \(k_0\).

\begin{theorem}
\label{thm:ec-sharpness-fixed-class}
For the class \(\mathfrak M\) defined above and any \(\epsilon>0\),
there are positive constants \(C_\epsilon\) and \(K_\epsilon\)
such that, for all \(k_0\geq K_\epsilon\),
\begin{equation}\label{eq:ec-sharpness-worst-case-lower-bound}
\operatorname*{ess\,sup}_{\mathcal M\in\mathfrak M}
\pr_{\mathcal M}\Bigl(
\exists k\geq k_0:
\|\theta_k-\theta^\star\|>c k^{-\beta}
\,\Bigm|\,\F_0
\Bigr)
\geq
C_\epsilon
k_0^{-\left(2(\delta-\beta)-1+\epsilon\right)}.
\end{equation}
Furthermore, there is a single process \(\mathcal M_\epsilon\in\mathfrak M\) under which the conditional exit probability satisfies the same lower bound for all \(k_0\geq K_\epsilon\).
\end{theorem}

\proof{Proof of Theorems~\ref{thm:sharpness}
and~\ref{thm:ec-sharpness-fixed-class}.}

We first consider the case \(d=d'=1\). For each \(\epsilon>0\), let \(\mu_\epsilon\) be the
distribution of a signed Pareto random variable \(Z\) satisfying
\begin{align}
\pr(Z>x)
=
\pr(Z<-x)
=
\frac12
\left(\frac{x_\epsilon}{x}\right)^{2+\epsilon/(\delta-\beta)},
\qquad x\geq x_\epsilon
:=
\sqrt{\frac{\epsilon}{2(\delta-\beta)+\epsilon}}.
\label{eq:sharpness-pareto-law}
\end{align}
Then \(\E[Z]=0\) and \(\E[Z^2]=1\). Thus, the constructed noise
distributions have unit variance for all \(\epsilon>0\), while the Pareto index
\(2+\epsilon/(\delta-\beta)\) can approach the finite-variance boundary
as \(\epsilon\downarrow0\).

Consider the scalar SA recursion with
\(F(\theta,w)=-\zeta(\theta-\theta^\star)+w\) and
\(P_\theta(w,\cdot)=\mu_\epsilon(\cdot)\).
Let \((\theta_0,W_0)\) have the prescribed joint distribution
\(\mu_{\mathrm{init}}\), and let \(W_1,W_2,\ldots\) be iid with
distribution \(\mu_\epsilon\), independently of \((\theta_0,W_0)\).
The recursion then becomes
\[
\theta_{k+1}
=
\Pi_\Theta\!\left[
\theta_k-\alpha_k\zeta(\theta_k-\theta^\star)+\alpha_kW_k
\right],
\]
where \(\Theta=[\theta^\star-\kappa,\theta^\star+\kappa]\), and the
mean field is \(\bar F(\theta)=-\zeta(\theta-\theta^\star)\).
Denote this process by \(\mathcal M_\epsilon\).

Fix \(\epsilon>0\) and define \(p:=2(\delta-\beta)+\epsilon>1\).
It suffices to identify \(C_\epsilon>0\) and an iteration threshold
\(K_\epsilon\) such that, for all \(k_0\geq K_\epsilon\),
\begin{equation}\label{eq:sharpness-survival-upper-bound}
\pr\left(
|\theta_k-\theta^\star|\leq c k^{-\beta}
\text{ for all }k\geq k_0
\,|\,\F_0
\right)
\leq
\exp\left(-2C_\epsilon k_0^{1-p}\right),
\end{equation}
and \(2C_\epsilon k_0^{1-p}\leq1\).
Indeed,
using \(1-e^{-x}\geq x/2\) for \(x\in[0,1]\) yields
the claimed lower bound under \(\mathcal M_\epsilon\), because
\(p-1=2(\delta-\beta)-1+\epsilon\).

For an integer \(k_0\geq1\) and \(n\geq k_0\), define
\begin{equation}\label{eq:sharpness-survival-event}
\mathcal S_n:=
\{|\theta_i-\theta^\star|\leq c i^{-\beta}
\text{ for all }i=k_0,\ldots,n\}.
\end{equation}
The bound \eqref{eq:sharpness-survival-upper-bound} will follow from the
one-step estimate
\begin{equation}\label{eq:sharpness-one-step-survival-bound}
\pr(\mathcal S_{n+1}\mid\F_0)
\leq
\pr(\mathcal S_n\mid\F_0)
\exp\left[-2C_\epsilon(p-1)n^{-p}\right],
\end{equation}
for \(n\geq k_0\geq K_\epsilon\). Indeed, assuming
\eqref{eq:sharpness-one-step-survival-bound}, it follows from the
definition of \(\mathcal S_n\) in \eqref{eq:sharpness-survival-event}
that
\begin{align*}
&\pr\left(
|\theta_k-\theta^\star|\leq c k^{-\beta}
\text{ for all }k\geq k_0
\,\bigm|\,\F_0
\right)
=\lim_{N\to\infty}\pr(\mathcal S_N\mid\F_0)\\
&\leq
\exp\biggl(
-2C_\epsilon(p-1)\sum_{n=k_0}^{\infty}n^{-p}
\biggr)
\leq
\exp\left(
-2C_\epsilon(p-1)\int_{k_0}^{\infty}x^{-p}\,\dd x
\right)
=\exp\left(-2C_\epsilon k_0^{1-p}\right),
\end{align*}
where the first inequality follows by iterating
\eqref{eq:sharpness-one-step-survival-bound} and using
\(\pr(\mathcal S_{k_0}\mid\F_0)\leq1\), and the second inequality is the
integral comparison for \(p>1\).

We next identify \(C_\epsilon\) and
\(K_\epsilon\), prove \eqref{eq:sharpness-one-step-survival-bound}, and
verify \(2C_\epsilon k_0^{1-p}\leq1\).
For \(n\geq1\), define
\begin{equation}\label{eq:sharpness-proof-quantities}
\psi_n:=c n^{-\beta},
\quad
L_n:=\frac{\psi_{n+1}+|1-\zeta\alpha_n|\psi_n}{\alpha_n},
\quad
C_0:=\frac{c(2+\zeta\alpha_0)}{\alpha_0},
\quad\mbox{and}\quad
C_\epsilon
:=\frac{(x_\epsilon/C_0)^{2+\epsilon/(\delta-\beta)}}{2(p-1)}.
\end{equation}
The threshold \(L_n\) is sufficient for an observation to cause a
next-step exit when the current iterate lies in the tube.
Since \(\psi_{n+1}\leq\psi_n\),
\(\alpha_n=\alpha_0n^{-\delta}\), and
\(|1-\zeta\alpha_n|\leq1+\zeta\alpha_0\),
\eqref{eq:sharpness-proof-quantities} implies
\begin{equation}\label{eq:sharpness-crossing-threshold-bound}
L_n\leq C_0n^{\delta-\beta}.
\end{equation}
Choose \(K_\epsilon\geq1\) such that, for all \(n\geq K_\epsilon\),
\begin{equation}\label{eq:sharpness-eventual-conditions}
C_0n^{\delta-\beta}\geq x_\epsilon
\quad\mbox{and}\quad
2C_\epsilon n^{1-p}\leq1.
\end{equation}
Such a threshold can be chosen because \(p>1\) and \(\delta-\beta>0\).

Fix \(k_0\geq K_\epsilon\) and \(n\geq k_0\).
Because
\(\Theta=\theta^\star+[-\kappa,\kappa]\), the SA recursion can be written as
\begin{equation}\label{eq:scalar-SA-recursion}
\theta_{n+1}-\theta^\star
=\Pi_{[-\kappa,\kappa]}\!\Bigl[
(1-\zeta\alpha_n)(\theta_n-\theta^\star)+\alpha_nW_n
\Bigr].
\end{equation}
On the event \(\mathcal S_n\) defined by
\eqref{eq:sharpness-survival-event}, we have
\(|\theta_n-\theta^\star|\leq\psi_n\), so
\(|W_n|>L_n\) implies
\begin{equation}\label{eq:unprojected-bound}
\bigl|(1-\zeta\alpha_n)(\theta_n-\theta^\star)+\alpha_nW_n\bigr|
\geq
\alpha_n|W_n|-|1-\zeta\alpha_n|\psi_n
>
\alpha_nL_n-|1-\zeta\alpha_n|\psi_n
=\psi_{n+1}.
\end{equation}
It then follows from \eqref{eq:scalar-SA-recursion} that
\begin{align*}
|\theta_{n+1}-\theta^\star|
&=
\min\left\{\kappa,\bigl|
(1-\zeta\alpha_n)(\theta_n-\theta^\star)+\alpha_nW_n
\bigr|\right\}
>\psi_{n+1},
\end{align*}
where the inequality follows from the fact that
\(\kappa>c\geq\psi_{n+1}\) and \eqref{eq:unprojected-bound}.
Therefore,
\[
\mathcal S_n\cap\{|W_n|>L_n\}
\subseteq
\{|\theta_{n+1}-\theta^\star|>\psi_{n+1}\}
\subseteq\mathcal S_{n+1}^{\complement},
\]
so
\(\ind_{\mathcal S_{n+1}}
\leq
\ind_{\mathcal S_n}
\ind_{\{|W_n|\leq L_n\}}\).
Here \(\mathcal S_n\in\F_{n-1}\) and \(W_n\) is independent of
\(\F_{n-1}\), including the initial pair, by construction.
Hence,
\begin{align}
\pr(\mathcal S_{n+1}\mid\F_{n-1})
\leq
\ind_{\mathcal S_n}\pr(|W_n|\leq L_n).
\label{eq:sharpness-one-step-survival-bound-2}
\end{align}
The bound \eqref{eq:sharpness-crossing-threshold-bound}, the first
condition in \eqref{eq:sharpness-eventual-conditions}, and
\eqref{eq:sharpness-pareto-law} imply that
\begin{equation}\label{eq:tail-bound}
\pr(|W_n|>L_n)
\geq
\pr(|W_n|>C_0n^{\delta-\beta})
=
\biggl(\frac{x_\epsilon}{C_0}\biggr)^{2+\epsilon/(\delta-\beta)}
n^{-\left(2(\delta-\beta)+\epsilon\right)}
=
2C_\epsilon(p-1)n^{-p},
\end{equation}
where the last identity uses the definition of \(C_\epsilon\) in
\eqref{eq:sharpness-proof-quantities}.

The tower property yields
\begin{align*}
\pr(\mathcal S_{n+1}\mid\F_0)
=\E\left[
\pr(\mathcal S_{n+1}\mid\F_{n-1})
\,\middle|\,\F_0
\right]&\leq
\pr(\mathcal S_n\mid\F_0)
\left(1-2C_\epsilon(p-1)n^{-p}\right)\\
&\leq
\pr(\mathcal S_n\mid\F_0)
\exp\left[-2C_\epsilon(p-1)n^{-p}\right],
\end{align*}
which proves \eqref{eq:sharpness-one-step-survival-bound}.
Here, the first inequality follows from
\eqref{eq:sharpness-one-step-survival-bound-2} and \eqref{eq:tail-bound},
and the second holds because \(1-x\leq e^{-x}\). Furthermore, the second
condition in \eqref{eq:sharpness-eventual-conditions} gives
\(2C_\epsilon k_0^{1-p}\leq1\).
The preceding reduction proves the lower bound for \(d=d'=1\).
The calculation uses only \(\pr(\mathcal S_{k_0}\mid\F_0)\leq1\), so
\(C_\epsilon\) and \(K_\epsilon\) do not depend on the realized initial pair.

We now extend the construction to arbitrary \(d,d'\geq1\).
Let \(e_1:=(1,0,\ldots,0)^\intercal\) be the first standard basis vector in
\(\R^d\), and let \(\nu_\epsilon\) be the distribution of
\((Z,0,\ldots,0)^\intercal\in\R^{d'}\), where \(Z\) has distribution
\(\mu_\epsilon\). For \(w\in\R^{d'}\), write \(w^{(1)}\) for its first
coordinate, and set
\(
F(\theta,w)=-\zeta(\theta-\theta^\star)+w^{(1)}e_1\) and
\(P_\theta(w,\cdot)=\nu_\epsilon(\cdot)\).
Let \((\theta_0,W_0)\) have the prescribed joint distribution
\(\mu_{\mathrm{init}}\), and let \(W_1,W_2,\ldots\) be iid with
distribution \(\nu_\epsilon\), independently of \((\theta_0,W_0)\).
Define \(\mathcal S_n\) by
\eqref{eq:sharpness-survival-event} with the Euclidean norm in place of
the absolute value, and retain \(\psi_n\) and \(L_n\) from
\eqref{eq:sharpness-proof-quantities}. On \(\mathcal S_n\),
\(|W_n^{(1)}|>L_n\) implies
\[
\begin{aligned}
\bigl\|(1-\zeta\alpha_n)(\theta_n-\theta^\star)
+\alpha_nW_n^{(1)}e_1\bigr\|
&\geq \alpha_n|W_n^{(1)}|
-|1-\zeta\alpha_n|\|\theta_n-\theta^\star\|\\
&\geq \alpha_n|W_n^{(1)}|-|1-\zeta\alpha_n|\psi_n
>\psi_{n+1}.
\end{aligned}
\]
Projection onto \(\Theta\) replaces this norm by its minimum with
\(\kappa\). Since \(\kappa>c\geq\psi_{n+1}\), we again have
\(\mathcal S_n\cap\{|W_n^{(1)}|>L_n\}
\subseteq\mathcal S_{n+1}^{\complement}\).
Hence \eqref{eq:sharpness-one-step-survival-bound-2} applies with
\(W_n^{(1)}\) in place of \(W_n\). Since \(W_n^{(1)}\) has distribution
\(\mu_\epsilon\), the Pareto-tail calculation and the iteration of the survival
bound are identical to the scalar case, and give the conditional lower
bound with the same \(C_\epsilon\) and \(K_\epsilon\).

For each \(\epsilon>0\), the constructed process belongs to
\(\mathfrak M\) and is fixed independently of \(k_0\) and the realized
initial pair. These processes share the prescribed numerical assumption constants across
\(\epsilon\). Each conditional exit probability satisfies
\eqref{eq:sharpness-failure-lower-bound}, proving
Theorem~\ref{thm:sharpness}. Since this probability is bounded above by
the essential supremum in
\eqref{eq:ec-sharpness-worst-case-lower-bound}, both assertions of
Theorem~\ref{thm:ec-sharpness-fixed-class} follow.
\Halmos\endproof

\section{MSE for the Extended SA Recursion}

Throughout this section, the iterates follow the extended SA recursion
\eqref{eq:extension-recursion}.  

\subsection{Extended Moments and Parameter Changes}
\label{subsec:ec-extension-moments-movement}

For \(p\in\{1,2\}\), Lemma~\ref{lemma:derived-U1-drift} and
\eqref{eq:extension-post-update-transition} give, for all \(k\geq0\),
\[
\E[U_p(W_{k+1})\mid\mathscr G_k^+]
\leq\lambda_pU_p(W_k)+b_p,
\]
where \(0\leq\lambda_p<1\).  Iterating by the tower property yields, for
\(0\leq t<k\),
\begin{equation}\label{eq:extension-post-update-moment}
\E[U_p(W_k)\mid\mathscr G_t^+]
\leq
\lambda_p^{k-t}U_p(W_t)
+\frac{b_p}{1-\lambda_p}(1-\lambda_p^{k-t}).
\end{equation}
Setting \(t=0\) in \eqref{eq:extension-post-update-moment} gives the required
bounds for \(k\geq1\).  At \(k=0\), the variable \(W_0\) is
\(\mathscr G_0^+\)-measurable, and \(U_1\leq2U_2\) and
\(U_1^2\leq2U_2\) give the same bounds.  Hence there is a positive constant \(C\) such
that
\begin{equation}\label{eq:EC-structured-post-moments}
\begin{aligned}
\sup_{k\geq0}\E[U_1(W_k)\mid\mathscr G_0^+]
&\leq C U_2(W_0),
&
\sup_{k\geq0}\E[U_2(W_k)\mid\mathscr G_0^+]
&\leq C U_2(W_0),\\
\sup_{k\geq0}\E[U_1(W_k)^2\mid\mathscr G_0^+]
&\leq C U_2(W_0).
\end{aligned}
\end{equation}

For \(k\geq1\), the inclusions in \eqref{eq:extension-interlacing}, the
centering identity in
Assumption~\ref{assump:structured-errors}\textup{(i)}, displayed in
\eqref{eq:extension-martingale-conditions}, and the tower property give
\[
\E[M_{k+1}\mid\mathscr G_{k-1}^+]
=\E\!\left[
\E[M_{k+1}\mid\mathscr G_k^-]
\,\middle|\,\mathscr G_{k-1}^+
\right]=0.
\]
Assumption~\ref{assump:structured-errors}\textup{(i)} also makes
\(M_{k+1}\) \(\mathscr G_k^+\)-measurable.  Hence
\(
(\sum_{i=1}^n\alpha_iM_{i+1},\mathscr G_n^+)_{n\geq1}
\)
is a martingale.

Corollary~\ref{corollary:lagged-update-field-comparison} also requires bounds
on parameter changes.  The following lemma incorporates the
martingale-difference noise and predictable bias.

\begin{lemma}
\label{lemma:ec-extension-movement}
Suppose Assumptions~\ref{assump:learning-rate}, \ref{assump:compact},
\ref{assump:ergodic}\ref{part:ergodic-reference}, and
\ref{assump:growth} hold, together with
parts~\textup{(i)}--\textup{(ii)} of
Assumption~\ref{assump:structured-errors}.
Then there is a positive constant \(C\) such that, for all \(0\leq u<v\),
\begin{equation}\label{eq:extension-movement-budget}
\E[U_1(W_v)\lVert\theta_v-\theta_u\rVert\mid\mathscr G_0^+]
\leq C U_2(W_0)\sum_{i=u}^{v-1}
\alpha_i(1+\SFm_i+\SFb_i).
\end{equation}
\end{lemma}

\proof{Proof of Lemma~\ref{lemma:ec-extension-movement}.}
Fix integers \(1\leq u<v\).  Choose \(C_0\) so that
\eqref{eq:EC-structured-post-moments} holds with \(C_0\).  For \(i\geq1\),
the inclusions in \eqref{eq:extension-interlacing} give
\(\mathscr G_0^+\subseteq\mathscr G_i^-\).  The tower property and the
conditional second-moment bound in
Assumption~\ref{assump:structured-errors}\textup{(i)},
displayed in \eqref{eq:extension-martingale-conditions}, yield
\[
\E[\lVert M_{i+1}\rVert^2\mid\mathscr G_0^+]
\leq\SFm_i^2\E[U_2(W_i)\mid\mathscr G_0^+]
\leq C_0U_2(W_0)\SFm_i^2.
\]
For each \(i=u,\ldots,v-1\), \(\Pi_\Theta\) is nonexpansive by
Assumption~\ref{assump:compact}, while the linear-growth bound
\eqref{eq:LinearGrowth-F} in Assumption~\ref{assump:growth} yields
\[
\lVert\theta_{i+1}-\theta_i\rVert
\leq\alpha_i\bigl(
C_FU_1(W_i)+\lVert M_{i+1}\rVert+\lVert B_i\rVert
\bigr).
\]
After multiplying this bound by \(U_1(W_v)\), summing over \(i\), and taking
conditional expectations, we obtain
\begin{align}
&\E[U_1(W_v)\lVert\theta_v-\theta_u\rVert
\mid\mathscr G_0^+]\nonumber\\
\leq{}&
\sum_{i=u}^{v-1}\alpha_i\biggl(
C_F\E[U_1(W_v)U_1(W_i)\mid\mathscr G_0^+]
+\E[U_1(W_v)\lVert M_{i+1}\rVert\mid\mathscr G_0^+]
+\E[U_1(W_v)\lVert B_i\rVert\mid\mathscr G_0^+]
\biggr)\nonumber\\
\leq{}&
C_0\max\{C_F,1\}U_2(W_0)
\sum_{i=u}^{v-1}\alpha_i(1+\SFm_i+\SFb_i),
\label{eq:extension-weighted-movement-decomposition}
\end{align}
The last inequality follows from Cauchy--Schwarz,
\eqref{eq:EC-structured-post-moments}, and the conditional second-moment bound
in Assumption~\ref{assump:structured-errors}\textup{(i)}, displayed in
\eqref{eq:extension-martingale-conditions}; part~\textup{(ii)} of that
assumption bounds \(\lVert B_i\rVert\).

The case \(u=0\) includes the \(\mathscr G_0^+\)-measurable first update.
Assumption~\ref{assump:compact} and the definition of \(d_\Theta\) control
this first parameter change directly:
\(\lVert\theta_1-\theta_0\rVert\leq d_\Theta\), and hence, for \(v\geq1\),
\[
\E[U_1(W_v)\lVert\theta_1-\theta_0\rVert\mid\mathscr G_0^+]
\leq C_0d_\Theta U_2(W_0)
\leq\frac{C_0d_\Theta}{\alpha_0}U_2(W_0)
\alpha_0(1+\SFm_0+\SFb_0).
\]
For \(v>1\), split the parameter change at \(\theta_1\) and apply
\eqref{eq:extension-weighted-movement-decomposition} with \(u=1\).  The case
\(v=1\) is covered by the last display.  Thus
\eqref{eq:extension-movement-budget} holds for all
\(0\leq u<v\) with
\(C:=C_0\max\{C_F,1,d_\Theta/\alpha_0\}\).
\Halmos\endproof

\subsection{Proof of Proposition~\ref{thm:structured-mse}}

\begin{lemma}
\label{lemma:ec-extension-lagged-restoring-drift}
Suppose Assumptions~\ref{assump:learning-rate}--\ref{assump:qsm} hold,
together with parts~\textup{(i)}--\textup{(ii)} of
Assumption~\ref{assump:structured-errors}.  Then there is a positive constant \(C\) such that, for all \(k>\ell\geq1\),
\begin{equation}\label{eq:shared-lagged-mean-field}
\E[\langle\theta_k-\theta^\star,F(\theta_k,W_k)\rangle
\mid\mathscr G_0^+]\leq
-\zeta\E[\lVert\theta_k-\theta^\star\rVert^2\mid\mathscr G_0^+]
+C U_2(W_0)\biggl(
\rho^\ell+\sum_{i=k-\ell}^{k-1}
\alpha_i(1+\SFm_i+\SFb_i)
\biggr).
\end{equation}
\end{lemma}

\proof{Proof of Lemma~\ref{lemma:ec-extension-lagged-restoring-drift}.}
Fix integers \(k>\ell\geq1\).  Decompose
\begin{align}
\langle\theta_k-\theta^\star,F(\theta_k,W_k)\rangle 
={}&
\langle\theta_{k-\ell}-\theta^\star,
F(\theta_{k-\ell},W_k)-\bar F(\theta_{k-\ell})\rangle \nonumber\\
&+\langle\theta_{k-\ell}-\theta^\star,
F(\theta_k,W_k)-F(\theta_{k-\ell},W_k)\rangle \nonumber\\
&+\langle\theta_{k-\ell}-\theta^\star,
\bar F(\theta_{k-\ell})-\bar F(\theta_k)\rangle \nonumber\\
&+\langle\theta_k-\theta_{k-\ell},
F(\theta_k,W_k)-\bar F(\theta_k)\rangle
+\langle\theta_k-\theta^\star,\bar F(\theta_k)\rangle.
\label{eq:extension-lagged-restoring-decomposition}
\end{align}
Assumption~\ref{assump:qsm} bounds the last term by
\(-\zeta\lVert\theta_k-\theta^\star\rVert^2\).  We next control the first
term using Corollary~\ref{corollary:lagged-update-field-comparison} and the
middle three terms using Lemma~\ref{lemma:ec-extension-movement}.

Under \eqref{eq:extension-post-update-transition}, the proof of
Corollary~\ref{corollary:lagged-update-field-comparison} applies with
\((\mathscr G_j^+)_{j\geq0}\) in place of \((\F_j)_{j\geq0}\).
Equations~\eqref{eq:extension-recursion} and
\eqref{eq:extension-interlacing} make \(\theta_{k-\ell}\)
\(\mathscr G_{k-\ell}^+\)-measurable.
Assumption~\ref{assump:interior-target} gives \(\theta^\star\in\Theta\), so the diameter
bound associated with Assumption~\ref{assump:compact} gives
\(\lVert\theta_{k-\ell}-\theta^\star\rVert\leq d_\Theta\).  Applying the
corollary conditionally on \(\mathscr G_{k-\ell}^+\) and then using the tower
property gives a constant \(C_0\) such that
\begin{align*}
&\bigl|\E[\langle\theta_{k-\ell}-\theta^\star,
F(\theta_{k-\ell},W_k)-\bar F(\theta_{k-\ell})\rangle
\mid\mathscr G_0^+]\bigr|\\
&\quad\leq
C_0\rho^\ell\E[U_1(W_{k-\ell})\mid\mathscr G_0^+]
+C_0\sum_{s=1}^{\ell-1}\rho^{\ell-s-1}
\E[U_1(W_{k-\ell+s})
\lVert\theta_{k-\ell+s}-\theta_{k-\ell}\rVert
\mid\mathscr G_0^+]\\
&\quad\leq
C_0 U_2(W_0)\biggl(
\rho^\ell+\sum_{i=k-\ell}^{k-1}\alpha_i(1+\SFm_i+\SFb_i)
\biggr).
\end{align*}
Equation~\eqref{eq:EC-structured-post-moments} and
Lemma~\ref{lemma:ec-extension-movement} imply the last inequality: each
partial sum bounding a parameter change is bounded by the full sum from \(k-\ell\) to \(k-1\), and the
geometric weights sum to at most \((1-\rho)^{-1}\).

Lemmas~\ref{lemma:ec-stationary-mean-field-lipschitz}
and~\ref{lemma:update-field-bounds} give
\(\lVert\bar F(\theta)-\bar F(\theta')\rVert
\lesssim\lVert\theta-\theta'\rVert\) and
\(\lVert\bar F(\theta)\rVert\lesssim1\) for
\(\theta,\theta'\in\Theta\).  Combining these conclusions with the two bounds
in Assumption~\ref{assump:growth}, the diameter bound from
Assumption~\ref{assump:compact}, and \(U_1\geq1\) bounds the absolute value of
the sum of the middle three terms in
\eqref{eq:extension-lagged-restoring-decomposition} by a constant times
\(U_1(W_k)\lVert\theta_k-\theta_{k-\ell}\rVert\).
Lemma~\ref{lemma:ec-extension-movement} therefore gives a constant \(C_1\)
that bounds the \(\mathscr G_0^+\)-conditional expectation of this absolute
value by
\[
C_1 U_2(W_0)\sum_{i=k-\ell}^{k-1}
\alpha_i(1+\SFm_i+\SFb_i).
\]
Combining these estimates in
\eqref{eq:extension-lagged-restoring-decomposition} proves
\eqref{eq:shared-lagged-mean-field} with \(C:=C_0+C_1\).
\Halmos\endproof

\proof{Proof of Proposition~\ref{thm:structured-mse}.}
Define
\(x_k:=\E[\lVert\theta_k-\theta^\star\rVert^2\mid\mathscr G_0^+]\).
For \(k\geq1\), the projection is nonexpansive by
Assumption~\ref{assump:compact}, while
Assumption~\ref{assump:interior-target} gives
\(\theta^\star\in\Theta\) and hence
\(\Pi_\Theta(\theta^\star)=\theta^\star\).  Equations
\eqref{eq:extension-recursion} and
\eqref{eq:extension-interlacing}, together with
Assumption~\ref{assump:structured-errors}\textup{(ii)}, make
\(\theta_k\), \(W_k\), and \(B_k\) \(\mathscr G_k^-\)-measurable.  The
centering and second-moment bounds in
Assumption~\ref{assump:structured-errors}\textup{(i)} then eliminate the
conditional expectations of the terms linear in \(M_{k+1}\) and control its
quadratic term.  Consequently,
\begin{equation}\label{eq:EC-structured-conditional-step}
\E[\lVert\theta_{k+1}-\theta^\star\rVert^2\mid\mathscr G_k^-]
\leq
\bigl\lVert\theta_k-\theta^\star
+\alpha_k\bigl(F(\theta_k,W_k)+B_k\bigr)\bigr\rVert^2
+\alpha_k^2\SFm_k^2U_2(W_k).
\end{equation}

When \(0<\delta<1\), Young's inequality and
Assumption~\ref{assump:structured-errors}\textup{(ii)} give
\(2\langle\theta_k-\theta^\star,B_k\rangle
\leq\zeta\lVert\theta_k-\theta^\star\rVert^2
+\zeta^{-1}\SFb_k^2\).  When \(\delta=1\), the proposition hypothesis
\(2\zeta\alpha_0>1\) implies
\(\zeta-1/(2\alpha_0)>0\).  Young's inequality and
Assumption~\ref{assump:structured-errors}\textup{(ii)} then give
\(2\langle\theta_k-\theta^\star,B_k\rangle
\leq(\zeta-1/(2\alpha_0))\lVert\theta_k-\theta^\star\rVert^2
+(\zeta-1/(2\alpha_0))^{-1}\SFb_k^2\).

The linear-growth bound \eqref{eq:LinearGrowth-F} in
Assumption~\ref{assump:growth} and
\eqref{eq:EC-structured-post-moments} imply
\(\E[\lVert F(\theta_k,W_k)\rVert^2\mid\mathscr G_0^+]
\lesssim U_2(W_0)\).  Assumption~\ref{assump:structured-errors}\textup{(ii)}
bounds \(\lVert B_k\rVert\), while
Assumption~\ref{assump:learning-rate} gives \(\alpha_k\leq\alpha_0\).
Let \(C_0\) dominate the constants in the two preceding Young inequalities,
the moment bound obtained from \eqref{eq:LinearGrowth-F} and
\eqref{eq:EC-structured-post-moments}, the constant in
Lemma~\ref{lemma:ec-extension-lagged-restoring-drift}, and the deterministic
factors obtained below from \eqref{eq:log-lag},
Assumption~\ref{assump:learning-rate}, and
Assumption~\ref{assump:structured-errors}\textup{(iii)}, displayed in
\eqref{eq:extension-power-laws}.  Iterating the inclusions in
\eqref{eq:extension-interlacing} gives
\(\mathscr G_0^+\subseteq\mathscr G_k^-\) for \(k\geq1\).  Taking
\(\mathscr G_0^+\)-conditional expectations in
\eqref{eq:EC-structured-conditional-step} and using the tower property yields
the following inequality for all \(k\geq1\) when \(0<\delta<1\):
\begin{equation}\label{eq:EC-structured-one-step-mse}
x_{k+1}
\leq(1+\zeta\alpha_k)x_k
+2\alpha_k\E[\langle\theta_k-\theta^\star,F(\theta_k,W_k)\rangle
\mid\mathscr G_0^+] 
+C_0 U_2(W_0)\bigl[
\alpha_k^2(1+\SFm_k^2)+\alpha_k\SFb_k^2
\bigr].
\end{equation}
When \(\delta=1\), the same inequality holds for all \(k\geq1\) with
\((1+\zeta\alpha_k)x_k\) replaced by
\(\bigl(1+(\zeta-1/(2\alpha_0))\alpha_k\bigr)x_k\).
The quadratic estimate
\(\lVert F(\theta_k,W_k)+B_k\rVert^2
\leq2\lVert F(\theta_k,W_k)\rVert^2+2\SFb_k^2\) and
\(\alpha_k\leq\alpha_0\) place the quadratic bias contribution on the scale
\(\alpha_k\SFb_k^2\).

Recall from \eqref{eq:log-lag} that
\(\ell_k=\max\{1,\lceil\delta\log k/|\log\rho|\rceil\}\).
Choose an iteration threshold \(K_0\geq2\) so that, for all \(k\geq K_0\),
\(k-\ell_k\geq k/2\) and the applicable contraction factor is positive:
\(\zeta\alpha_0k^{-\delta}<1\) when \(0<\delta<1\), and
\((\zeta\alpha_0+1/2)k^{-1}<1\) when \(\delta=1\).
For \(k-\ell_k\leq i\leq k-1\), we then have
\(\alpha_i\leq2^\delta\alpha_k\) by
Assumption~\ref{assump:learning-rate}.
Assumption~\ref{assump:structured-errors}\textup{(iii)}, displayed in
\eqref{eq:extension-power-laws}, gives
\(\SFm_i\leq C_{\mathsf M}k^{\omega_{\mathsf M}}\) and
\(\SFb_i\leq C_{\mathsf B}\).
Also, \(\ell_k\lesssim1+\log k\) and
\(\rho^{\ell_k}\leq k^{-\delta}=\alpha_k/\alpha_0\).  Consequently,
\begin{equation}\label{eq:EC-structured-lagged-budget-reduction}
\begin{aligned}
\alpha_k\biggl(
\rho^{\ell_k}+\sum_{i=k-\ell_k}^{k-1}
\alpha_i(1+\SFm_i+\SFb_i)
\biggr)
&\leq C_0\alpha_k^2(1+\log k)(1+k^{\omega_{\mathsf M}})\\
&\leq C_0\alpha_k^2(1+\log k)(1+k^{2\omega_{\mathsf M}}).
\end{aligned}
\end{equation}
The last inequality uses the restriction \(\omega_{\mathsf M}\geq0\) in
Assumption~\ref{assump:structured-errors}\textup{(iii)}.  The predictable
bias contributes at most \(C_0\alpha_k^2(1+\log k)\) through the bound on parameter
changes, while its direct contribution remains on the scale
\(\alpha_k\SFb_k^2\).

For \(k\geq1\),
\((k+1)^{2\omega_{\mathsf M}}
\leq2^{2\omega_{\mathsf M}}k^{2\omega_{\mathsf M}}\) and
\((k+1)^{-2\omega_{\mathsf B}}\leq k^{-2\omega_{\mathsf B}}\); the
resulting numerical factors are included in \(C_0\).

Combining \eqref{eq:EC-structured-one-step-mse} with
Lemma~\ref{lemma:ec-extension-lagged-restoring-drift},
\eqref{eq:EC-structured-lagged-budget-reduction},
Assumption~\ref{assump:learning-rate}, and
Assumption~\ref{assump:structured-errors}\textup{(iii)}, displayed in
\eqref{eq:extension-power-laws}, gives the following recursion for all
\(k\geq K_0\) when \(0<\delta<1\):
\begin{equation}\label{eq:EC-structured-mse-recursion}
\begin{aligned}
x_{k+1}
\leq{}&(1-\zeta\alpha_0k^{-\delta})x_k
+C_0U_2(W_0)\bigl[
(1+\log k)k^{-2\delta}
+(1+\log k)k^{-2\delta+2\omega_{\mathsf M}}
+k^{-\delta-2\omega_{\mathsf B}}
\bigr].
\end{aligned}
\end{equation}
When \(\delta=1\), the same recursion holds for all \(k\geq K_0\) with
\((1-\zeta\alpha_0k^{-1})x_k\) replaced by
\(\bigl(1-(\zeta\alpha_0+1/2)k^{-1}\bigr)x_k\).

We finish by induction.  The direct bias term satisfies
\(k^{-\delta-2\omega_{\mathsf B}}
\leq k^{-\delta}k^{-\min\{2\omega_{\mathsf B},\delta\}}\): equality
holds when \(2\omega_{\mathsf B}\leq\delta\), while
\(k^{-\delta-2\omega_{\mathsf B}}\leq k^{-2\delta}\) when
\(2\omega_{\mathsf B}>\delta\).  Thus, the additive term in
\eqref{eq:EC-structured-mse-recursion} satisfies
\begin{align}
&C_0U_2(W_0)\bigl[
(1+\log k)k^{-2\delta}
+(1+\log k)k^{-2\delta+2\omega_{\mathsf M}}
+k^{-\delta-2\omega_{\mathsf B}}
\bigr] \nonumber\\
\leq{}& C_0U_2(W_0)k^{-\delta}\bigl[
(1+\log k)\bigl(k^{-\delta}+k^{-\delta+2\omega_{\mathsf M}}\bigr)
+k^{-\min\{2\omega_{\mathsf B},\delta\}}
\bigr].
\label{eq:EC-structured-rate-profile}
\end{align}
The proposition hypotheses imply that the three decay exponents
\(\delta\), \(\delta-2\omega_{\mathsf M}\), and
\(\min\{2\omega_{\mathsf B},\delta\}\) belong to \((0,\delta]\).  The
elementary power comparison therefore shows that each associated power at
\(k+1\) is at least \(1-\delta/k\) times its value at \(k\).  Since
\(1+\log k\) is increasing, the same comparison holds for the complete rate
profile.  Thus, for all \(k\geq1\),
\begin{align}
&(1+\log(k+1))\bigl((k+1)^{-\delta}
+(k+1)^{-\delta+2\omega_{\mathsf M}}\bigr)
+(k+1)^{-\min\{2\omega_{\mathsf B},\delta\}} \nonumber\\
\geq{}&\left(1-\frac{\delta}{k}\right)\bigl[
(1+\log k)\bigl(k^{-\delta}+k^{-\delta+2\omega_{\mathsf M}}\bigr)
+k^{-\min\{2\omega_{\mathsf B},\delta\}}
\bigr].
\label{eq:EC-structured-rate-profile-step}
\end{align}

When \(0<\delta<1\), choose the iteration threshold \(K\geq K_0\) so that
\(\delta K^{\delta-1}\leq\zeta\alpha_0/2\).  Comparing
\eqref{eq:EC-structured-rate-profile-step} with the applicable contraction in
\eqref{eq:EC-structured-mse-recursion} leaves at least
\((\zeta\alpha_0/2)k^{-\delta}\) times the bracket on the right-hand side of
\eqref{eq:EC-structured-rate-profile}.  When \(\delta=1\), set
\(K:=K_0\).  The corresponding lower bound is
\((\zeta\alpha_0-1/2)k^{-1}\) times that bracket, and
\(\zeta\alpha_0-1/2>0\) by the proposition hypothesis.

Choose \(C_1\) so that its product with the applicable positive gap,
\(\zeta\alpha_0/2\) when \(0<\delta<1\) or
\(\zeta\alpha_0-1/2\) when \(\delta=1\), is at least \(C_0\), and so that
\(C_1\) times the bracket on the right-hand side of
\eqref{eq:EC-structured-rate-profile}, evaluated at \(K\), is at least
\(d_\Theta^2\).  We prove by induction that, for all \(k\geq K\),
\[
x_k\leq C_1U_2(W_0)\bigl[
(1+\log k)\bigl(k^{-\delta}+k^{-\delta+2\omega_{\mathsf M}}\bigr)
+k^{-\min\{2\omega_{\mathsf B},\delta\}}
\bigr].
\]
Assumption~\ref{assump:interior-target} gives \(\theta^\star\in\Theta\), so
the diameter bound associated with Assumption~\ref{assump:compact}, together
with \(U_2(W_0)\geq1\), establishes the claim at \(k=K\).  If it holds at
some \(k\geq K\), the applicable contraction factor is nonnegative by the
choice of \(K_0\), so the induction hypothesis may be substituted.  The inequalities 
\eqref{eq:EC-structured-rate-profile} and
\eqref{eq:EC-structured-rate-profile-step}, and the choice of \(C_1\) then
establish the claim at \(k+1\).

Finally,
\(k^{-\min\{2\omega_{\mathsf B},\delta\}}
\leq k^{-2\omega_{\mathsf B}}+k^{-\delta}\).  Thus, for all \(k\geq K\),
\[
x_k\leq2C_1U_2(W_0)\bigl[
(1+\log k)k^{-\delta}
+(1+\log k)k^{-\delta+2\omega_{\mathsf M}}
+k^{-2\omega_{\mathsf B}}
\bigr].
\]
This proves the proposition with \(C:=2C_1\).
\Halmos\endproof

\section{Shrinking-Tube Concentration for the Extended SA Recursion}

Throughout this section, the iterates follow the extended SA
recursion~\eqref{eq:extension-recursion}.

\subsection{Bias-Adjusted Checkpoint Control}
\label{subsec:ec-extension-trajectory}

In Proposition~\ref{thm:structured-mse}, the term of order
\(k^{-2\omega_{\mathsf B}}\) bounds the contribution of predictable bias to the mean-squared
parameter error.  To obtain the block-entrance estimate required by
Theorem~\ref{thm:structured-trajectory} under its condition
\(\beta<\omega_{\mathsf B}\), we control the squared distance beyond a
deterministic radius around the target.

For \(a\geq0\), define \(\Phi_a(e):=(\lVert e\rVert-a)_+^2\), set
\(h_a(0):=0\), and set
\(h_a(e):=(1-a/\lVert e\rVert)_+e\) for \(e\neq0\).  Then
\(\Phi_a(e)=\lVert h_a(e)\rVert^2\), and the map \(h_a\) is 1-Lipschitz.
Moreover, \(\lVert x-h_a(x)\rVert\leq a\), so the triangle inequality gives
\((\lVert x+u\rVert-a)_+\leq\lVert h_a(x)+u\rVert\).  Squaring and
expanding yields
\begin{equation}\label{eq:EC-structured-distance-smoothness}
\Phi_a(x+u)
\leq
\Phi_a(x)+2\langle h_a(x),u\rangle+\lVert u\rVert^2.
\end{equation}
Assumption~\ref{assump:interior-target} gives
\(d_\star:=\inf_{\theta\in\Theta^{\complement}}
\lVert\theta^\star-\theta\rVert>0\).

\begin{lemma}
\label{lemma:ec-extension-bias-radius}
Suppose Assumptions~\ref{assump:learning-rate}--\ref{assump:structured-errors} hold, with \(1/2<\delta\leq1\).  When
\(\delta=1\), also suppose \(2\zeta\alpha_0>1\).  Fix
\(0\leq\omega_{\mathsf M}<\delta-1/2\) and
\(0\leq\beta<\min\{\omega_{\mathsf B},
\delta-\omega_{\mathsf M}-1/2\}\).  If \(B_k=0\) almost surely for all
\(k\), drop the restriction involving \(\omega_{\mathsf B}\); if
\(M_{k+1}=0\) almost surely for all \(k\), take
\(\omega_{\mathsf M}=0\).  The deterministic bias scale in
Assumption~\ref{assump:structured-errors} may be chosen nonincreasing.  Then
there are a positive iteration threshold \(K\) and a nonnegative deterministic
sequence \(\{\mathsf g_k\}_{k\geq0}\)
such that, for all \(k\geq K\),
\begin{equation}\label{eq:EC-structured-bias-radius-recursion}
\mathsf g_{k+1}=(1-\zeta\alpha_k)\mathsf g_k+\alpha_k\SFb_k,
\end{equation}
and
\[
\mathsf g_{k+1}\leq \mathsf g_k,
\qquad
\mathsf g_k\geq\SFb_k/\zeta,
\qquad
\mathsf g_k=o(k^{-\beta}),
\qquad
\mathsf g_k<d_\star.
\]
Moreover, for all \(k\geq K\) and \(r\geq0\),
\begin{align}
(r-\mathsf g_{k+1})_+^2
-2\zeta\alpha_k r(r-\mathsf g_{k+1})_+
+2\alpha_k\SFb_k(r-\mathsf g_{k+1})_+
\leq
(1-2\zeta\alpha_k)(r-\mathsf g_k)_+^2.
\label{eq:EC-structured-radial-bias-absorption}
\end{align}
If \(B_k=0\) almost surely for all \(k\), the bias scale and radius may be
taken as \(\SFb_k=\mathsf g_k=0\).
\end{lemma}

\proof{Proof of Lemma~\ref{lemma:ec-extension-bias-radius}.}
Suppose first that \(B_k\) is nonzero with positive probability for at least
one \(k\).  Because the deterministic scale bounding \(B_k\) is not unique,
Assumption~\ref{assump:structured-errors}\textup{(ii)}--\textup{(iii)} allows
us to take \(\SFb_k=C_{\mathsf B}(k+1)^{-\omega_{\mathsf B}}\) for
\(k\geq0\).  This scale is nonincreasing and satisfies
\(\lVert B_k\rVert\leq\SFb_k\).  Choose an iteration threshold \(K_0\) so that
\(0<2\zeta\alpha_k<1\) for all \(k\geq K_0\) and
\(\SFb_{K_0}/\zeta<d_\star\).  Set
\(\mathsf g_{K_0}:=\SFb_{K_0}/\zeta\) and define \(\mathsf g_k\) recursively by
\eqref{eq:EC-structured-bias-radius-recursion} for \(k\geq K_0\).
Choose arbitrary nonnegative values for \(\mathsf g_k\), \(k<K_0\).

We first verify the monotonicity properties.  If
\(\mathsf g_k\geq\SFb_k/\zeta\), then
\[
\mathsf g_{k+1}
=\mathsf g_k-\alpha_k(\zeta\mathsf g_k-\SFb_k)
\leq\mathsf g_k,
\]
and, because \(1-\zeta\alpha_k\geq0\) and \(\SFb_{k+1}\leq\SFb_k\),
\[
\mathsf g_{k+1}
\geq
(1-\zeta\alpha_k)\SFb_k/\zeta
+\alpha_k\SFb_k
=\SFb_k/\zeta
\geq\SFb_{k+1}/\zeta.
\]
Induction therefore gives
\(\mathsf g_{k+1}\leq\mathsf g_k\),
\(\mathsf g_k\geq\SFb_k/\zeta\), and
\(\mathsf g_k\leq\mathsf g_{K_0}<d_\star\) for all \(k\geq K_0\).

It remains to compare \(\mathsf g_k\) with the tube radius.  Choose
\(\beta_0\) so that
\[
\beta<\beta_0<
\begin{cases}
\omega_{\mathsf B},&1/2<\delta<1,\\
\min\{\omega_{\mathsf B},\zeta\alpha_0\},&\delta=1.
\end{cases}
\]
Such a choice is possible because \(\beta<\omega_{\mathsf B}\) and, when
\(\delta=1\), \(\zeta\alpha_0>1/2>\beta\).  For any \(C_1>0\), the
inequality \((1+1/k)^{-\beta_0}\geq1-\beta_0/k\) gives
\[
C_1(k+1)^{-\beta_0}
-(1-\zeta\alpha_k)C_1k^{-\beta_0}
\geq
C_1k^{-\beta_0}
\bigl(\zeta\alpha_0k^{-\delta}-\beta_0/k\bigr).
\]
The expression in parentheses is bounded below by a positive constant times
\(k^{-\delta}\) for all sufficiently large \(k\).  Since
\(\SFb_k=\mathcal O(k^{-\omega_{\mathsf B}})=o(k^{-\beta_0})\), first
choose an iteration threshold \(K\geq K_0\) so that, for all \(k\geq K\),
\[
(k+1)^{-\beta_0}
\geq(1-\zeta\alpha_k)k^{-\beta_0}+\alpha_k\SFb_k.
\]
Next choose \(C_1\geq1\) so that
\(C_1K^{-\beta_0}\geq\mathsf g_K\).  Multiplying the preceding comparison by
\(C_1\) preserves it, and induction gives
\(\mathsf g_k\leq C_1k^{-\beta_0}\) for all \(k\geq K\).  Hence
\(\mathsf g_k=o(k^{-\beta})\).

We next prove the one-step comparison.  Fix \(k\geq K\) and \(r\geq0\).
The recursion gives
\(\mathsf g_k-\mathsf g_{k+1}
=\alpha_k(\zeta\mathsf g_k-\SFb_k)\geq0\).  If
\(r>\mathsf g_k\), then
\(r-\mathsf g_{k+1}=r-\mathsf g_k+\mathsf g_k-\mathsf g_{k+1}\), and
the difference between the left- and
right-hand sides of
\eqref{eq:EC-structured-radial-bias-absorption} equals
\[
-(\mathsf g_k-\mathsf g_{k+1})^2
-2\zeta\alpha_k(r-\mathsf g_k)
(\mathsf g_k-\mathsf g_{k+1})\leq0.
\]
If \(\mathsf g_{k+1}<r\leq\mathsf g_k\), then
\(0<r-\mathsf g_{k+1}\leq\mathsf g_k-\mathsf g_{k+1}\).  In this case,
the left-hand side of
\eqref{eq:EC-structured-radial-bias-absorption} is
\[
(r-\mathsf g_{k+1})
\bigl[(1-2\zeta\alpha_k)(r-\mathsf g_{k+1})
-2(1-\zeta\alpha_k)(\mathsf g_k-\mathsf g_{k+1})\bigr]
\leq-(\mathsf g_k-\mathsf g_{k+1})(r-\mathsf g_{k+1})\leq0.
\]
If \(r\leq\mathsf g_{k+1}\), both positive-part terms vanish.  This proves
the comparison.

When \(B_k=0\) almost surely for all \(k\), set
\(\mathsf g_k=\SFb_k=0\) and choose an iteration threshold \(K\) so that
\(0<2\zeta\alpha_k<1\) for all \(k\geq K\).  All conclusions then hold
directly.
\Halmos\endproof

For the remainder of the proof of Theorem~\ref{thm:structured-trajectory},
fix the sequence \(\{\mathsf g_k\}_{k\geq0}\) provided by
Lemma~\ref{lemma:ec-extension-bias-radius} and the corresponding choice of
the deterministic bias scale \(\{\SFb_k\}_{k\geq0}\).

\begin{lemma}
\label{lemma:ec-extension-radius-restoring-drift}
Suppose Assumptions~\ref{assump:learning-rate}--\ref{assump:qsm} hold,
together with parts~\textup{(i)}--\textup{(ii)} of
Assumption~\ref{assump:structured-errors}.  Then there is a positive constant \(C\),
independent of \(a\geq0\), such that, for all \(k>\ell\geq1\),
\begin{align}
&\E[\langle h_a(\theta_k-\theta^\star),F(\theta_k,W_k)\rangle
\mid\mathscr G_0^+]
\nonumber\\
\leq{}&
-\zeta\E\bigl[
\lVert\theta_k-\theta^\star\rVert
(\lVert\theta_k-\theta^\star\rVert-a)_+
\mid\mathscr G_0^+\bigr]
+C U_2(W_0)\biggl(
\rho^\ell+\sum_{i=k-\ell}^{k-1}
\alpha_i(1+\SFm_i+\SFb_i)
\biggr).
\label{eq:EC-structured-radius-restoring-drift}
\end{align}
\end{lemma}

\proof{Proof of Lemma~\ref{lemma:ec-extension-radius-restoring-drift}.}
Fix \(a\geq0\) and integers \(k>\ell\geq1\).  Adding and subtracting the
lagged field terms gives the decomposition
\begin{align*}
\langle h_a(\theta_k-\theta^\star),F(\theta_k,W_k)\rangle
={}&
\underbrace{\langle h_a(\theta_k-\theta^\star),
\bar F(\theta_k)\rangle}_{\Psi_1}\\
&+\underbrace{\langle h_a(\theta_{k-\ell}-\theta^\star),
F(\theta_{k-\ell},W_k)-\bar F(\theta_{k-\ell})\rangle}_{\Psi_2}\\
&+\underbrace{\langle h_a(\theta_k-\theta^\star)
-h_a(\theta_{k-\ell}-\theta^\star),
F(\theta_k,W_k)-\bar F(\theta_k)\rangle}_{\Psi_3}\\
&+\underbrace{\langle h_a(\theta_{k-\ell}-\theta^\star),
F(\theta_k,W_k)-F(\theta_{k-\ell},W_k)
-\bar F(\theta_k)+\bar F(\theta_{k-\ell})\rangle}_{\Psi_4}.
\end{align*}

For \(\Psi_1\), the definition of \(h_a\) and
Assumption~\ref{assump:qsm} give
\[
\E[\Psi_1\mid\mathscr G_0^+]
\leq
-\zeta\E\bigl[
\lVert\theta_k-\theta^\star\rVert
(\lVert\theta_k-\theta^\star\rVert-a)_+
\mid\mathscr G_0^+\bigr].
\]

We next control \(\Psi_2\).  Under
\eqref{eq:extension-post-update-transition}, the proof of
Corollary~\ref{corollary:lagged-update-field-comparison} applies with
\((\mathscr G_j^+)_{j\geq0}\) in place of \((\F_j)_{j\geq0}\).
Equations~\eqref{eq:extension-recursion} and
\eqref{eq:extension-interlacing} make
\(h_a(\theta_{k-\ell}-\theta^\star)\)
\(\mathscr G_{k-\ell}^+\)-measurable.  Assumptions~\ref{assump:compact}
and~\ref{assump:interior-target} also give
\(\lVert h_a(\theta_{k-\ell}-\theta^\star)\rVert\leq d_\Theta\).
Applying the corollary conditionally on \(\mathscr G_{k-\ell}^+\) and then
using the tower property, \eqref{eq:EC-structured-post-moments},
Lemma~\ref{lemma:ec-extension-movement}, and
\(\sum_{j=1}^{\ell-1}\rho^{\ell-j-1}\leq(1-\rho)^{-1}\) gives a constant
\(C_0\), independent of \(a\), such that
\begin{align*}
\bigl|\E[\Psi_2\mid\mathscr G_0^+]\bigr|
&\leq
C_0\rho^\ell\E[U_1(W_{k-\ell})\mid\mathscr G_0^+]
+C_0\sum_{j=1}^{\ell-1}\rho^{\ell-j-1}
\E[U_1(W_{k-\ell+j})
\lVert\theta_{k-\ell+j}-\theta_{k-\ell}\rVert
\mid\mathscr G_0^+]\\
&\leq
C_0 U_2(W_0)\biggl(
\rho^\ell+\sum_{i=k-\ell}^{k-1}
\alpha_i(1+\SFm_i+\SFb_i)
\biggr).
\end{align*}
In the second inequality, each partial sum bounding a parameter change is bounded by the full sum from
\(k-\ell\) to \(k-1\).

For \(\Psi_3\), the 1-Lipschitz property of \(h_a\), the linear-growth bound
\eqref{eq:LinearGrowth-F} in Assumption~\ref{assump:growth}, and
\eqref{eq:F_bar_bound} in Lemma~\ref{lemma:update-field-bounds}, together
with \(U_1\geq1\), give a constant \(C_1\), independent of \(a\), such that
\[
|\Psi_3|
\leq C_1 U_1(W_k)\lVert\theta_k-\theta_{k-\ell}\rVert.
\]
Lemma~\ref{lemma:ec-extension-movement} therefore gives a constant \(C_2\),
independent of \(a\), such that
\[
\E[|\Psi_3|\mid\mathscr G_0^+]
\leq
C_2 U_2(W_0)\sum_{i=k-\ell}^{k-1}
\alpha_i(1+\SFm_i+\SFb_i).
\]

For \(\Psi_4\), Assumptions~\ref{assump:compact}
and~\ref{assump:interior-target} give
\(\lVert h_a(\theta_{k-\ell}-\theta^\star)\rVert\leq d_\Theta\).
The weighted Lipschitz bound \eqref{eq:Lipschitz-F} in
Assumption~\ref{assump:growth},
Lemma~\ref{lemma:ec-stationary-mean-field-lipschitz}, and \(U_1\geq1\) give
a constant \(C_3\), independent of \(a\), such that
\[
|\Psi_4|
\leq C_3 U_1(W_k)\lVert\theta_k-\theta_{k-\ell}\rVert.
\]
Lemma~\ref{lemma:ec-extension-movement} therefore gives a constant \(C_4\),
independent of \(a\), such that
\[
\E[|\Psi_4|\mid\mathscr G_0^+]
\leq
C_4 U_2(W_0)\sum_{i=k-\ell}^{k-1}
\alpha_i(1+\SFm_i+\SFb_i).
\]
Combining the four bounds proves
\eqref{eq:EC-structured-radius-restoring-drift} with
\(C:=C_0+C_2+C_4\).
\Halmos\endproof

\begin{lemma}
\label{lemma:ec-extension-bias-aware-entrance}
Suppose Assumptions~\ref{assump:learning-rate}--\ref{assump:structured-errors} hold, with \(1/2<\delta\leq1\).  When
\(\delta=1\), also suppose \(2\zeta\alpha_0>1\).  Fix
\(0\leq\omega_{\mathsf M}<\delta-1/2\) and
\(0\leq\beta<\min\{\omega_{\mathsf B},
\delta-\omega_{\mathsf M}-1/2\}\).  If \(B_k=0\) almost surely for all
\(k\), drop the restriction involving \(\omega_{\mathsf B}\); if
\(M_{k+1}=0\) almost surely for all \(k\), take
\(\omega_{\mathsf M}=0\).  Let \(\{\mathsf g_k\}_{k\geq0}\) be the sequence
fixed above.  Then there are positive constants \(C\) and \(K\), with \(K\) an iteration threshold, such
that, for all \(k\geq K\),
\begin{equation}\label{eq:EC-structured-excess-mse}
\E[\Phi_{\mathsf g_k}(\theta_k-\theta^\star)\mid\mathscr G_0^+]
\leq
C U_2(W_0)(1+\log k)
k^{-(\delta-2\omega_{\mathsf M})}.
\end{equation}
For either block construction, write \(I:=n_r\).  For
all \(x>0\), there is a positive iteration threshold \(K_x\geq K\) such
that every block with \(I\geq K_x\) satisfies
\begin{equation}\label{eq:EC-structured-entrance-rate}
\pr\Bigl(
\lVert\theta_I-\theta^\star\rVert>x n_{r+1}^{-\beta}
\,\Bigm|\,\mathscr G_0^+\Bigr)
\leq
C x^{-2}U_2(W_0)(1+\log I)
I^{-\delta+2\omega_{\mathsf M}+2\beta}.
\end{equation}
\end{lemma}

\proof{Proof of Lemma~\ref{lemma:ec-extension-bias-aware-entrance}.}
Define
\(y_k:=\E[\Phi_{\mathsf g_k}(\theta_k-\theta^\star)
\mid\mathscr G_0^+]\).
Fix \(k\) for which the conclusions of
Lemma~\ref{lemma:ec-extension-bias-radius} hold.  Since
\(\mathsf g_{k+1}<d_\star\), the point
\(\theta_k-h_{\mathsf g_{k+1}}(\theta_k-\theta^\star)\) lies in
\(\Theta\).  Projection nonexpansiveness and the comparison used in
\eqref{eq:EC-structured-distance-smoothness} give
\begin{align*}
\Phi_{\mathsf g_{k+1}}(\theta_{k+1}-\theta^\star)
\leq{}&
\Phi_{\mathsf g_{k+1}}(\theta_k-\theta^\star)
+2\alpha_k\langle h_{\mathsf g_{k+1}}(\theta_k-\theta^\star),
F(\theta_k,W_k)+M_{k+1}+B_k\rangle\\
&+\alpha_k^2
\lVert F(\theta_k,W_k)+M_{k+1}+B_k\rVert^2.
\end{align*}
Conditioning first on \(\mathscr G_k^-\),
Assumption~\ref{assump:structured-errors}\textup{(i)} eliminates the terms linear in the martingale-difference noise
and controls its conditional second moment.  Hence
\begin{align*}
\E[\Phi_{\mathsf g_{k+1}}(\theta_{k+1}-\theta^\star)
\mid\mathscr G_k^-]
\leq{}&
(\lVert\theta_k-\theta^\star\rVert-\mathsf g_{k+1})_+^2
+2\alpha_k\langle h_{\mathsf g_{k+1}}(\theta_k-\theta^\star),
F(\theta_k,W_k)+B_k\rangle\\
&+\alpha_k^2\lVert F(\theta_k,W_k)+B_k\rVert^2
+\alpha_k^2\SFm_k^2U_2(W_k).
\end{align*}
Assumption~\ref{assump:structured-errors}\textup{(ii)} and the identity
\(\lVert h_a(e)\rVert=(\lVert e\rVert-a)_+\) give
\[
\langle h_{\mathsf g_{k+1}}(\theta_k-\theta^\star),B_k\rangle
\leq
\SFb_k(\lVert\theta_k-\theta^\star\rVert-\mathsf g_{k+1})_+.
\]
Assumption~\ref{assump:growth},
Assumption~\ref{assump:structured-errors}\textup{(ii)}, and
\eqref{eq:EC-structured-post-moments} control the quadratic terms.  Choose a
constant \(C_0\) that dominates the constants in these estimates and in
Lemma~\ref{lemma:ec-extension-radius-restoring-drift}.  Then
\[
\E[\lVert F(\theta_k,W_k)+B_k\rVert^2
+\SFm_k^2U_2(W_k)\mid\mathscr G_0^+]
\leq
C_0 U_2(W_0)(1+\SFm_k^2+\SFb_k^2).
\]
Taking conditional expectations given \(\mathscr G_0^+\), applying
Lemma~\ref{lemma:ec-extension-radius-restoring-drift} with
\(a=\mathsf g_{k+1}\), and then using the one-step comparison in
Lemma~\ref{lemma:ec-extension-bias-radius} give, for all
\(k>\ell\geq1\),
\begin{align}
y_{k+1}
\leq{}&(1-2\zeta\alpha_k)y_k
+C_0 U_2(W_0)\alpha_k\biggl(
\rho^\ell+\sum_{i=k-\ell}^{k-1}
\alpha_i(1+\SFm_i+\SFb_i)
\biggr)
\nonumber\\
&+C_0 U_2(W_0)\alpha_k^2
(1+\SFm_k^2+\SFb_k^2).
\label{eq:EC-structured-adjusted-mse-recursion}
\end{align}

Set \(\ell=\ell_k\), where \(\ell_k\) is defined in
\eqref{eq:log-lag}.  That definition, Assumption~\ref{assump:learning-rate},
and Assumption~\ref{assump:structured-errors}\textup{(iii)} give a constant
\(C_1\geq C_0\) such that, for all sufficiently large \(k\),
\(k-\ell_k\geq k/2\), \(\ell_k\leq C_1(1+\log k)\), and
\(\rho^{\ell_k}\leq k^{-\delta}\), while, for all
\(k-\ell_k\leq i\leq k-1\),
\(\alpha_i\leq C_1k^{-\delta}\),
\(\SFm_i\leq C_1k^{\omega_{\mathsf M}}\), and \(\SFb_i\leq C_1\).
Consequently,
\begin{align*}
\alpha_k\biggl(
\rho^{\ell_k}+\sum_{i=k-\ell_k}^{k-1}
\alpha_i(1+\SFm_i+\SFb_i)
\biggr)
+\alpha_k^2(1+\SFm_k^2+\SFb_k^2)\leq
C_1(1+\log k)k^{-(2\delta-2\omega_{\mathsf M})}.
\end{align*}
Thus \eqref{eq:EC-structured-adjusted-mse-recursion}, with
\(\ell=\ell_k\), becomes
\begin{equation}\label{eq:EC-structured-adjusted-mse-power-recursion}
y_{k+1}
\leq
(1-2\zeta\alpha_0k^{-\delta})y_k
+C_1 U_2(W_0)(1+\log k)
k^{-(2\delta-2\omega_{\mathsf M})}
\end{equation}
for all sufficiently large \(k\).

The condition \(\omega_{\mathsf M}<\delta-1/2\) gives
\(\delta-2\omega_{\mathsf M}>0\).  Hence
\((1+1/k)^{-(\delta-2\omega_{\mathsf M})}
\geq1-(\delta-2\omega_{\mathsf M})/k\), and
\(1+\log(k+1)\geq1+\log k\), so
\begin{align*}
& (1+\log(k+1))(k+1)^{-(\delta-2\omega_{\mathsf M})}
-(1-2\zeta\alpha_0k^{-\delta})
(1+\log k)k^{-(\delta-2\omega_{\mathsf M})}\\
\geq{}&
\bigl[2\zeta\alpha_0
-(\delta-2\omega_{\mathsf M})k^{\delta-1}\bigr]
(1+\log k)k^{-(2\delta-2\omega_{\mathsf M})}.
\end{align*}
When \(\delta<1\), the coefficient in brackets is at least
\(\zeta\alpha_0\) for all sufficiently large \(k\).  When \(\delta=1\),
it equals
\(2\zeta\alpha_0-(1-2\omega_{\mathsf M})>0\).

Choose an iteration threshold \(K\) so that the conclusions of
Lemma~\ref{lemma:ec-extension-bias-radius},
the recursion in \eqref{eq:EC-structured-adjusted-mse-power-recursion}, the
nonnegativity of its contraction coefficient, and the applicable positive
comparison gap all hold for \(k\geq K\).  Choose \(C_2\) large
enough that
\(C_2\zeta\alpha_0\geq C_1\) when \(\delta<1\),
\(C_2[2\zeta\alpha_0-(1-2\omega_{\mathsf M})]\geq C_1\) when
\(\delta=1\), and
\[
y_K\leq C_2U_2(W_0)(1+\log K)
K^{-(\delta-2\omega_{\mathsf M})}.
\]
Such a deterministic choice is possible because \(y_K\leq d_\Theta^2\)
and \(U_2(W_0)\geq1\).  Induction in
\eqref{eq:EC-structured-adjusted-mse-power-recursion} gives, for all
\(k\geq K\),
\[
y_k\leq C_2U_2(W_0)(1+\log k)
k^{-(\delta-2\omega_{\mathsf M})}.
\]
This proves \eqref{eq:EC-structured-excess-mse}.

Finally, fix \(x>0\), and let \(I:=n_r\).  By
Lemma~\ref{lemma:ec-extension-bias-radius},
\(\mathsf g_I=o(I^{-\beta})\), while
Lemma~\ref{lemma:block-arithmetic} gives \(n_{r+1}/I\leq2\) for all sufficiently
large blocks.  Choose an iteration threshold \(K_x\geq K\) so that
\(\mathsf g_I\leq x n_{r+1}^{-\beta}/2\) for all blocks with
\(I\geq K_x\).
Then
\[
\{\lVert\theta_I-\theta^\star\rVert>x n_{r+1}^{-\beta}\}
\subseteq
\biggl\{
\Phi_{\mathsf g_I}(\theta_I-\theta^\star)
>x^2n_{r+1}^{-2\beta}/4
\biggr\}.
\]
Markov's inequality and \eqref{eq:EC-structured-excess-mse} yield
\begin{align*}
\pr\Bigl(
\lVert\theta_I-\theta^\star\rVert>x n_{r+1}^{-\beta}
\,\Bigm|\,\mathscr G_0^+\Bigr)
&\leq
4x^{-2}n_{r+1}^{2\beta}y_I\\
&\leq
2^{2\beta+2}C_2 x^{-2}U_2(W_0)(1+\log I)
I^{-\delta+2\omega_{\mathsf M}+2\beta}.
\end{align*}
Taking \(C:=2^{2\beta+2}C_2\), the constant is independent of \(x\), while
only the iteration threshold \(K_x\) depends on \(x\).  This proves
\eqref{eq:EC-structured-entrance-rate}.
\Halmos\endproof

\subsection{Concentration Within a Block}

The block argument for the SA recursion in
\eqref{eq:update-scheme-projection} extends after accounting for two
additional sums.  Martingale-difference noise and predictable bias increase the bound on
parameter changes over the logarithmic lag interval, and their
accumulated contributions must also be controlled within each block.

\begin{lemma}
\label{lemma:ec-extension-block-comparison}
Suppose Assumptions~\ref{assump:learning-rate}--\ref{assump:growth} hold,
together with parts~\textup{(i)}--\textup{(ii)} of
Assumption~\ref{assump:structured-errors}.  Let
\(I\leq J\), \(\ell\geq1\), \(L:=J-I+1\), and \(I\geq2\ell\).  Define the
accumulated Markovian noise by
\[
S^{\mathrm{markov}}(m)
:=
\sum_{k=I}^{m}\alpha_k
\bigl(F(\theta_{k-\ell},W_k)-\bar F(\theta_{k-\ell})\bigr),
\qquad I\leq m\leq J.
\]
There is a positive constant \(C\) such that, for all \(x>0\),
\begin{align}
&\pr\biggl(
\max_{I\leq m\leq J}\lVert S^{\mathrm{markov}}(m)\rVert\geq x
\,\biggm|\,\mathscr G_0^+
\biggr)
\nonumber\\
&\quad\leq
C U_2(W_0)\biggl[
 x^{-2}\ell^2L\alpha_I^2
+x^{-1}L\alpha_I\biggl(
\rho^\ell+
\max_{I\leq k\leq J}\sum_{i=k-\ell}^{k-1}
\alpha_i(1+\SFm_i+\SFb_i)
\biggr)
\biggr].
\label{eq:extension-block-maximal-bound}
\end{align}
\end{lemma}

\proof{Proof of Lemma~\ref{lemma:ec-extension-block-comparison}.}
For \(k=I,\ldots,J\), define the lagged Markovian noise and its
lagged-centering decomposition by
\[
\xi_{k,\ell}
:=F(\theta_{k-\ell},W_k)-\bar F(\theta_{k-\ell}),
\quad
\xi_{k,\ell}^{\mathrm{pred}}
:=\E[\xi_{k,\ell}\mid\mathscr G_{k-\ell}^+],
\quad\mbox{and}\quad
\xi_{k,\ell}^{\mathrm{cent}}
:=\xi_{k,\ell}-\xi_{k,\ell}^{\mathrm{pred}}.
\]
Equation~\eqref{eq:extension-post-update-transition} has the kernel form
required by Corollary~\ref{corollary:lagged-update-field-comparison}.
That corollary and Lemma~\ref{lemma:ec-extension-movement} give a constant
\(C_0\) such that
\begin{equation}\label{eq:EC-extension-block-predictable-input}
\E[\lVert\xi_{k,\ell}^{\mathrm{pred}}\rVert\mid\mathscr G_0^+]
\leq
C_0 U_2(W_0)\biggl(
\rho^\ell+
\sum_{i=k-\ell}^{k-1}\alpha_i(1+\SFm_i+\SFb_i)
\biggr).
\end{equation}

Equation~\eqref{eq:EC-structured-post-moments} implies that these random
vectors are square integrable.  Assumption~\ref{assump:growth} and Jensen's
inequality give a constant \(C_1\) such that
\begin{equation}\label{eq:EC-extension-block-centered-input}
\E[\lVert\xi_{k,\ell}^{\mathrm{cent}}\rVert^2
\mid\mathscr G_t^+]
\leq
C_1
\begin{cases}
\lambda_2^{k-t}U_2(W_t)+1,
&0\leq t\leq k-\ell,\\
U_2(W_t)+U_2(W_{k-\ell})+1,
&k-\ell<t<k.
\end{cases}
\end{equation}
Finally, if \(I\leq i<k\leq J\), \(0\leq t\leq i\), and
\(k-i\geq\ell\), then \eqref{eq:extension-interlacing} makes
\(\xi_{i,\ell}^{\mathrm{cent}}\)
\(\mathscr G_{k-\ell}^+\)-measurable.  Conditional centering and the tower
property yield
\begin{equation}\label{eq:EC-extension-block-cross-terms}
\E[(\xi_{k,\ell}^{\mathrm{cent}})^\intercal
\xi_{i,\ell}^{\mathrm{cent}}\mid\mathscr G_t^+]=0.
\end{equation}

Define
\[
S^{\mathrm{pred}}(m)
:=\sum_{k=I}^{m}\alpha_k\xi_{k,\ell}^{\mathrm{pred}},
\qquad
S^{\mathrm{cent}}(m)
:=\sum_{k=I}^{m}\alpha_k\xi_{k,\ell}^{\mathrm{cent}},
\]
for \(I\leq m\leq J\), and set both sums equal to zero at \(m=I-1\).  Thus
\(S^{\mathrm{markov}}=S^{\mathrm{pred}}+S^{\mathrm{cent}}\), and
\(S^{\mathrm{cent}}(m)\) is adapted to \((\mathscr G_m^+)\).
Here \(\xi_{k,\ell}^{\mathrm{cent}}\) is centered Markovian noise relative to
\(\mathscr G_{k-\ell}^+\).

The centered-sum estimates from
\eqref{eq:cen-full-block-second-moment} through \eqref{eq:cen-max-final} in
the proof of Lemma~\ref{lemma:ec-canonical-block-maximal} use the
conditional second-moment bound and the disappearance of cross terms at
separations of at least \(\ell\).  Use those estimates with
\(\F_t\) replaced by \(\mathscr G_t^+\),
\eqref{eq:lag-centered-conditional-moment} replaced by
\eqref{eq:EC-extension-block-centered-input},
\eqref{eq:lag-covariance-cancellation} replaced by
\eqref{eq:EC-extension-block-cross-terms}, and the moment bound replaced by
\eqref{eq:EC-structured-post-moments}.  Choose a constant \(C_2\) that
dominates the constants in the resulting full-block and future-sum estimates
and in the conditional-L\'evy reduction below.  Then
\[
\E[\lVert S^{\mathrm{cent}}(J)\rVert^2\mid\mathscr G_0^+]
\leq
C_2 U_2(W_0)\ell L\alpha_I^2.
\]
Applying the covariance-band count to each future interval
\(\{m+1,\ldots,J\}\) gives, for all \(y>0\), the following bound. The
remainder at \(m=J\) is zero, so the associated union bound has exactly
\(L\) nonterminal indices:
\begin{align*}
&\pr\biggl(
\max_{I-1\leq m\leq J}
\E[\lVert S^{\mathrm{cent}}(J)-S^{\mathrm{cent}}(m)\rVert^2
\mid\mathscr G_m^+]
\geq y
\,\biggm|\,\mathscr G_0^+
\biggr)
\leq
C_2 U_2(W_0)y^{-1}\ell^2L\alpha_I^2.
\end{align*}
Apply Lemma~\ref{lemma:conditional-levy-remainder} with
\(n=L\), \(\mathcal G_q:=\mathscr G_{I+q-1}^+\) for
\(q=0,\ldots,L\), \(\mathcal H:=\mathscr G_0^+\), and
\(X_q:=\alpha_{I+q-1}\xi_{I+q-1,\ell}^{\mathrm{cent}}\) for
\(q=1,\ldots,L\). This gives
\begin{equation}\label{eq:EC-extension-centered-block-maximum}
\pr\biggl(
\max_{I\leq m\leq J}\lVert S^{\mathrm{cent}}(m)\rVert\geq x
\,\biggm|\,\mathscr G_0^+
\biggr)
\leq
C_2 U_2(W_0)x^{-2}\ell^2L\alpha_I^2.
\end{equation}

For \(S^{\mathrm{pred}}\), the triangle inequality, Markov's inequality,
and \eqref{eq:EC-extension-block-predictable-input} give a constant \(C_3\)
such that
\begin{align}
\pr\biggl(
\max_{I\leq m\leq J}\lVert S^{\mathrm{pred}}(m)\rVert\geq x
\,\biggm|\,\mathscr G_0^+
\biggr)
\leq
C_3 U_2(W_0)x^{-1}L\alpha_I\biggl(
\rho^\ell+
\max_{I\leq k\leq J}\sum_{i=k-\ell}^{k-1}
\alpha_i(1+\SFm_i+\SFb_i)
\biggr).
\label{eq:EC-extension-predictable-block-maximum}
\end{align}
Applying \eqref{eq:EC-extension-centered-block-maximum} and
\eqref{eq:EC-extension-predictable-block-maximum} with threshold \(x/2\)
proves \eqref{eq:extension-block-maximal-bound} with
\(C:=4C_2+2C_3\).
\Halmos\endproof

\begin{lemma}
\label{lemma:ec-extension-one-block}
Suppose Assumptions~\ref{assump:learning-rate}--\ref{assump:structured-errors} hold, with \(1/2<\delta\leq1\).  When
\(\delta=1\), also suppose \(2\zeta\alpha_0>1\).  Fix
\(0\leq\omega_{\mathsf M}<\delta-1/2\),
\(0\leq\beta<\min\{\omega_{\mathsf B},
\delta-\omega_{\mathsf M}-1/2\}\), and \(c>0\).  If \(B_k=0\) almost
surely for all \(k\), drop the restriction involving
\(\omega_{\mathsf B}\); if \(M_{k+1}=0\) almost surely for all \(k\), take
\(\omega_{\mathsf M}=0\).  Define
\(s:=\delta-2\beta-2\omega_{\mathsf M}\).  There are positive constants \(C\) and \(R\), with \(R\) a
block-index threshold, such that, for all \(r\geq R\), the applicable
block with \(I:=n_r\) and \(J:=n_{r+1}-1\) satisfies
\begin{equation}\label{eq:EC-structured-one-block-failure}
\pr\biggl(
\max_{I\leq k\leq J}\lVert\theta_k-\theta^\star\rVert
>c n_{r+1}^{-\beta}
\,\biggm|\,\mathscr G_0^+
\biggr)
\leq
C U_2(W_0)(1+\log I)^2I^{-s}.
\end{equation}
\end{lemma}

\proof{Proof of Lemma~\ref{lemma:ec-extension-one-block}.}
Set \(\widetilde c:=\min\{c,d_\star/2\}\).  Since \(\widetilde c\leq c\),
an exit from the tube of radius \(c n_{r+1}^{-\beta}\) is also an exit from
the tube of radius \(\widetilde c n_{r+1}^{-\beta}\).  It therefore suffices
to bound the latter event.  Fix a block from the applicable polynomial or dyadic
construction, set \(I:=n_r\), \(J:=n_{r+1}-1\),
\(L:=J-I+1=n_{r+1}-n_r\), and set \(\ell:=\ell_I\) using
\eqref{eq:log-lag}.

Let \(R_0\) be the block-index threshold in
Lemma~\ref{lemma:block-arithmetic}, and let \(C_0\) be
\(\max\{1,\alpha_0\}\) times the constant in that lemma.  Then, for all
\(r\geq R_0\),
\(I\geq2\ell\), \(n_{r+1}/I\leq2\), \(L\leq C_0I^\delta\),
\(\ell\leq C_0(1+\log I)\), and \(\rho^\ell\leq I^{-\delta}\).
Since \(\alpha_k\leq\alpha_I\) for \(k\geq I\) by
Assumption~\ref{assump:learning-rate}, the original bound on \(L\) in
Lemma~\ref{lemma:block-arithmetic} and the definition of \(C_0\) also give
\(\sum_{k=I}^{J}\alpha_k\leq L\alpha_I\leq C_0\).
Moreover, \(I\geq2\ell\) and \(n_{r+1}/I\leq2\) imply
\(I/2\leq i<2I\) whenever \(I\leq k\leq J\) and \(k-\ell\leq i<k\).

For \(I\leq m\leq J\), define
\begin{align*}
S_r^{\mathrm{markov}}(m)
&:=\sum_{k=I}^{m}\alpha_k
\bigl(F(\theta_{k-\ell},W_k)-\bar F(\theta_{k-\ell})\bigr),\\
S_r^{\mathrm{field}}(m)
&:=\sum_{k=I}^{m}\alpha_k\bar F(\theta_{k-\ell}),\\
S_r^{\mathrm{lag}}(m)
&:=\sum_{k=I}^{m}\alpha_k
\bigl(F(\theta_k,W_k)-F(\theta_{k-\ell},W_k)\bigr),\\
S_r^{\mathrm{mart}}(m)
&:=\sum_{k=I}^{m}\alpha_kM_{k+1},\\
S_r^{\mathrm{bias}}(m)
&:=\sum_{k=I}^{m}\alpha_kB_k.
\end{align*}
These five sums are, respectively, accumulated Markovian noise, accumulated
mean field, parameter-lag remainder, accumulated martingale-difference
noise, and accumulated predictable bias.
Set all five sums equal to zero at \(m=I-1\), and let
\(\psi_r:=\widetilde c n_{r+1}^{-\beta}\).  Define
\[
\mathcal E_r^{\mathrm{ent}}
:=\{\lVert\theta_I-\theta^\star\rVert>\psi_r/8\}
\quad\mbox{and}\quad
\mathcal E_r^{\mathord\bullet}
:=\biggl\{
\max_{I\leq m\leq J}\lVert S_r^{\mathord\bullet}(m)\rVert
>\psi_r/8
\biggr\},
\]
where \(\mathord\bullet\) stands for
\(\mathrm{markov}\), \(\mathrm{field}\), or \(\mathrm{lag}\).  Define
separately
\[
\mathcal E_r^{\mathrm{mart}}
:=\biggl\{
\max_{I\leq m\leq J}\lVert S_r^{\mathrm{mart}}(m)\rVert
>\psi_r/8
\biggr\}
\quad\mbox{and}\quad
\mathcal E_r^{\mathrm{bias}}
:=\biggl\{
\max_{I\leq m\leq J}\lVert S_r^{\mathrm{bias}}(m)\rVert
>\psi_r/8
\biggr\}.
\]

We first verify that these six events cover an exit from the block.  Fix an
outcome outside their union and suppose that the first exit occurs at
\(\tau\in\{I,\ldots,J\}\).  Because the outcome lies outside
\(\mathcal E_r^{\mathrm{ent}}\), we have \(\tau>I\).  Every
iterate before \(\tau\) lies within distance
\(\psi_r\leq\widetilde c<d_\star\) of \(\theta^\star\), and hence lies in
the interior of \(\Theta\).  Every completed projection before \(\tau\) has
an interior output and is therefore inactive; nonexpansiveness controls the
projection in the update producing \(\theta_\tau\).  Consequently,
\begin{align*}
\lVert\theta_\tau-\theta^\star\rVert
\leq{}&
\lVert\theta_I-\theta^\star\rVert
+\lVert S_r^{\mathrm{markov}}(\tau-1)\rVert
+\lVert S_r^{\mathrm{field}}(\tau-1)\rVert\\
&+\lVert S_r^{\mathrm{lag}}(\tau-1)\rVert
+\lVert S_r^{\mathrm{mart}}(\tau-1)\rVert
+\lVert S_r^{\mathrm{bias}}(\tau-1)\rVert
\leq6\psi_r/8<\psi_r,
\end{align*}
which contradicts the definition of \(\tau\).  Therefore
\begin{align}
\Bigl\{\max_{I\leq k\leq J}\lVert\theta_k-\theta^\star\rVert
>\psi_r\Bigr\}
\subseteq{}&
\mathcal E_r^{\mathrm{ent}}
\cup\mathcal E_r^{\mathrm{markov}}
\cup\mathcal E_r^{\mathrm{field}}
\cup\mathcal E_r^{\mathrm{lag}}
\cup\mathcal E_r^{\mathrm{mart}}
\cup\mathcal E_r^{\mathrm{bias}}.
\label{eq:EC-structured-six-event-containment}
\end{align}

We now bound the six component events in turn.  The condition
\(\beta<\delta-\omega_{\mathsf M}-1/2\) gives
\(s-(1-\delta)=2(\delta-\beta-\omega_{\mathsf M})-1>0\), so \(s>0\).
Assumption~\ref{assump:learning-rate} and
Assumption~\ref{assump:structured-errors}\textup{(iii)}, together with
\(I/2\leq i<2I\), give a constant \(C_1\) and a block-index threshold
\(R_1\geq R_0\) such that, for all \(r\geq R_1\), \(I\leq k\leq J\), and
\(k-\ell\leq i<k\),
\begin{equation}\label{eq:EC-extension-lookback-envelope}
\alpha_i(1+\SFm_i+\SFb_i)
\leq C_1I^{-\delta+\omega_{\mathsf M}}.
\end{equation}

For \(\mathcal E_r^{\mathrm{ent}}\),
Lemma~\ref{lemma:ec-extension-bias-aware-entrance}, applied with
\(x=\widetilde c/8\), gives a constant \(C_2\).  Choose a block-index
threshold \(R_2\geq R_1\) so that \(n_r\geq K_{\widetilde c/8}\) for all
\(r\geq R_2\).  Then
\begin{equation}\label{eq:EC-structured-block-entrance-bound}
\pr(\mathcal E_r^{\mathrm{ent}}\mid\mathscr G_0^+)
\leq
C_2 U_2(W_0)(1+\log I)I^{-s}.
\end{equation}

For \(\mathcal E_r^{\mathrm{markov}}\), summing
\eqref{eq:EC-extension-lookback-envelope} over
\(i=k-\ell,\ldots,k-1\) gives
\[
\max_{I\leq k\leq J}\sum_{i=k-\ell}^{k-1}
\alpha_i(1+\SFm_i+\SFb_i)
\leq C_1\ell I^{-\delta+\omega_{\mathsf M}}.
\]
Applying Lemma~\ref{lemma:ec-extension-block-comparison} at threshold
\(\psi_r/8=(\widetilde c/8)n_{r+1}^{-\beta}\) and
Lemma~\ref{lemma:block-arithmetic} gives a constant
\(C_3\) such that
\begin{align}
\pr(\mathcal E_r^{\mathrm{markov}}\mid\mathscr G_0^+)
&\leq
C_3 U_2(W_0)\Bigl[
\ell^2I^{-\delta+2\beta}
+I^{-\delta+\beta}
+\ell I^{-\delta+\beta+\omega_{\mathsf M}}
\Bigr]
\nonumber\\
&\leq
C_3 U_2(W_0)(1+\log I)^2I^{-s}.
\label{eq:EC-structured-block-noise-bound}
\end{align}
The three decay exponents are at least \(s\), with differences
\(2\omega_{\mathsf M}\), \(\beta+2\omega_{\mathsf M}\), and
\(\beta+\omega_{\mathsf M}\), respectively.  Martingale-difference noise
and predictable bias enter this estimate only through the sum bounding parameter changes
\(\sum_{i=k-\ell}^{k-1}\alpha_i(1+\SFm_i+\SFb_i)\) in
Lemma~\ref{lemma:ec-extension-block-comparison}.

For \(\mathcal E_r^{\mathrm{field}}\), let \(K_0\) be the iteration threshold
in Lemma~\ref{lemma:ec-extension-bias-aware-entrance}, and choose a
block-index threshold \(R_3\geq R_2\) so that \(I-\ell\geq K_0\) for all
\(r\geq R_3\).  Choose a constant \(C_4\) that dominates the constants
from \eqref{eq:EC-structured-excess-mse},
Lemma~\ref{lemma:block-arithmetic}, and
Lemma~\ref{lemma:ec-stationary-mean-field-lipschitz}.  Then, for all
\(r\geq R_3\) and uniformly over \(I\leq k\leq J\),
\begin{equation}\label{eq:EC-extension-uniform-excess-mse}
\E[\Phi_{\mathsf g_{k-\ell}}(\theta_{k-\ell}-\theta^\star)
\mid\mathscr G_0^+]
\leq
C_4 U_2(W_0)(1+\log I)I^{-\delta+2\omega_{\mathsf M}}.
\end{equation}
Lemma~\ref{lemma:ec-stationary-mean-field-lipschitz} and
Assumption~\ref{assump:interior-target}, which gives
\(\bar F(\theta^\star)=0\), imply
\[
\lVert\bar F(\theta_{k-\ell})\rVert
\leq
C_4\bigl(
\sqrt{\Phi_{\mathsf g_{k-\ell}}(\theta_{k-\ell}-\theta^\star)}
+\mathsf g_{k-\ell}
\bigr).
\]
Consequently,
\begin{align*}
\max_{I\leq m\leq J}\lVert S_r^{\mathrm{field}}(m)\rVert
\leq C_4\sum_{k=I}^{J}\alpha_k
\sqrt{\Phi_{\mathsf g_{k-\ell}}(\theta_{k-\ell}-\theta^\star)}
+C_4\sum_{k=I}^{J}\alpha_k\mathsf g_{k-\ell}.
\end{align*}
Lemma~\ref{lemma:ec-extension-bias-radius} gives
\(\mathsf g_{I-\ell}=o(I^{-\beta})\).  Since the radius sequence is
nonincreasing,
\[
\sum_{k=I}^{J}\alpha_k\mathsf g_{k-\ell}
\leq
\biggl(\sum_{k=I}^{J}\alpha_k\biggr)\mathsf g_{I-\ell}
=o(I^{-\beta})
=o(n_{r+1}^{-\beta}).
\]
Choose a block-index threshold \(R_4\geq R_3\) so that
\(C_4\sum_{k=I}^{J}\alpha_k\mathsf g_{k-\ell}
\leq\psi_r/16\) for all
\(r\geq R_4\).
Weighted Cauchy--Schwarz and
\eqref{eq:EC-extension-uniform-excess-mse} give
\begin{align*}
\E\biggl[\biggl(
\sum_{k=I}^{J}\alpha_k
\sqrt{\Phi_{\mathsf g_{k-\ell}}(\theta_{k-\ell}-\theta^\star)}
\biggr)^2\,\biggm|\,\mathscr G_0^+\biggr]
&\leq
\biggl(\sum_{k=I}^{J}\alpha_k\biggr)
\sum_{k=I}^{J}\alpha_k
\E[\Phi_{\mathsf g_{k-\ell}}(\theta_{k-\ell}-\theta^\star)
\mid\mathscr G_0^+]\\
&\leq
C_4 U_2(W_0)(1+\log I)I^{-\delta+2\omega_{\mathsf M}}.
\end{align*}
If \(\mathcal E_r^{\mathrm{field}}\) occurs, the excess-distance sum exceeds
\(\psi_r/(16C_4)\).  Markov's inequality therefore gives a constant
\(C_5\) such that
\begin{equation}\label{eq:EC-structured-block-drift-bound}
\pr(\mathcal E_r^{\mathrm{field}}\mid\mathscr G_0^+)
\leq
C_5 U_2(W_0)(1+\log I)I^{-s}.
\end{equation}

For \(\mathcal E_r^{\mathrm{lag}}\), Assumption~\ref{assump:growth} and the
triangle inequality give
\[
\max_{I\leq m\leq J}\lVert S_r^{\mathrm{lag}}(m)\rVert
\leq
L_F\sum_{k=I}^{J}\alpha_kU_1(W_k)
\lVert\theta_k-\theta_{k-\ell}\rVert.
\]
Choose a constant \(C_6\) that dominates the constants obtained by taking
conditional expectations, applying
Lemma~\ref{lemma:ec-extension-movement} and
\eqref{eq:EC-extension-lookback-envelope}, and then converting the resulting
bound by Markov's inequality.  Then
\[
\E\biggl[
\max_{I\leq m\leq J}\lVert S_r^{\mathrm{lag}}(m)\rVert
\,\biggm|\,\mathscr G_0^+
\biggr]
\leq
C_6 U_2(W_0)\ell I^{-\delta+\omega_{\mathsf M}}.
\]
Markov's inequality at threshold \(\psi_r/8\) and
Lemma~\ref{lemma:block-arithmetic} now give
\begin{align}
\pr(\mathcal E_r^{\mathrm{lag}}\mid\mathscr G_0^+)
&\leq
C_6 U_2(W_0)\ell I^{-\delta+\beta+\omega_{\mathsf M}}
\leq
C_6 U_2(W_0)(1+\log I)I^{-s}.
\label{eq:EC-structured-block-adaptation-bound}
\end{align}
Here \(\delta-\beta-\omega_{\mathsf M}\geq s\) because
\((\delta-\beta-\omega_{\mathsf M})-s
=\beta+\omega_{\mathsf M}\geq0\).

For \(\mathcal E_r^{\mathrm{mart}}\), the accumulated martingale-difference
noise \(S_r^{\mathrm{mart}}\) is controlled as follows.
Assumption~\ref{assump:structured-errors}\textup{(i)}, displayed in
\eqref{eq:extension-martingale-conditions},
\eqref{eq:extension-interlacing}, and the tower property show that
\((S_r^{\mathrm{mart}}(m))_{m=I-1}^{J}\) is a martingale relative to the
filtration \((\mathscr G_m^+)_{m=I-1}^{J}\).  Choose a constant
\(C_7\) that dominates the deterministic factors below.  Doob's maximal
inequality, martingale orthogonality, and
\eqref{eq:extension-martingale-conditions}, followed by
\eqref{eq:EC-structured-post-moments},
Assumption~\ref{assump:learning-rate},
Assumption~\ref{assump:structured-errors}\textup{(iii)}, and
Lemma~\ref{lemma:block-arithmetic}, give
\begin{align}
\pr(\mathcal E_r^{\mathrm{mart}}\mid\mathscr G_0^+)
&\leq
C_7 n_{r+1}^{2\beta}
\E[\lVert S_r^{\mathrm{mart}}(J)\rVert^2
\mid\mathscr G_0^+]
\nonumber\\
&=
C_7 n_{r+1}^{2\beta}\sum_{k=I}^{J}\alpha_k^2
\E[\lVert M_{k+1}\rVert^2\mid\mathscr G_0^+]
\nonumber\\
&\leq
C_7 U_2(W_0)n_{r+1}^{2\beta}
\sum_{k=I}^{J}\alpha_k^2\SFm_k^2
\leq
C_7 U_2(W_0)I^{-s}.
\label{eq:EC-structured-block-martingale-bound}
\end{align}

For \(\mathcal E_r^{\mathrm{bias}}\), if \(B_k=0\) almost surely for all
\(k\), then
\(S_r^{\mathrm{bias}}\equiv0\).  Otherwise,
Assumption~\ref{assump:structured-errors}\textup{(ii)}--\textup{(iii)} and
Lemma~\ref{lemma:block-arithmetic} give a constant \(C_8\) such that
\[
\max_{I\leq m\leq J}\lVert S_r^{\mathrm{bias}}(m)\rVert
\leq
\sum_{k=I}^{J}\alpha_k\SFb_k
\leq
C_8I^{-\omega_{\mathsf B}}
=o(n_{r+1}^{-\beta}).
\]
Indeed, since \(n_{r+1}/I\leq2\) and \(\beta<\omega_{\mathsf B}\),
\(I^{-\omega_{\mathsf B}}=o(n_{r+1}^{-\beta})\).  Choose a block-index threshold
\(R_5\geq R_4\) so that \(\mathcal E_r^{\mathrm{bias}}\) is empty for all
\(r\geq R_5\); in the zero-bias case, take \(R_5:=R_4\).

Set \(R:=R_5\) and
\(C:=C_2+C_3+C_5+C_6+C_7\).  The containment
\eqref{eq:EC-structured-six-event-containment}, the union bound, and
\eqref{eq:EC-structured-block-entrance-bound}--\eqref{eq:EC-structured-block-noise-bound}, together with
\eqref{eq:EC-structured-block-drift-bound}--\eqref{eq:EC-structured-block-martingale-bound}, give the stated bound for
exits from the tube of radius \(\psi_r\).  Since an exit from the
tube of radius \(c n_{r+1}^{-\beta}\) is contained in this event,
\eqref{eq:EC-structured-one-block-failure} follows.
\Halmos\endproof

\subsection{Proof of Theorem~\ref{thm:structured-trajectory}}

\proof{Proof of Theorem~\ref{thm:structured-trajectory}.}
Set \(s:=\delta-2\beta-2\omega_{\mathsf M}\).  The assumed range of
\(\beta\) gives
\(s-(1-\delta)=2(\delta-\beta-\omega_{\mathsf M})-1>0\).
Recall the applicable polynomial or dyadic block construction.  For an
integer \(k_0\geq n_1\), let \(r_0\) be the index of the block containing
\(k_0\), so that \(n_{r_0}\leq k_0<n_{r_0+1}\).  Lemmas
\ref{lemma:blockwise-tube-reduction} and
\ref{lemma:ec-extension-one-block}, together with the union bound, give a
constant \(C_0\) and block-index threshold \(R_0\) such that, 
\begin{align}
\pr\biggl(
\lVert\theta_k-\theta^\star\rVert>c k^{-\beta}
\text{ for some }k\geq k_0
\,\biggm|\,\mathscr G_0^+\biggr)\leq
C_0 U_2(W_0)
\sum_{r\geq r_0}(1+\log n_r)^2n_r^{-s},
\label{eq:EC-structured-future-block-sum}
\end{align}
for all
\(r_0\geq R_0\).
It remains to bound this deterministic series.

Suppose first that \(1/2<\delta<1\).  Recall that
\(n_r=\lceil r^{1/(1-\delta)}\rceil\).  The ceiling definition and
\(n_{r_0}\leq k_0<n_{r_0+1}\) give
\((k_0-1)^{1-\delta}-1<r_0\leq k_0^{1-\delta}\).  Since
\(s/(1-\delta)>1\), integral comparison gives a constant \(C_1\) and a
block-index threshold \(R_1\geq R_0\) such that, for all
\(r_0\geq R_1\),
\begin{align*}
\sum_{r\geq r_0}(1+\log n_r)^2n_r^{-s}
&\leq
C_1(1+\log r_0)^2r_0^{\,1-s/(1-\delta)}\\
&\leq
C_1(1+\log k_0)^2k_0^{\,1-\delta-s}.
\end{align*}

Now suppose that \(\delta=1\).  Recall that \(n_r=2^r\).  Since
\(n_{r_0}\leq k_0<2n_{r_0}\) and \(s>0\), writing
\(r=r_0+j\) gives a constant \(C_2\) such that
\begin{align*}
\sum_{r\geq r_0}(1+\log n_r)^2n_r^{-s}
&=
\sum_{j\geq0}
\bigl(1+\log(2^j n_{r_0})\bigr)^2
(2^j n_{r_0})^{-s}\\
&\leq
C_2(1+\log n_{r_0})^2n_{r_0}^{-s}\\
&\leq
C_2(1+\log k_0)^2k_0^{-s}.
\end{align*}
In this case, set \(R_1:=R_0\).

Set \(C:=C_0C_1\) when \(\delta<1\) and \(C:=C_0C_2\) when \(\delta=1\).
In either case, for all \(r_0\geq R_1\),
\[
C_0\sum_{r\geq r_0}(1+\log n_r)^2n_r^{-s}
\leq
C(1+\log k_0)^2k_0^{-(s-(1-\delta))}.
\]
Set \(K:=n_{R_1}\).  For all \(k_0\geq K\), substituting this
deterministic-series bound into \eqref{eq:EC-structured-future-block-sum} and
using
\(s-(1-\delta)=2(\delta-\beta-\omega_{\mathsf M})-1\) give
\[
\pr\biggl(
\lVert\theta_k-\theta^\star\rVert>c k^{-\beta}
\text{ for some }k\geq k_0
\,\biggm|\,\mathscr G_0^+\biggr)
\leq
C U_2(W_0)(1+\log k_0)^2
k_0^{-\left(2(\delta-\beta-\omega_{\mathsf M})-1\right)}.
\]
Taking complements proves \eqref{eq:structured-trajectory-bound}.
\Halmos\endproof

\section{Sharpness with Growing Martingale-Difference Noise}
\label{sec:EC-structured-sharpness-proofs}

Fix the step-size and tube parameters and the exponents
\(\omega_{\mathsf M},\omega_{\mathsf B}\) as in
Theorem~\ref{thm:structured-sharpness}. Adopt the dimensions \(d,d'\),
spaces \(\Theta,\W\), target \(\theta^\star\), projection-ball radius
\(\kappa\), numerical constants for
Assumptions~\ref{assump:learning-rate}--\ref{assump:qsm}, and joint initial
distribution \(\mu_{\mathrm{init}}\) prescribed at the start of
Section~\ref{sec:EC-sharpness-proofs}, including the condition
\(2\zeta\alpha_0>1\) when \(\delta=1\).
Also fix \(b_{\mathsf B}>0\), and set \(C_{\mathsf M}:=1\) and
\(C_{\mathsf B}:=b_{\mathsf B}\).

The class \(\widetilde{\mathfrak M}\) consists of all data-generating
processes following the extended SA recursion~\eqref{eq:extension-recursion}
with these prescribed choices, satisfying
Assumptions~\ref{assump:learning-rate}--\ref{assump:structured-errors}
and the transition condition \eqref{eq:extension-post-update-transition},
with \(\chi_1=M_1=B_0=0\). Thus,
\(\mathscr G_0^+=\sigma(\theta_0,W_0)\) is defined on the common initial
space with distribution \(\mu_{\mathrm{init}}\).
The prescribed parameters, numerical constants, and joint initial
distribution do not depend on \(\epsilon\) or \(k_0\).

\begin{theorem}
\label{thm:ec-structured-sharpness-fixed-class}
For the class \(\widetilde{\mathfrak M}\) defined above and any
\(\epsilon>0\), there are positive constants \(C_\epsilon\) and
\(K_\epsilon\) such that, for all \(k_0\geq K_\epsilon\),
\begin{equation}\label{eq:ec-structured-sharpness-worst-case-lower-bound}
\operatorname*{ess\,sup}_{\mathcal M\in\widetilde{\mathfrak M}}
\pr_{\mathcal M}\Bigl(
\exists k\geq k_0:
\|\theta_k-\theta^\star\|>c k^{-\beta}
\,\Bigm|\,\mathscr G_0^+
\Bigr)
\geq
C_\epsilon
k_0^{-\left(2(\delta-\beta-\omega_{\mathsf M})-1+\epsilon\right)}.
\end{equation}
Furthermore, there is a single process
\(\widetilde{\mathcal M}_\epsilon\in\widetilde{\mathfrak M}\) under which
the conditional exit probability satisfies the same lower bound for all
\(k_0\geq K_\epsilon\).
\end{theorem}

\proof{Proof of Theorems~\ref{thm:structured-sharpness}
and~\ref{thm:ec-structured-sharpness-fixed-class}.}
We first consider the case \(d=d'=1\), with the scalar data-generating process
\begin{equation}\label{eq:structured-sharpness-scalar-model}
\begin{aligned}
\Theta&=[\theta^\star-\kappa,\theta^\star+\kappa],
&\W&=\R,
&\nu&=\tfrac12\delta_{-1}+\tfrac12\delta_1,\\
P_\theta(w,\cdot)&=\nu,
&F(\theta,w)&=-\zeta(\theta-\theta^\star)-w.
\end{aligned}
\end{equation}
Let \((\theta_0,W_0)\) have the prescribed joint distribution
\(\mu_{\mathrm{init}}\), and let \(W_1,W_2,\ldots\) be iid with distribution
\(\nu\), independently of the initial pair. The chain reaches its stationary
distribution in one step, and
\(
\bar F(\theta)=-\zeta(\theta-\theta^\star).
\)

For each \(\epsilon>0\), let \(\widetilde\mu_\epsilon\) be the distribution of
a signed Pareto random variable \(Z\) satisfying
\begin{align}
\pr(Z>x)
=
\pr(Z<-x)
=
\frac12
\left(\frac{\widetilde x_\epsilon}{x}\right)^{
2+\epsilon/(\delta-\beta-\omega_{\mathsf M})},
\qquad x\geq \widetilde x_\epsilon
:=
\sqrt{\frac{\epsilon}{
2(\delta-\beta-\omega_{\mathsf M})+\epsilon}}.
\label{eq:structured-sharpness-pareto-law}
\end{align}
Then, \(\E[Z]=0\) and \(\E[Z^2]=1\).  Let
\(Z_1^{(\epsilon)},Z_2^{(\epsilon)},\ldots\) be iid with distribution
\(\widetilde\mu_\epsilon\), independently of the initial pair and the future
state sequence. With the prescribed initialization and the information
fields in \eqref{eq:extension-interlacing}, set
\begin{equation}\label{eq:structured-sharpness-scaled-pareto-innovation}
\chi_{k+1}:=Z_k^{(\epsilon)},\quad 
M_{k+1}:=-(k+1)^{\omega_{\mathsf M}}Z_k^{(\epsilon)},
\quad \mbox{and} \quad 
B_k:=b_{\mathsf B}(k+1)^{-\omega_{\mathsf B}},
\end{equation}
for all \(k\geq1\). 
The independence and unit-variance normalization of
\(Z_k^{(\epsilon)}\) give
\[
\E[M_{k+1}\mid\mathscr G_k^-]=0\quad\mbox{and} \quad 
\E[M_{k+1}^2\mid\mathscr G_k^-]
=(k+1)^{2\omega_{\mathsf M}}
\leq (k+1)^{2\omega_{\mathsf M}}U_2(W_k),
\]
for all \(k\geq1\). 
Assumption~\ref{assump:structured-errors}
therefore holds with
\(
\SFm_k=(k+1)^{\omega_{\mathsf M}}\),
\(\SFb_k=b_{\mathsf B}(k+1)^{-\omega_{\mathsf B}}\), \(
C_{\mathsf M}=1\),
and 
\(C_{\mathsf B}=b_{\mathsf B}\).
Denote the resulting process by \(\widetilde{\mathcal M}_\epsilon\).

Fix \(\epsilon>0\) and define
\(p:=2(\delta-\beta-\omega_{\mathsf M})+\epsilon>1\).
It suffices to identify \(C_\epsilon>0\) and an iteration threshold
\(K_\epsilon\) such that, for all \(k_0\geq K_\epsilon\),
\begin{equation}\label{eq:structured-sharpness-survival-upper-bound}
\pr\left(
|\theta_k-\theta^\star|\leq c k^{-\beta}
\text{ for all }k\geq k_0
\,\middle|\,
\mathscr G_0^+
\right)
\leq
\exp\left(-2C_\epsilon k_0^{1-p}\right),
\end{equation}
and \(2C_\epsilon k_0^{1-p}\leq1\).  Indeed, using
\(1-e^{-x}\geq x/2\) for \(x\in[0,1]\) yields exactly
\eqref{eq:structured-sharpness-martingale-lower-bound}, because
\(p-1=2(\delta-\beta-\omega_{\mathsf M})-1+\epsilon\).

For an integer \(k_0\geq1\) and \(n\geq k_0\), define
\begin{equation}\label{eq:structured-sharpness-survival-event}
\mathcal S_n:=
\{|\theta_i-\theta^\star|\leq c i^{-\beta}
\text{ for all }i=k_0,\ldots,n\}.
\end{equation}
The bound \eqref{eq:structured-sharpness-survival-upper-bound} will follow
from the one-step estimate
\begin{equation}\label{eq:structured-sharpness-one-step-survival-bound}
\pr(\mathcal S_{n+1}\mid\mathscr G_0^+)
\leq
\pr(\mathcal S_n\mid\mathscr G_0^+)
\exp\left[-2C_\epsilon(p-1)n^{-p}\right],
\end{equation}
for \(n\geq k_0\geq K_\epsilon\). Indeed, the iteration and integral
comparison following \eqref{eq:sharpness-one-step-survival-bound} in
Section~\ref{sec:EC-sharpness-proofs} apply with \(\mathscr G_0^+\) in
place of \(\F_0\) and the present value of \(p\).

We next identify \(C_\epsilon\) and \(K_\epsilon\), prove
\eqref{eq:structured-sharpness-one-step-survival-bound}, and verify
\(2C_\epsilon k_0^{1-p}\leq1\). For \(n\geq1\),
define
\begin{equation}\label{eq:structured-sharpness-proof-quantities}
\begin{gathered}
\psi_n:=c n^{-\beta},\quad
L_n:=\frac{\psi_{n+1}+|1-\zeta\alpha_n|\psi_n}{\alpha_n}, \\  
\quad C_0:=\frac{c(2+\zeta\alpha_0)}{\alpha_0},
\quad
C_1:=C_0+1+b_{\mathsf B},
\quad\mbox{and}\quad
C_\epsilon
:=\frac{(\widetilde x_\epsilon/C_1)^{
2+\epsilon/(\delta-\beta-\omega_{\mathsf M})}}{2(p-1)}.
\end{gathered}
\end{equation}
The calculation in \eqref{eq:sharpness-crossing-threshold-bound} applies to
the definitions of \(\psi_n\), \(L_n\), and \(C_0\) in
\eqref{eq:structured-sharpness-proof-quantities} and gives
\(L_n\leq C_0n^{\delta-\beta}\).  Consequently, for all \(n\geq1\),
\begin{equation}\label{eq:structured-sharpness-effective-threshold}
\frac{L_n+1+b_{\mathsf B}(n+1)^{-\omega_{\mathsf B}}}
{(n+1)^{\omega_{\mathsf M}}}
\leq
C_1 n^{\delta-\beta-\omega_{\mathsf M}}.
\end{equation}
Choose \(K_\epsilon\geq1\) such that, for all \(n\geq K_\epsilon\),
\begin{equation}\label{eq:structured-sharpness-eventual-conditions}
C_1n^{\delta-\beta-\omega_{\mathsf M}}\geq \widetilde x_\epsilon
\quad\mbox{and}\quad
2C_\epsilon n^{1-p}\leq1.
\end{equation}
Such a threshold can be chosen because
\(\delta-\beta-\omega_{\mathsf M}>0\) and \(p>1\).

Fix \(k_0\geq K_\epsilon\) and \(n\geq k_0\).  Because
\(\Theta=\theta^\star+[-\kappa,\kappa]\), the extended SA recursion can be written as
\begin{equation}\label{eq:structured-sharpness-SA-recursion}
\theta_{n+1}-\theta^\star=
\Pi_{[-\kappa,\kappa]}\!\left[
(1-\zeta\alpha_n)(\theta_n-\theta^\star)
-\alpha_nW_n
-\alpha_n(n+1)^{\omega_{\mathsf M}}Z_n^{(\epsilon)}
+\alpha_n b_{\mathsf B}(n+1)^{-\omega_{\mathsf B}}
\right].
\end{equation}
On the event \(\mathcal S_n\) defined by
\eqref{eq:structured-sharpness-survival-event}, we have
\(|\theta_n-\theta^\star|\leq\psi_n\) and \(|W_n|=1\), so
\(
(n+1)^{\omega_{\mathsf M}}|Z_n^{(\epsilon)}|
>L_n+1+b_{\mathsf B}(n+1)^{-\omega_{\mathsf B}}
\)
implies
\begin{equation}\label{eq:structured-sharpness-unprojected-bound}
\begin{aligned}
&\bigl|
(1-\zeta\alpha_n)(\theta_n-\theta^\star)
-\alpha_nW_n
-\alpha_n(n+1)^{\omega_{\mathsf M}}Z_n^{(\epsilon)}
+\alpha_n b_{\mathsf B}(n+1)^{-\omega_{\mathsf B}}
\bigr|\\
\geq{}&
\alpha_n(n+1)^{\omega_{\mathsf M}}|Z_n^{(\epsilon)}|
-|1-\zeta\alpha_n|\psi_n
-\alpha_n
-\alpha_n b_{\mathsf B}(n+1)^{-\omega_{\mathsf B}}\\
>{}&
\alpha_nL_n-|1-\zeta\alpha_n|\psi_n
=\psi_{n+1}.
\end{aligned}
\end{equation}
Thus, it follows from \eqref{eq:structured-sharpness-SA-recursion} that
\[
\begin{aligned}
|\theta_{n+1}-\theta^\star|
&=\min\!\biggl\{\kappa,
\bigl|
(1-\zeta\alpha_n)(\theta_n-\theta^\star)
-\alpha_nW_n
-\alpha_n(n+1)^{\omega_{\mathsf M}}Z_n^{(\epsilon)}
+\alpha_n b_{\mathsf B}(n+1)^{-\omega_{\mathsf B}}
\bigr|
\biggr\}
>\psi_{n+1},
\end{aligned}
\]
where the inequality follows from \(\kappa>c\geq\psi_{n+1}\) and
\eqref{eq:structured-sharpness-unprojected-bound}.  Therefore,
\eqref{eq:structured-sharpness-effective-threshold} gives
\begin{align*}
\mathcal S_n\cap
\{|Z_n^{(\epsilon)}|>C_1n^{\delta-\beta-\omega_{\mathsf M}}\}
& \subseteq
\mathcal S_n \cap\{(n+1)^{\omega_{\mathsf M}}|Z_n^{(\epsilon)}|
>L_n+1+b_{\mathsf B}(n+1)^{-\omega_{\mathsf B}} \} \\ 
& \subseteq
\{|\theta_{n+1}-\theta^\star|>\psi_{n+1}\}
\subseteq\mathcal S_{n+1}^{\complement},
\end{align*}
so
\(
\ind_{\mathcal S_{n+1}}
\leq
\ind_{\mathcal S_n}
\ind_{\{|Z_n^{(\epsilon)}|\leq
C_1n^{\delta-\beta-\omega_{\mathsf M}}\}}
\).
Here \(\mathcal S_n\in\mathscr G_n^-\) and
\(Z_n^{(\epsilon)}\) is independent of \(\mathscr G_n^-\).  Hence,
\begin{align}
\pr(\mathcal S_{n+1}\mid\mathscr G_n^-)
&\leq
\ind_{\mathcal S_n}
\pr\left(
|Z_n^{(\epsilon)}|\leq
C_1n^{\delta-\beta-\omega_{\mathsf M}}
\right).
\label{eq:structured-sharpness-one-step-survival-bound-2}
\end{align}
The first condition in \eqref{eq:structured-sharpness-eventual-conditions}
and \eqref{eq:structured-sharpness-pareto-law} imply that
\begin{equation}\label{eq:structured-sharpness-tail-bound}
\pr\left(
|Z_n^{(\epsilon)}|>C_1n^{\delta-\beta-\omega_{\mathsf M}}
\right)
=\left(\frac{\widetilde x_\epsilon}{C_1}\right)^{
2+\epsilon/(\delta-\beta-\omega_{\mathsf M})}
n^{-\left(2(\delta-\beta-\omega_{\mathsf M})+\epsilon\right)}
=2C_\epsilon(p-1)n^{-p}.
\end{equation}
Here, the last identity uses the definition of \(C_\epsilon\) in
\eqref{eq:structured-sharpness-proof-quantities}.
The tower property yields
\[
\begin{aligned}
\pr(\mathcal S_{n+1}\mid\mathscr G_0^+)
=\E\left[
\pr(\mathcal S_{n+1}\mid\mathscr G_n^-)
\,\middle|\,\mathscr G_0^+
\right]&\leq
\pr(\mathcal S_n\mid\mathscr G_0^+)
\left[1-2C_\epsilon(p-1)n^{-p}\right]\\
&\leq
\pr(\mathcal S_n\mid\mathscr G_0^+)
\exp\left[-2C_\epsilon(p-1)n^{-p}\right],
\end{aligned}
\]
which proves \eqref{eq:structured-sharpness-one-step-survival-bound}.  Here,
the first inequality follows from
\eqref{eq:structured-sharpness-one-step-survival-bound-2} and
\eqref{eq:structured-sharpness-tail-bound}, and the second holds because
\(1-x\leq e^{-x}\).
Moreover, the second
condition in \eqref{eq:structured-sharpness-eventual-conditions} gives
\(2C_\epsilon k_0^{1-p}\leq1\). 
The preceding reduction proves the lower bound for \(d=d'=1\).
The calculation uses only \(\pr(\mathcal S_{k_0}\mid\mathscr G_0^+)\leq1\),
so \(C_\epsilon\) and \(K_\epsilon\) do not depend on the realized initial pair.

For general \(d,d'\geq1\), let \(e_1\) be the first standard basis vector
in \(\R^d\), and let \(\nu\) assign probability \(1/2\) to each of
\(\pm(1,0,\ldots,0)^\intercal\in\R^{d'}\). Write \(w^{(1)}\) for the
first coordinate of \(w\), and set
\(F(\theta,w)=-\zeta(\theta-\theta^\star)-w^{(1)}e_1\) and
\(P_\theta(w,\cdot)=\nu(\cdot)\).
Use the prescribed joint initial distribution and the same independence
conditions for the future states and scalar Pareto variables. Retain
\(\chi_{k+1}=Z_k^{(\epsilon)}\), and set
\(M_{k+1}=-(k+1)^{\omega_{\mathsf M}}Z_k^{(\epsilon)}e_1\) and
\(B_k=b_{\mathsf B}(k+1)^{-\omega_{\mathsf B}}e_1\) for \(k\geq1\).
Define \(\mathcal S_n\) by \eqref{eq:structured-sharpness-survival-event}
with the Euclidean norm in place of the absolute value, and retain
\(\psi_n,L_n,C_1,C_\epsilon,K_\epsilon\) from the scalar argument.
On \(\mathcal S_n\), the same threshold condition gives
\[
\begin{aligned}
&\bigl\|(1-\zeta\alpha_n)(\theta_n-\theta^\star)
-\alpha_nW_n^{(1)}e_1
-\alpha_n(n+1)^{\omega_{\mathsf M}}Z_n^{(\epsilon)}e_1
+\alpha_n b_{\mathsf B}(n+1)^{-\omega_{\mathsf B}}e_1\bigr\|\\
&\quad\geq
\alpha_n(n+1)^{\omega_{\mathsf M}}|Z_n^{(\epsilon)}|
-|1-\zeta\alpha_n|\psi_n-\alpha_n
-\alpha_n b_{\mathsf B}(n+1)^{-\omega_{\mathsf B}}
>\psi_{n+1},
\end{aligned}
\]
where \(|W_n^{(1)}|=1\) for \(n\geq1\). Projection onto \(\Theta\)
replaces this norm by its minimum with \(\kappa\), which still exceeds
\(\psi_{n+1}\) because \(\kappa>c\geq\psi_{n+1}\). Thus,
\[
\mathcal S_n\cap
\{|Z_n^{(\epsilon)}|>C_1n^{\delta-\beta-\omega_{\mathsf M}}\}
\subseteq\mathcal S_{n+1}^{\complement}.
\]
The conditional estimate
\eqref{eq:structured-sharpness-one-step-survival-bound-2} and the subsequent
scalar argument therefore give the lower bound with the same
\(C_\epsilon\) and \(K_\epsilon\).

For each \(\epsilon>0\), the constructed process belongs to
\(\widetilde{\mathfrak M}\), and these processes share the prescribed
numerical assumption constants. Their conditional exit probabilities
satisfy \eqref{eq:structured-sharpness-martingale-lower-bound}, proving
Theorem~\ref{thm:structured-sharpness}. Each probability is bounded above
by the essential supremum in
\eqref{eq:ec-structured-sharpness-worst-case-lower-bound}, and the same
process applies for all \(k_0\geq K_\epsilon\). Both assertions of
Theorem~\ref{thm:ec-structured-sharpness-fixed-class} follow.
\Halmos\endproof

\section{Proofs for the Inventory Application}
\label{sec:EC-inventory-proofs}

We verify Assumptions~\ref{assump:learning-rate} and \ref{assump:structured-errors}
for Section~\ref{sec:inventory-application}. 
The step sizes and projection interval give
Assumptions~\ref{assump:learning-rate}--\ref{assump:compact}. 
Assumption~\ref{assump:ergodic}\textup{(ii)} on \((\theta_0,W_0)\) is imposed in 
Theorem~\ref{thm:inventory-accuracy}. 
We focus on the rest below.

\subsection{Reference Chains and Kernel Continuity}
\label{subsec:ec-inventory-foundations}

\begin{lemma}
\label{lemma:inventory-kernel}
The inventory kernels satisfy Assumptions~\ref{assump:ergodic}\textup{(i)}
and~\ref{assump:kernel-Lipschitz} on \(\Theta\). For all
\(\theta\geq0\), they have a unique stationary distribution
\(\nu_\theta\).
\end{lemma}

\proof{Proof of Lemma~\ref{lemma:inventory-kernel}.}
Use \(\W=[0,\infty)\times[0,1]\) with the Euclidean norm and
\(U_p(w)=1+\lVert w\rVert^p\).
For \(\theta\in\Theta\), let \(w=(d,l)\) be the state before the
transition. Then
\begin{gather*}
0\leq D_1\leq[a\theta_{\max}+(1-a)m+\varepsilon_1]^+,
\qquad 0\leq L_1\leq a,\\
P_\theta(w,\{(0,0)\})
\geq\Phi(-(a\theta_{\max}+(1-a)m)/\sigma)>0.
\end{gather*}
Consequently,
\[
\sup_{\theta\in\Theta,\,w\in\W}(P_\theta U_2)(w)
\leq C_0:=1+a^2+
\E\bigl[([a\theta_{\max}+(1-a)m+\varepsilon_1]^+)^2\bigr]<\infty.
\]
In Assumption~\ref{assump:ergodic}\textup{(i)}, take
\(\W_0=\W\), \(N=1\),
\(\epsilon=\Phi(-(a\theta_{\max}+(1-a)m)/\sigma)\),
\(\varphi=\delta_{(0,0)}\), \(\lambda=0\), and drift constant
\(C_0\). The state \((0,0)\) can be reached in one step from every
state and has positive one-step return probability, giving
\(\delta_{(0,0)}\)-irreducibility and aperiodicity.

To check weighted kernel continuity, change the Gaussian inputs to
\(s=u-\theta\) and \(t=\varepsilon+a\theta\). The next state has the
parameter-independent coordinates
\[
D'=[a\min\{d+s,0\}+(1-a)m+t]^+
\quad\mbox{and}\quad
L'=a\ind\{D'>0\}\bigl[\ind\{d+s>0\}+l\ind\{d+s\leq0\}\bigr].
\]
Their input density is \(\phi_\sigma(s+\theta)\phi_\sigma(t-a\theta)\),
whose derivative is
\[
\partial_\theta[\phi_\sigma(s+\theta)\phi_\sigma(t-a\theta)]
=\frac{-(s+\theta)+a(t-a\theta)}{\sigma^2}
  \phi_\sigma(s+\theta)\phi_\sigma(t-a\theta).
\]
Since \(U_1((D',L'))\leq1+a+(1-a)m+|t|\), independent
\(u,\varepsilon\sim N(0,\sigma^2)\) give the state-independent bound
\begin{align}\label{eq:ec-inventory-kernel-constant}
\sup_{\theta\in\Theta}\int_{\mathbb R^2}
&(1+a+(1-a)m+|t|)
\bigl|\partial_\theta[\phi_\sigma(s+\theta)
                     \phi_\sigma(t-a\theta)]\bigr|\,ds\,dt\notag\\
&\leq L_P:=\frac{1}{\sigma^2}
\E[(1+a+(1-a)m+a\theta_{\max}+|\varepsilon|)
          (|u|+a|\varepsilon|)]<\infty.
\end{align}
Integration of the density derivative along the interval between
\(\theta\) and \(\theta'\), followed by Tonelli's theorem, therefore
shows, for all \(|v|\leq U_1\),
\[
|(P_\theta v)(w)-(P_{\theta'}v)(w)|
\leq L_P|\theta-\theta'|
\leq L_PU_1(w)|\theta-\theta'|.
\]
Thus Assumption~\ref{assump:kernel-Lipschitz} holds.

The same transition-probability and moment bounds, with \(\theta_{\max}\)
replaced by any \(v>0\), give a unique stationary distribution
\(\nu_\theta\) for every \(\theta\in[0,v]\).
Lemma~\ref{lemma:exp-ergodic}\ref{part:exp-ergodic-invariant-moment}
gives \(\nu_\theta(U_2)<\infty\).
Since \(v\) is arbitrary, both conclusions hold for all \(\theta\geq0\).
\Halmos\endproof

\subsection{Mean Field}

\begin{lemma}
\label{lemma:inventory-field}
The limit in \eqref{eq:inventory-ideal-direction} exists and defines
\(F\) satisfying Assumption~\ref{assump:growth} on \(\Theta\).
\end{lemma}

\proof{Proof of Lemma~\ref{lemma:inventory-field}.}
Define the continuous-cost direction
\[
g_0(\theta,(d,l))
:=\bigl[h\ind\{d<\theta\}-b\ind\{d\geq\theta\}\bigr](1-l).
\]
For \(\theta\geq0\), \(w\in\W\), and
\(W'=(D',L')\sim P_\theta(w,\cdot)\), define the weighted
positive-demand density by
\(\E[(1-L')\ind\{D'\in dy\}]=p_\theta(y\mid w)\,dy\),
\(y>0\). Conditional on \(u\), when the next demand from
\(w=(d,l)\) is positive, its sensitivity is \(al\) if \(u\leq\theta-d\)
and \(a\) otherwise. Its conditional demand density is Gaussian, so
\begin{align}\label{eq:ec-inventory-weighted-density}
p_\theta(y\mid(d,l))
={}&(1-al)\int_{-\infty}^{\theta-d}
\phi_\sigma(u)\phi_\sigma(y-a(d+u)-(1-a)m)\,du\notag\\
&+(1-a)[1-\Phi((\theta-d)/\sigma)]
       \phi_\sigma(y-a\theta-(1-a)m).
\end{align}
The weighted probabilities sum to at most one. Since
\(\sup_x\phi_\sigma(x)=\phi_\sigma(0)\) and
\(\sup_x|\phi_\sigma'(x)|=\phi_\sigma(\sigma)/\sigma\),
bounding the density and its derivative under the integral gives,
uniformly for \(\theta\in\Theta\),
\(w\in\W\), and \(y>0\),
\[
p_\theta(y\mid w)\leq\phi_\sigma(0)\quad\mbox{and}\quad
|\partial_y p_\theta(y\mid w)|\leq\phi_\sigma(\sigma)/\sigma.
\]

For \(\theta\in\Theta\), \(w\in\W\), and \(\eta>0\), define
\(\bar g_\eta(\theta,w)\) as the mean of the estimator in
\eqref{eq:inventory-fd-update} under \(W'\sim P_\theta(w,\cdot)\),
with \(\theta_k,W_k,\eta_k\) replaced by \(\theta,w,\eta\).
Except when \(D'=\theta\), which has conditional probability zero,
the continuous forward quotient differs from its limiting slope only on
\((\theta,\theta+\eta]\), where the correction is
\((h+b)(1-(D'-\theta)/\eta)\). The fixed-cost quotient is
\(-q\ind\{\theta<D'\leq\theta+\eta\}/\eta\).
Multiplying by \(1-L'\) and integrating gives
\begin{align}
\bar g_\eta(\theta,w)
=(P_\theta g_0(\theta,\cdot))(w)
+\frac{h+b}{\eta}\int_0^\eta(\eta-s)p_\theta(\theta+s\mid w)\,ds
-\frac{q}{\eta}\int_0^\eta p_\theta(\theta+s\mid w)\,ds.
\label{eq:ec-inventory-estimator-mean}
\end{align}
When \(D'=0\), demand lies below both thresholds, and its contribution
is included in the first term. The density bounds give the limit in
\eqref{eq:inventory-ideal-direction} and
\begin{equation}\label{eq:ec-inventory-ideal-formula}
F(\theta,w)
=-(P_\theta g_0(\theta,\cdot))(w)
  +q p_\theta(\theta\mid w).
\end{equation}

Differentiating \eqref{eq:ec-inventory-weighted-density} at \(y=\theta\)
gives the boundary contribution
\(a(1-l)\phi_\sigma(\theta-d)\phi_\sigma((1-a)(\theta-m))\).
The remaining Gaussian-derivative terms have total coefficient mass
at most one, so
\[
\left|\frac{d}{d\theta}p_\theta(\theta\mid(d,l))\right|
\leq a\phi_\sigma(0)^2+\phi_\sigma(\sigma)/\sigma.
\]
Splitting the change in \(P_\theta g_0\) into its threshold and
kernel contributions gives
\[
|(P_\theta g_0(\theta,\cdot))(w)
 -(P_{\theta'}g_0(\theta',\cdot))(w)|
\leq [(h+b)\phi_\sigma(0)+\max\{h,b\}L_P]|\theta-\theta'|.
\]
Here the density bounds the threshold contribution, and weighted
kernel continuity bounds the second, since \(|g_0|\leq\max\{h,b\}\).
Together with \(p_\theta\leq\phi_\sigma(0)\) and \(U_1\geq1\),
these estimates verify Assumption~\ref{assump:growth} with 
\begin{align*} 
C_F:=\max\{h,b\}+q\phi_\sigma(0),\quad\mbox{and}\quad
L_F:=(h+b)\phi_\sigma(0)+\max\{h,b\}L_P
 +q\left[\phi_\sigma(\sigma)/\sigma
                       +a\phi_\sigma(0)^2\right]. \Halmos
\end{align*}
\endproof

\subsection{Stationary Objective and Projection Interval}

\begin{lemma}
\label{lemma:inventory-stationary-gradient}
The marginal functions
\(g_{\mathrm{h}},g_{\mathrm{b}},g_{\mathrm{s}}\) are twice continuously
differentiable on \((0,\infty)\), have finite right limits at zero,
and satisfy \(g_{\mathrm{b}}>0\) on \([0,\infty)\). Moreover,
\(
f'(\theta)
=h g_{\mathrm{h}}(\theta)-b g_{\mathrm{b}}(\theta)-q g_{\mathrm{s}}(\theta),
\)
for all \(\theta\geq0\),
where the derivative at zero is from the right.
On \(\Theta\), \(\bar F=-f'\).
\end{lemma}

\proof{Proof of Lemma~\ref{lemma:inventory-stationary-gradient}.}
To differentiate the stationary objective, fix \(v>0\) and use the
same independent Gaussian input pairs, indexed by all integers, for every
\(\theta\in[0,v]\). If \(\varepsilon_j\leq-[av+(1-a)m]\), then
\((D_j,L_j)=(0,0)\), regardless of the previous state, since
\(\min\{d+u,\theta\}\leq v\). The probability that no input in
\(n\) consecutive periods satisfies this inequality is
\(\Phi((av+(1-a)m)/\sigma)^n\),
which tends to zero. Thus, almost surely, each integer \(j\) has a most
recent time at or before it satisfying the inequality. Starting from
\((0,0)\) at that time and applying finitely many updates defines
\(W_j(\theta)=(D_j(\theta),L_j(\theta))\).
Shifting the input indices shifts this construction by the same amount
without changing the input distribution. Hence each process is stationary,
with marginal distribution \(\nu_\theta\) by
Lemma~\ref{lemma:inventory-kernel}.

If \(0\leq\theta\leq\theta'\leq v\) and
\(0\leq d'-d\leq a(\theta'-\theta)\), then
\[
0\leq\min\{d'+u,\theta'\}-\min\{d+u,\theta\}
\leq\theta'-\theta.
\]
Starting both chains from \((0,0)\) at the same time and applying
this inequality at each update gives
\begin{equation}\label{eq:ec-inventory-stationary-coupling}
0\leq D_j(\theta')-D_j(\theta)\leq a(\theta'-\theta)
\quad\mbox{and}\quad 
0\leq L_j(\theta)\leq a.
\end{equation}
For each deterministic starting time with state \((0,0)\) and each
fixed parameter, Gaussian conditioning gives probability zero to equality
at the stock-level cap or to a zero argument of the positive-part operation.
A countable union over starting and ending times gives the same conclusion
for the random starting time chosen above. Choose \(v>\theta\). The
finitely many updates up to \(j\) then keep the same strict comparisons
for nearby parameter values. Differentiating these updates from initial
derivative zero gives \eqref{eq:inventory-sensitivity}, and hence
derivative \(L_j(\theta)\), from the right at zero.
When \(D_j(\theta)=0\), the argument of the final positive-part operation
is strictly negative, so demand is locally zero and \(L_j(\theta)=0\).
These derivatives concern the stationary reference process; the
algorithm allows any admissible \(L_0\in[0,1]\).

For \(\theta'>\theta\geq0\), choosing \(v>\theta'\), the residual
inventory satisfies
\[
(1-a)(\theta'-\theta)
\leq[\theta'-D_0(\theta')]-[\theta-D_0(\theta)]
\leq\theta'-\theta.
\]
Thus the holding and shortage quantities have difference quotients
bounded in absolute value by one. For \(\theta>0\),
\(\pr(D_0(\theta)=\theta)=0\), so the chain rule and dominated convergence
give
\begin{equation}\label{eq:ec-inventory-marginals}
g_{\mathrm{h}}(\theta)=\E[(1-L_0(\theta))\ind\{D_0(\theta)<\theta\}]
\quad\mbox{and}\quad
g_{\mathrm{b}}(\theta)=\E[(1-L_0(\theta))\ind\{D_0(\theta)>\theta\}].
\end{equation}
Consequently,
\(h g_{\mathrm{h}}(\theta)-b g_{\mathrm{b}}(\theta)=\nu_\theta(g_0(\theta,\cdot))\),
with \(g_0\) defined in the proof of Lemma~\ref{lemma:inventory-field}.

To differentiate \(s\), define the random untruncated conditional mean
\[
A(\theta)=a\min\{D_{-1}(\theta)+u_0,\theta\}+(1-a)m.
\]
Conditioning on all inputs before \(\varepsilon_0\) yields,
for all \(\theta\geq0\),
\[
s(\theta)=\E\!\left[1-\Phi
             \left(\frac{\theta-A(\theta)}{\sigma}\right)\right].
\]
For each fixed parameter, equality at the stock-level cap has
probability zero, so
\[
A'(\theta)
=a\bigl[\ind\{D_{-1}(\theta)+u_0>\theta\}
+L_{-1}(\theta)\ind\{D_{-1}(\theta)+u_0\leq\theta\}\bigr]
\in[0,a].
\]
Moreover, \(0\leq A(\theta')-A(\theta)\leq a(\theta'-\theta)\)
for \(\theta'\geq\theta\). The difference quotients of the
conditional Gaussian tail are consequently bounded by \(\phi_\sigma(0)\).
For \(\theta>0\), dominated convergence proves
\begin{equation}\label{eq:ec-inventory-stockout-derivative}
s'(\theta)
=-\E[(1-A'(\theta))
                    \phi_\sigma(\theta-A(\theta))]
=-\nu_\theta(p_\theta(\theta\mid\cdot)).
\end{equation}
The last equality follows by conditioning on \(W_{-1}(\theta)\)
as in \eqref{eq:ec-inventory-weighted-density}; at the positive
threshold the sensitivity \(L_0(\theta)\) equals \(A'(\theta)\).
Conditioning on the final Gaussian input also gives
\begin{equation}\label{eq:ec-inventory-conditional-marginals}
\begin{aligned}
g_{\mathrm{h}}(\theta)
&=\E\!\left[(1-A'(\theta))
       \Phi\!\left(\frac{\theta-A(\theta)}{\sigma}\right)
       +A'(\theta)\Phi\!\left(-\frac{A(\theta)}{\sigma}\right)\right],\\
g_{\mathrm{b}}(\theta)
&=\E\!\left[(1-A'(\theta))
       \left[1-\Phi\!\left(\frac{\theta-A(\theta)}{\sigma}\right)\right]\right],\\
g_{\mathrm{s}}(\theta)
&=\E[(1-A'(\theta))\phi_\sigma(\theta-A(\theta))].
\end{aligned}
\end{equation}
The second term in the holding formula accounts for zero demand.
Invariance and \eqref{eq:ec-inventory-stockout-derivative} now give
\(\nu_\theta(F(\theta,\cdot))=-f'(\theta)\).

For the marginal regularity, the Gaussian translation in the proof of
Lemma~\ref{lemma:inventory-kernel} makes \(P_\theta\) twice continuously
differentiable in total-variation operator norm: the first two derivatives
of its input density have locally uniform integrable bounds independent
of the current state. The lower bound for the probability of moving
to \((0,0)\) in Lemma~\ref{lemma:inventory-kernel} gives a strict contraction
in total variation on signed measures of total mass zero.
Hence \(I-P_\theta\) has a locally bounded inverse on that space,
and stationarity gives
\[
\nu_{\theta'}-\nu_\theta
=\nu_{\theta'}(P_{\theta'}-P_\theta)(I-P_\theta)^{-1}.
\]
This identity first gives continuity of \(\nu_\theta\), including from
the right at zero, and then, by difference quotients,
\(\frac{d\nu_\theta}{d\theta}
=\nu_\theta(\partial_\theta P_\theta)(I-P_\theta)^{-1}\) for \(\theta>0\).
The identity for differences of inverses makes the inverse continuously
differentiable, so this derivative is continuously differentiable as well.
In \eqref{eq:ec-inventory-conditional-marginals}, condition on
\(W_{-1}=(d,l)\) and substitute \(u_0=s+\theta\). The conditional mean
becomes \(a\min\{d+s,0\}+a\theta+(1-a)m\), while the sensitivity factor
becomes \(a[\ind\{d+s>0\}+l\ind\{d+s\leq0\}]\), independent of
\(\theta\) for fixed \(d,l,s\). Gaussian derivative bounds justify
two differentiations under the integral, uniformly in \((d,l)\).
The resulting conditional marginal functions are twice continuously
differentiable in the supremum norm; averaging against \(\nu_\theta\)
proves the claimed marginal regularity.

At zero, the same Gaussian bounds and total-variation continuity of
\(\nu_\theta\) identify finite right limits with
\eqref{eq:ec-inventory-conditional-marginals} evaluated at zero.
Since \(1-A'(\theta)\geq1-a>0\) and Gaussian tails are strictly positive,
the shortage formula gives \(g_{\mathrm{b}}(\theta)>0\), including at zero.
The bound in \eqref{eq:ec-inventory-stationary-coupling} and the
conditional Gaussian tail formula give continuity of the cost components
at zero. The mean value theorem then identifies their
right derivatives with \(g_{\mathrm{h}}(0),-g_{\mathrm{b}}(0),
-g_{\mathrm{s}}(0)\), respectively.
\Halmos\endproof

\begin{lemma}
\label{lemma:inventory-stability}
Under the conditions \eqref{eq:inventory-stability-conditions}, the construction
\eqref{eq:ec-inventory-projection-levels} is well-defined and gives an
interval \(\Theta\subset(0,\bar\theta)\) on which
\(\bar F=-f'\) satisfies
Assumptions~\ref{assump:interior-target} and~\ref{assump:qsm}, with
\(\zeta>0\) defined in \eqref{eq:ec-inventory-qsm-coefficient}.
The target \(\theta^\star\in\operatorname{int}(\Theta)\) is the unique
minimizer of \(f\) on \(\Theta\).
\end{lemma}

\proof{Proof of Lemma~\ref{lemma:inventory-stability}.}
For arbitrary positive costs and \(\theta\geq\bar\theta\), apply
\(1-\Phi(y)\leq\tfrac12e^{-y^2/2}\) for \(y\geq0\) in
\eqref{eq:ec-inventory-conditional-marginals}. Dropping the nonnegative
second term in the holding formula and using
\(\theta-A(\theta)\geq(1-a)(\theta-m)\) and \(1-A'(\theta)\geq1-a\) give
\begin{equation}\label{eq:ec-inventory-positive-gradient-tail}
\begin{aligned}
f'(\theta)
\geq(1-a)\left\{h-
 \left[\frac{h+b}{2}+q\phi_\sigma(0)\right]
 \exp\!\left(-\frac{(1-a)^2(\theta-m)^2}{2\sigma^2}\right)\right\}
\geq\frac{h(1-a)}2>0,
\end{aligned}
\end{equation}
for all \(\theta\geq\bar\theta\).
This places every positive stationary point below \(\bar\theta\).

The marginal right limits and the first condition in
\eqref{eq:inventory-stability-conditions} give
\(\lim_{\theta\downarrow0}f'(\theta)<0\).
Together with \(f'(\bar\theta)>0\) and continuity of \(f'\), this
confines the stationary points to a compact subinterval of
\((0,\bar\theta)\). By Lemma~\ref{lemma:inventory-stationary-gradient},
\(f\) is three times continuously differentiable on \((0,\infty)\).
At a stationary point, the equalities
\(f''=f'''=0\) would make the matrix in
\eqref{eq:inventory-critical-levels} annihilate
\((h,-b,-q)^\intercal\), contradicting the second condition in
\eqref{eq:inventory-stability-conditions}.
Taylor's theorem therefore makes every stationary point isolated;
compactness makes their number finite.
By continuity, \(f\) attains a minimum on \([0,\bar\theta]\).
The endpoint signs place a minimizer \(\widehat\theta\) in the interior,
where Taylor's theorem and the preceding exclusion give
\(f''(\widehat\theta)>0\).

The factorization \(f'=g_{\mathrm{b}}(G-b)\) gives
\(G(\widehat\theta)=b\) and
\(G'(\widehat\theta)=f''(\widehat\theta)/g_{\mathrm{b}}(\widehat\theta)>0\).
The component of
\(\{\theta\in(0,\bar\theta):G'(\theta)>0\}\)
containing \(\widehat\theta\) has endpoint values on opposite sides of \(b\).
Each qualifying component contains a unique stationary point, so
finitely many qualify and the leftmost component
\((\theta_-,\theta_+)\) exists.
Strict increase gives unique solutions to
\eqref{eq:ec-inventory-projection-levels}, with
\(
\theta_-<\theta_{\min}<\theta^\star<\theta_{\max}<\theta_+\) and \(G(\theta^\star)=b\).
In particular, \(\Theta\subset(0,\bar\theta)\). The factorization makes
\(f'\) negative to the left and positive to the right
of \(\theta^\star\), proving uniqueness of the minimizer and
Assumption~\ref{assump:interior-target}.

Since \(\Theta\) lies inside this component,
\(\min_{t\in\Theta}G'(t)>0\). The mean value theorem gives
\[
(\theta-\theta^\star)f'(\theta)
\geq g_{\mathrm{b}}(\theta)\left[\min_{t\in\Theta}G'(t)\right]
       |\theta-\theta^\star|^2,
\]
for all \(\theta\in\Theta\).
Let \(u_1,\varepsilon_1\) be independent Gaussian inputs, independent
of the current state. If
\(
u_1>\theta_{\max}\) and 
\(\varepsilon_1>(1-a)(\theta_{\max}-m)\),
then the stock-level cap is active for any nonnegative current demand and
\(\theta\in\Theta\), giving
\(D_1(\theta)=a\theta+(1-a)m+\varepsilon_1>\theta\) and
\(L_1(\theta)=a\). On this event, the integrand defining
\(g_{\mathrm{b}}\) in \eqref{eq:ec-inventory-marginals} equals \(1-a\).
Its probability is the product of the two Gaussian tail probabilities.
Combining this contribution with the preceding inequality proves
Assumption~\ref{assump:qsm} with the coefficient in
\eqref{eq:ec-inventory-qsm-coefficient}.
\Halmos\endproof

\subsection{Gradient Estimator}

\begin{lemma}
\label{lemma:inventory-estimator}
Let \(\eta_k=\eta_0 k^{-\gamma}\) for all \(k\geq 1\), with
\(\eta_0,\gamma>0\). The decomposition \eqref{eq:inventory-decomposition}
satisfies Assumption~\ref{assump:structured-errors} with
\(\omega_{\mathsf M}=\gamma/2\) and \(\omega_{\mathsf B}=\gamma\).
\end{lemma}

\proof{Proof of Lemma~\ref{lemma:inventory-estimator}.}
For \(\theta\in\Theta\), \(w\in\W\), and \(0<\eta\leq \eta_0\), define
\(B(\theta,w;\eta):=-F(\theta,w)-\bar g_\eta(\theta,w)\).
Subtracting \eqref{eq:ec-inventory-estimator-mean} from the negative
mean field in \eqref{eq:ec-inventory-ideal-formula} gives
\begin{align*}
B(\theta,w;\eta)
=-\frac{h+b}{\eta}\int_0^\eta(\eta-s)p_\theta(\theta+s\mid w)\,ds
+\frac{q}{\eta}\int_0^\eta
 [p_\theta(\theta+s\mid w)-p_\theta(\theta\mid w)]\,ds.
\end{align*}
The uniform density and derivative bounds in the proof of
Lemma~\ref{lemma:inventory-field} yield
\begin{equation}\label{eq:ec-inventory-bias-bound}
|B(\theta,w;\eta)|
\leq\frac{\eta}{2}\left[(h+b)\phi_\sigma(0)
             +q\phi_\sigma(\sigma)/\sigma\right].
\end{equation}
For \(W'=(D',L')\sim P_\theta(w,\cdot)\), the inequality
\((1-L')^2\leq1-L'\) and the weighted density bound give
\[
\E\!\left[\frac{q^2}{\eta^2}(1-L')^2
          \ind\{\theta<D'\leq\theta+\eta\}\right]
\leq\frac{q^2}{\eta^2}\int_\theta^{\theta+\eta}p_\theta(y\mid w)\,dy
\leq\frac{q^2\phi_\sigma(0)}{\eta}.
\]
The continuous contribution to \(\widehat g_k\) is bounded in absolute
value by \(\max\{h,b\}\). Thus \((x+y)^2\leq2x^2+2y^2\) and centering give
\[
\E[M_{k+1}^2\mid\mathscr G_k^-]
\leq2\max\{h,b\}^2+2q^2\phi_\sigma(0)/\eta_k.
\]
The update input \(\chi_{k+1}\) is the auxiliary Gaussian pair, and
\((\theta_k,W_k)\) is \(\mathscr G_k^-\)-measurable.
Independence of the auxiliary inputs from \(\mathscr G_k^-\) gives
\[
\E[\widehat g_k\mid\mathscr G_k^-]
=\bar g_{\eta_k}(\theta_k,W_k)
\quad\mbox{and}\quad B_k=B(\theta_k,W_k;\eta_k).
\]
Hence \(B_k\) is \(\mathscr G_k^-\)-measurable, while
\(M_{k+1}\) is \(\mathscr G_k^+\)-measurable and conditionally
centered, with finite conditional second moment. These facts verify
Assumption~\ref{assump:structured-errors}\textup{(i)--(ii)} with the
corresponding deterministic bounds.

Set
\(\SFm_k=C_{\mathsf M}(k+1)^{\gamma/2}\) and
\(\SFb_k=C_{\mathsf B}(k+1)^{-\gamma}\), where
\begin{align*}
C_{\mathsf M}=\left(2\max\{h,b\}^2
                   +2q^2\phi_\sigma(0)/\eta_0\right)^{1/2}
\quad\mbox{and}\quad
C_{\mathsf B}=2^{\gamma-1}\eta_0\left[(h+b)\phi_\sigma(0)
                   +q\phi_\sigma(\sigma)/\sigma\right]. 
\end{align*}
Then, for all \(k\geq0\),
\[
\eta_k^{-1}\leq\eta_0^{-1}(k+1)^\gamma \quad\mbox{and}\quad
\eta_k\leq2^\gamma\eta_0(k+1)^{-\gamma}.
\]
Together with \((k+1)^\gamma\geq1\) and \(U_2(W_k)\geq1\),
these inequalities show that \(\SFm_k^2U_2(W_k)\) bounds the
conditional second moment and \(\SFb_k\) bounds the bias, completing
Assumption~\ref{assump:structured-errors}.
\Halmos\endproof

\subsection{Application of the Extended SA Results}

\proof{Proof of Theorem~\ref{thm:inventory-accuracy}.}
All assumptions have been verified. The operating inputs are
independent of \(\mathscr G_k^+\), so the transition under stored
\(\theta_k\) satisfies \eqref{eq:extension-post-update-transition}.
Substitute \(\omega_{\mathsf M}=\gamma/2\) and
\(\omega_{\mathsf B}=\gamma\) into Proposition~\ref{thm:structured-mse} and Theorem~\ref{thm:structured-trajectory} yields  and 
\eqref{eq:inventory-mse}, respectively.
\Halmos\endproof

\end{document}